\documentclass[12pt]{amsart}

\usepackage{amsfonts, amstext, amsmath, amsthm, amscd, amssymb, enumitem, url}
\usepackage{tikz-cd}
\usetikzlibrary{shapes.geometric, arrows}
\usepackage{xcolor}
\usepackage{svg}
\usepackage{pdfpages}
\usepackage{pdflscape}
\usepackage{spverbatim}

\usepackage{microtype}
\allowdisplaybreaks
\usepackage[parfill]{parskip}

\usepackage[colorlinks,citecolor=cyan,linkcolor=magenta]{hyperref}
\usepackage[nameinlink]{cleveref}

\usepackage{array}
\newcolumntype{C}[1]{>{\centering\let\newline\\\arraybackslash\hspace{0pt}}m{#1}}
\usepackage{multicol}
\usepackage{multirow}
\usepackage{makecell}
\usepackage{booktabs}

\newtheorem{thm}{Theorem}[section]
\newtheorem{cor}[thm]{Corollary}
\newtheorem{lemma}[thm]{Lemma}
\newtheorem{prop}[thm]{Proposition}
\newtheorem{claim}[thm]{Claim}
\newtheorem{conj}[thm]{Conjecture}

\theoremstyle{definition}
\newtheorem{defn}[thm]{Definition}
\newtheorem{rmk}[thm]{Remark}

\newtheorem{quest}[thm]{Question}

\numberwithin{equation}{section}

\renewcommand{\epsilon}{\varepsilon}

\usepackage{stmaryrd}
\newcommand{\cut}{\!\bbslash\!}

\newcommand{\Hu}{\Huge}

\newcommand{\La}{\Large}

\newcommand{\ns}{\normalsize}

\newcommand{\fns}{\footnotesize}

\newcommand{\s}{\text{\reflectbox{$s$}}}
\renewcommand{\t}{\text{\reflectbox{$t$}}}
\newcommand{\f}{\text{\reflectbox{$f$}}}

\newcommand{\salt}{\text{\reflectbox{$s^2_1$}}}

\begin{document}

\title[Normalized dilatations of fully-punctured pseudo-Anosov maps]{On the set of normalized dilatations of fully-punctured pseudo-Anosov maps}

\author{Chi Cheuk Tsang}
\address{University of California, Berkeley \\
    970 Evans Hall \#3840 \\
    Berkeley, CA 94720-3840}
\email{chicheuk@math.berkeley.edu}
\thanks{Chi Cheuk Tsang was partially supported by a grant from the Simons Foundation \#376200.}

\begin{abstract}
We improve the bound on the number of tetrahedra in the veering triangulation of a fully-punctured pseudo-Anosov mapping torus in terms of the normalized dilatation. When the mapping torus has only one boundary component, we can improve the bound further. Together with the author's work with Hironaka in the case when the mapping torus has at least two boundary components, this allows us to understand small elements of the set $\mathcal{D}$ of normalized dilatations of fully-punctured pseudo-Anosov maps using computational means. In particular, we certify that the minimum element of $\mathcal{D}$ is $\mu^2$ and the minimum accumulation point of $\mathcal{D}$ is $\mu^4$, where $\mu$ is the golden ratio.
\end{abstract}

\maketitle

\section{Introduction} \label{sec:intro}

An orientation-preserving surface homeomorphism $f:S \to S$ is \textit{pseudo-Anosov} if there exists a transverse pair of singular measured foliations $\ell^s$ and $\ell^u$ such that $f$ contracts the leaves of $\ell^s$ and expands the leaves of $\ell^u$ by a factor of $\lambda(f) > 1$. The number $\lambda(f)$ is called the \textit{dilatation} of $f$. 

In this case, $\ell^s$ and $\ell^u$ determine a conformal structure on $S$. Contracting and expanding the leaves of the two foliations deforms the conformal structure and determines a geodesic path on the \textit{Teichm\"uller space} of $S$ (with the \textit{Teichm\"uller metric}). In particular, $f$ determines a closed (possibly non-primitive) geodesic of length $\log \lambda(f)$ on the \textit{moduli space} $\mathcal{M}(S)$. Conversely, every closed geodesic on $\mathcal{M}(S)$ comes from a pseudo-Anosov map on $S$. See \cite{Abi80} for details.
This gives a natural motivation for

\begin{quest}[Minimum dilatation problem] \label{quest:mindil}
What is the minimum dilatation $\delta_{g,s}$ among all pseudo-Anosov maps defined on a given surface $S_{g,s}$ with genus $g$ and $s$ punctures?
\end{quest}

\Cref{quest:mindil} has been studied since at least \cite{Pen91}, but so far it has only been solved for a handful of surfaces with small values of $g$ and $s$. We refer to \cite{LT11a} and \cite{LT11b} for details and references.
We remark that between these known values and some upper bounds (see, for example, \cite{Hir10}, \cite{AD10}, \cite{KT13}), the pattern of these minimum dilatations seem erratic, and it is not even clear what a good set of conjectural values should be.

The situation becomes simpler if instead of asking for the minimum dilatation on specific surfaces, one considers the asymptotics of these minimum dilatations. In particular we have the following well-known conjecture by Hironaka. See also \cite[P.44 Question]{McM00}.

\begin{conj}[{Golden ratio conjecture, Hironaka \cite[Question 1.12]{Hir10}}] \label{conj:goldenratio}
The minimum dilatations $\delta_{g,0}$ on the closed orientable surfaces of genus $g$ grow as
$$\lim_{g \to \infty} \delta_{g,0}^g = \mu^2 \approx 2.618.$$
where $\mu = \frac{1 + \sqrt{5}}{2}$ is the golden ratio.
\end{conj}

This simplification is due to the fact that pseudo-Anosov maps naturally come in \textit{flow equivalence classes}. 
Given a pseudo-Anosov map $f:S \to S$, its \textit{mapping torus} is a $3$-manifold with a fibration over $S^1$ and a transverse \textit{suspension flow}. Two pseudo-Anosov maps $f_1:S_1 \to S_1$ and $f_2:S_2 \to S_2$ are \textit{flow equivalent} if their mapping torus is the same $3$-manifold $M$ and their suspension flows are the same.
Thurston-Fried fibered face theory states that maps of a single flow equivalence class correspond to interior rational points in a polyhedron $F$, and the \textit{normalized dilatation} $\lambda(f)^{|\chi(S)|}$ extends to a continuous convex function on $F$ that goes to infinity at $\partial F$. 
See \cite[Exposé 14]{FLP79}.

To simplify the problem further, one can restrict to \textit{fully-punctured} pseudo-Anosov maps. These are maps where all the singularities of the foliations are at the punctures of the surface. 
Given any pseudo-Anosov map, one can puncture at all the singularities to get a fully-punctured map with the same dilatation. On the level of mapping tori, this corresponds to drilling out the singular orbits of the suspension flow. Consequently, every flow equivalence class is contained in a fully-punctured one. 
If one understands the fully-punctured flow equivalence classes which give small dilatations, then one can hope to recover information about general flow equivalence classes by Dehn filling the corresponding $3$-manifolds and performing an analysis as in \cite{KKT13}.
This motivates

\begin{quest}[Fully-punctured normalized dilatation problem] \label{quest:fullypuncnormdil}
Let $\mathcal{D}$ be the set of normalized dilatations of fully-punctured pseudo-Anosov maps. What are the smallest elements of $\mathcal{D}$ and what are the maps that attain them?
\end{quest}

In this paper, we make some progress on \Cref{quest:fullypuncnormdil}.
Our main theorem is

\begin{thm} \label{thm:introdilthm}
The set $\mathcal{D}$ of normalized dilatations of fully-punctured pseudo-Anosov maps is the union of the isolated points
$$
\begin{array}{ccc}
    \frac{3+\sqrt{5}}{2} \approx 2.618, & \frac{4+\sqrt{12}}{2} \approx 3.732, & (\text{Lehmer's number})^9 \approx 4.311, \\
    \frac{5+\sqrt{21}}{2} \approx 4.791, & \left| LT_{1,2} \right|^3 \approx 5.107, & \frac{6+\sqrt{32}}{2} \approx 5.828,
\end{array}
$$
and a dense subset of $[\mu^4, \infty)$.
In particular the minimum element of $\mathcal{D}$ is $\mu^2 = \frac{3+\sqrt{5}}{2}$ and the minimum accumulation point of $\mathcal{D}$ is $\mu^4 = \frac{7+\sqrt{45}}{2} \approx 6.854$.

\begin{figure}[ht]
    \centering
    \fontsize{8pt}{8pt}\normalfont
    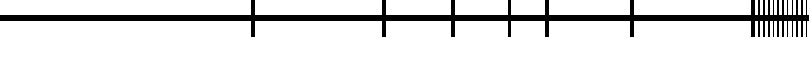
\end{figure}

Moreover, the fully-punctured pseudo-Anosov maps whose normalized dilatations attain the isolated points are those listed in \Cref{tab:introisolateddil}. The fully-punctured pseudo-Anosov maps whose normalized dilatations attain the minimum accumulation point are those listed in \Cref{tab:introminaccumdil}. 
\end{thm}

In \Cref{thm:introdilthm}, $|LT_{1,2}|$ is the largest real root of $t^4-t^3-t^2-t+1$ and \textit{Lehmer's number} is the largest real root of $t^{10}+t^9-t^7-t^6-t^5-t^4-t^3+t+1$. 

\begin{table}[ht]
    \centering
    \footnotesize
    \caption{The isolated points of $\mathcal{D}$, the fully-punctured pseudo-Anosov maps $f$ that attain them, and the corresponding layered veering triangulations.}
    \renewcommand{\arraystretch}{1.5}
    \begin{tabular}{|*3{>{\renewcommand{\arraystretch}{1}}c|}}
    \hline
    Normalized dilatation & Pseudo-Anosov maps & Veering triangulations \\
    \hline
    \hline
    \multirow{2}{*}{$\frac{3+\sqrt{5}}{2}$} & Map on $S_{1,1}$ induced by $\begin{bmatrix} 2 & 1 \\ 1 & 1 \end{bmatrix} = RL$ & \texttt{cPcbbbiht\_12} \\
    \cline{2-3}
    & Map on $S_{1,1}$ induced by $\begin{bmatrix} -2 & -1 \\ -1 & -1 \end{bmatrix} = -RL$ & \texttt{cPcbbbdxm\_10} \\
    \hline
    \multirow{2}{*}{$\frac{4+\sqrt{12}}{2}$} & Map on $S_{1,1}$ induced by $\begin{bmatrix} 3 & 2 \\ 1 & 1 \end{bmatrix} = R^2L$ & \texttt{dLQbccchhfo\_122} \\
    \cline{2-3}
    & Map on $S_{1,1}$ induced by $\begin{bmatrix} -3 & -2 \\ -1 & -1 \end{bmatrix} = -R^2L$ & \texttt{dLQbccchhsj\_122} \\
    \hline
    $(\text{Lehmer's number})^9$ & \makecell{Map on $S_{5,1}$ induced by quotient of \\ geodesic flow on $P(-2,3,7)$} & \texttt{dLQacccjsnk\_200}  \\
    \hline
    \multirow{2}{*}{$\frac{5+\sqrt{21}}{2}$} & Map on $S_{1,1}$ induced by $\begin{bmatrix} 4 & 3 \\ 1 & 1 \end{bmatrix} = R^3L$ & \texttt{eLMkbcdddhhhml\_1221} \\
    \cline{2-3}
    & Map on $S_{1,1}$ induced by $\begin{bmatrix} -4 & -3 \\ -1 & -1 \end{bmatrix} = -R^3L$ & \texttt{eLMkbcdddhhhdu\_1221} \\
    \hline
    \multirow{2}{*}{$\left| LT_{1,2} \right|^3$} & \makecell{Map on $S_{2,1}$ lifted from \\ minimum dilatation $5$-braid} & \texttt{eLPkaccddjnkaj\_2002} \\
    \cline{2-3}
    & \makecell{Map on $S_{2,1}$ lifted from \\ minimum dilatation $5$-braid} & \texttt{eLPkbcdddhrrcv\_1200} \\
    \hline
    \multirow{4}{*}{$\frac{6+\sqrt{32}}{2}$} & Map on $S_{1,1}$ induced by $\begin{bmatrix} 5 & 2 \\ 2 & 1 \end{bmatrix} = R^2L^2$ & \texttt{eLMkbcdddhhqqa\_1220} \\
    \cline{2-3}
    & Map on $S_{1,1}$ induced by $\begin{bmatrix} -5 & -2 \\ -2 & -1 \end{bmatrix} = -R^2L^2$ & \texttt{eLMkbcdddhhqxh\_1220} \\
    \cline{2-3}
    & Map on $S_{1,1}$ induced by $\begin{bmatrix} 5 & 4 \\ 1 & 1 \end{bmatrix} = R^4L$ & \texttt{fLMPcbcdeeehhhhkn\_12211} \\
    \cline{2-3}
    & Map on $S_{1,1}$ induced by $\begin{bmatrix} -5 & -4 \\ -1 & -1 \end{bmatrix} = -R^4L$ & \texttt{fLMPcbcdeeehhhhvc\_12211} \\
    \hline
    \end{tabular}
    \label{tab:introisolateddil}
\end{table}

\begin{table}[ht]
    \centering
    \footnotesize
    \caption{The fully-punctured pseudo-Anosov maps with normalized dilatation $\mu^4$, and the veering triangulations and Betti number of the corresponding mapping tori.}
    \renewcommand{\arraystretch}{1.5}
    \begin{tabular}{|*3{>{\renewcommand{\arraystretch}{1}}c|}}
    \hline
    Pseudo-Anosov maps & Veering triangulations & Betti number \\
    \hline
    \hline
    Map on $S_{1,1}$ induced by $\begin{bmatrix} 5 & 3 \\ 3 & 2 \end{bmatrix} = RLRL$ & \texttt{eLMkbcdddhxqdu\_1200} & 1 \\
    \hline
    Map on $S_{1,1}$ induced by $\begin{bmatrix} -5 & -3 \\ -3 & -2 \end{bmatrix} = -RLRL$ & \texttt{eLMkbcdddhxqlm\_1200} & 1 \\
    \hline
    \makecell{Map on $S_{1,2}$ induced by $\begin{bmatrix} 2 & 1 \\ 1 & 1 \end{bmatrix}$ \\ punctured at a period $2$ point} & \texttt{fLLQcbeddeehhnkhh\_21112} & 1 \\
    \hline
    Map on $S_{1,1}$ induced by $\begin{bmatrix} 6 & 5 \\ 1 & 1 \end{bmatrix} = R^5L$ & \texttt{gLMzQbcdefffhhhhhit\_122112} & 1 \\
    \hline
    Map on $S_{1,1}$ induced by $\begin{bmatrix} -6 & -5 \\ -1 & -1 \end{bmatrix} = -R^5L$ & \texttt{gLMzQbcdefffhhhhhpe\_122112} & 1 \\
    \hline
    Map on $S_{0,4}$ induced by $\begin{bmatrix} 2 & 1 \\ 1 & 1 \end{bmatrix} = RL$ & \texttt{eLMkbcdddhxqlm\_1200} & 2 \\
    \hline
    \makecell{Map on $S_{1,2}$ induced by $\begin{bmatrix} -2 & -1 \\ -1 & -1 \end{bmatrix}$ \\ punctured at two fixed points} & \texttt{fLLQcbeddeehhbghh\_01110} & 2 \\
    \hline
    \end{tabular}
    \label{tab:introminaccumdil}
\end{table}

We elaborate on the descriptions of the pseudo-Anosov maps in \Cref{tab:introisolateddil} and \Cref{tab:introminaccumdil}.
As before, we denote by $S_{g,s}$ the orientable surface with genus $g$ and $s$ punctures. 

We identify the once-punctured torus $S_{1,1}$ with $(\mathbb{R}^2 \backslash \mathbb{Z}^2)/\mathbb{Z}^2$ and the 4-punctured sphere $S_{0,4}$ with $(\mathbb{R}^2 \backslash \mathbb{Z}^2)/(\pm(2\mathbb{Z})^2)$.
Under this identification, every matrix in $\mathrm{SL}_2 \mathbb{Z}$ induces a map on $S_{1,0}$ and a map on $S_{0,4}$.
The dilatation of these maps is the largest eigenvalue of the matrix.
We also included a factorization of each element of $\mathrm{SL}_2 \mathbb{Z}$ in \Cref{tab:introisolateddil} and \Cref{tab:introminaccumdil} into a word in $R=\begin{bmatrix} 1 & 1 \\ 0 & 1 \end{bmatrix}$ and $L=\begin{bmatrix} 1 & 0 \\ 1 & 1 \end{bmatrix}$, which the experts may find convenient.

We now move on to the more sporadic examples. For the map defined on $S_{5,1}$ in \Cref{tab:introisolateddil}, recall that the double branched cover of $S^3$ over the pretzel knot $P(-2,3,7)$ is the unit tangent bundle over the orbifold $S^2(2,3,7)$. The deck transformation is the map induced by reflection of $S^2(2,3,7)$ across a curve $c$ dividing the orbifold into two triangles, see \cite{BS09}. The union of fibers lying over $c$, with the full lift of $c$ removed, is a Birkhoff section to the geodesic flow on $T^1 S^2(2,3,7)$. The monodromy on the quotient of this section is the described map.

For the two maps defined on $S_{2,1}$ in \Cref{tab:introisolateddil}, recall from \cite{HS07} that $\sigma_1 \sigma_2 \sigma_3 \sigma_4 \sigma_1 \sigma_2$ is the fully-punctured $5$-braid with minimum dilatation $|LT_{1,2}|$. The braid points are $1$-pronged singularities while the point at infinity is $3$-pronged. Consider the double cover $S_{2,6} \to S_{0,6}$ with degree two over each of the punctures. Lift the braid monodromy to a map on $S_{2,6}$. There are two choices here, which differ by the deck transformation of the double cover. For either choice, $5$ of the $6$ punctures of $S_{2,6}$ are $2$-pronged hence can be filled in. Together these give the described maps.

For the first map defined on $S_{1,2}$ in \Cref{tab:introminaccumdil}, we puncture the map induced by $\begin{bmatrix} 2 & 1 \\ 1 & 1 \end{bmatrix}$ on $S_{1,0}$ at a pair of points of period 2. There are two choices here for which pair of points to puncture but they give conjugate maps.
For the other map defined on $S_{1,2}$ in \Cref{tab:introminaccumdil}, we puncture the map induced by $\begin{bmatrix} -2 & -1 \\ -1 & -1 \end{bmatrix}$ on $S_{1,0}$ at two fixed points. Again, there are a number of choices here, but the resulting maps are all conjugate.

The columns labeled `veering triangulations' in \Cref{tab:introisolateddil} and \Cref{tab:introminaccumdil} indicate the isoSig code of the veering triangulation associated to each pseudo-Anosov map. See \Cref{subsec:layeredvt}, \Cref{subsec:isolatedpoints}, and \Cref{subsec:accumpoint}.

\subsection{Previous work}

It has been known since work of Thurston \cite{FLP79} that the dynamics of a fully-punctured pseudo-Anosov map $f$ can be encoded using the combinatorial tool of \textit{train tracks}.
More specifically, one can approximate the stable foliation $\ell^s$ using a train track $\tau$. The fact that $f$ contracts along $\ell^s$ and expands along $\ell^u$ translates to the fact that $f(\tau)$ can be obtained by folding $\tau$. One can then compute the dilatation of $f$ from the \textit{transition matrix} which record how the branches of $\tau$ fold over those of $f(\tau)$.

In \cite{Ago11}, Agol showed that one can choose $\tau$ such that there is a canonical periodic splitting sequence from $f(\tau)$ to $\tau$. The dual triangulations of this splitting sequence determines an ideal triangulation of the mapping torus, which we refer to as the \textit{veering triangulation} associated to $f$.
We will recall the definition of veering triangulations in \Cref{subsec:vt}. For the moment, it suffices to know that these are ideal triangulations satisfying certain combinatorial conditions which impose strong constraints on the local structure.
This makes it possible to enumerate veering triangulations up to 16 tetrahedra \cite{GSS}.

In \cite{AT24}, by studying the Perron-Frobenius components of the transition matrix associated to the splitting sequence, Agol and the author proved the following theorem.

\begin{thm}[{\cite{AT24}}] \label{thm:introAT}
Let $f:S \to S$ be a fully-punctured pseudo-Anosov map with normalized dilatation $\lambda^{-\chi} \leq P$, then the mapping torus of $f$ admits a veering triangulation with $\leq \frac{P^3-1}{2} (\frac{2 \log P^3}{\log(2 P^{-3}+1)}-1)$ tetrahedra.
\end{thm}

We remark that a non-quantitative version of this result was proved by Farb, Leininger, and Margalit \cite{FLM11} earlier. 

Notice that \Cref{quest:fullypuncnormdil} is theoretically solved by \Cref{thm:introAT}: 
Suppose one is interested in the elements of $\mathcal{D}$ smaller than some number $P$. Then one can look at all veering triangulations with $\leq \frac{P^3-1}{2} (\frac{2 \log P^3}{\log(2 P^{-3}+1)}-1)$ tetrahedra, and for each of them compute the normalized dilatations of the associated maps, using the \textit{Teichmüller polynomial} defined in \cite{McM00}, then read off those maps whose normalized dilatations are less than $P$.

However, this strategy is not actually feasible in practice. 
For example, with \Cref{conj:goldenratio} in mind, if one puts in $P=\mu^4$, then one has to look at all veering triangulations with $\leq 299193$ tetrahedra, which is much larger than any census of veering triangulations we can possibly generate currently.

A different idea was explored in \cite{HT22} by Hironaka and the author. We showed that if $f:S \to S$ is a fully-punctured pseudo-Anosov map with at least two puncture orbits, then one can take $\tau$ to be a \textit{standardly embedded train track}. Then by applying the theory of Perron-Frobenius digraphs, developed by McMullen in \cite{McM15}, on the real edges of $\tau$, we showed \Cref{thm:introHT} below.

\begin{thm}[{\cite{HT22}}] \label{thm:introHT}
Let $f:S \to S$ be a fully-punctured pseudo-Anosov map with at least two puncture orbits. Then the normalized dilatation $\lambda^{-\chi}$ of $f$ satisfies the inequality
$$\lambda^{-\chi} \geq 
\begin{cases}
\mu^4 \approx 6.854 & \text{if $|\chi(S)|=2$} \\
\left| LT_{1,\frac{|\chi(S)|}{2}} \right|^{|\chi(S)|} > \mu^4 & \text{if $|\chi(S)|$ is even and $\geq 4$} \\
8 > \mu^4 & \text{if $|\chi(S)|$ is odd} \\
\end{cases}$$

Moreover, for each $k \geq 1$, equality for the first two cases is achieved by some fully-punctured pseudo-Anosov maps.
\end{thm}

We remind the reader that the condition of $f$ having at least two puncture orbits is equivalent to its mapping torus having at least two boundary components.
Since a mapping torus with at least two boundary components must have Betti number at least two, the examples in \Cref{thm:introHT} imply that the minimum accumulation point of $\mathcal{D}$ is at most $\mu^4$.
One can also deduce from \Cref{thm:introHT} that, in order to show \Cref{thm:introdilthm}, it remains to study pseudo-Anosov maps whose mapping tori has only one boundary component. (See \Cref{subsec:isolatedpoints} and \Cref{subsec:accumpoint} for details.)

\subsection{Improved bounds on veering tetrahedra}

The main contribution of this paper towards \Cref{thm:introdilthm} is an improvement of \Cref{thm:introAT}. We are able to improve the general bound from sextic ($+\epsilon$) to quadratic.

\begin{thm} \label{thm:introsinglehook}
Let $f:S \to S$ be a fully-punctured pseudo-Anosov map with normalized dilatation $\lambda^{-\chi} \leq P$. Then the mapping torus of $f$ admits a veering triangulation with less than or equal to $\frac{1}{2}P^2$ tetrahedra.
\end{thm}

Now repeating the strategy with \Cref{thm:introsinglehook} replacing \Cref{thm:introAT}, when one puts in $P=\mu^4$, one only has to look at all veering triangulations with $\leq 23$ tetrahedra. Unfortunately, this is still slightly out of reach of the current veering triangulation census.
But recall that \Cref{thm:introHT} already takes care of the case with at least two boundary components. In the case with one boundary component (and when $P=\mu^4$), we are able to further improve the bound, to the point where we can apply the current census.

\begin{thm} \label{thm:intro16tet}
Let $f:S \to S$ be a fully-punctured pseudo-Anosov map with normalized dilatation $\lambda^{-\chi} \leq 6.86$. Suppose the mapping torus of $f$ has only one boundary component, then the mapping torus of $f$ admits a veering triangulation with less than or equal to $16$ tetrahedra.
\end{thm}

With \Cref{thm:intro16tet} in place, we can finally complete the proof of \Cref{thm:introdilthm} by going through the census \cite{GSS}.
Now, this in itself is still a nontrivial task since there are $51766$ layered veering triangulations in the census. To carry out the computation, we wrote up SageMath scripts which integrate the Veering code of Parlak, Schleimer, and Segerman (\cite{Veering}), to compute normalized dilatations. 
We include these scripts in the auxiliary files and provide a rundown of the code in \Cref{sec:code}.

\subsection{Ideas in the proof of \Cref{thm:introsinglehook} and \Cref{thm:intro16tet}}

For the rest of the introduction, we outline some technical ideas in the proof of \Cref{thm:introsinglehook} and \Cref{thm:intro16tet}. We assume that the reader is familiar with basic definitions for veering triangulations.

Let $\Delta$ be the veering triangulation associated to a fully punctured pseudo-Anosov map $f:S \to S$.
Let $B$ be the \textit{stable branched surface} of $\Delta$. One can lift $B$ to a branched surface $\widehat{B}$ in the infinite cyclic cover corresponding to the fibration associated to $f$. Note that $\widehat{B}$ is a measured branched surface. 

Recall that the \textit{dual graph} $\Gamma$ is a $(2,2)$-valent directed graph.
In particular, there are many Eulerian circuits of $\Gamma$. Each Eulerian circuit $c$ can be lifted and perturbed to a descending path $\alpha$ on $\widehat{B}$. 
The fact that $c$ is Eulerian implies that the `height' between the starting and ending sectors of $\alpha$, i.e. the deck transformation that takes the latter to the former, is determined by the Euler characteristic $\chi(S)$, thus the difference in weights between these sectors can be expressed in terms of the normalized dilatation.
Meanwhile, the number of times $\alpha$ intersects $\widehat{\Gamma}$ is exactly two times the number of tetrahedra.

We define a special type of circuit, called a \textit{hook circuit}, that allows one to bound the former in terms of the latter from below.
Essentially, the difference in weights between the starting and ending sector is the sum of weights of sectors that merge in at each intersection point between $\alpha$ and $\widehat{\Gamma}$.
The sectors that merge into $\alpha$ far away from the starting sector have low height thus large weight. The hook condition ensures that we can do a pairing trick on the remaining initial merging sectors.
When worked out precisely, this gives the bound in \Cref{thm:introsinglehook}. 
Finally, we demonstrate that in all but one layered veering triangulations, one can pick an Eulerian hook circuit. 

This last part makes heavy use of the knowledge of the local combinatorics of veering triangulations. 
We interpret the existence of Eulerian circuits of $\Gamma$ with certain prescribed subpaths as the connectedness of $\Gamma$ after certain resolutions at vertices.
We introduce the notion of \textit{A-} and \textit{B-quads}, which are dual to resolving $\Gamma$ at a vertex in an anti-branching or branching way respectively. In turn, resolving $\Gamma$ at a couple of vertices is dual to building a 2-complex $Q$ out of these quads. 
Using Alexander duality, the connectedness of $\Gamma$ after resolution can be interpreted in terms of the second homology map of the inclusion of $Q$ into the 3-manifold $M$.
The upshot is that the obstruction for Eulerian hook circuits is a specific set of edge identifications for $\Delta$. If these edge identifications hold everywhere, then $\Delta$ can only be the triangulation \texttt{cPcbbbdxm\_10}.
Upon analyzing this exceptional triangulation by hand, this shows \Cref{thm:introsinglehook}.
We refer to \Cref{sec:singlehook} for details.

Instead of a single Eulerian hook circuit, it is sometimes possible to pick double Eulerian hook circuits, i.e. two hook circuits that together pass through each edge exactly once. 
In this case, one can improve \Cref{thm:introsinglehook} by a factor of $2$ (\Cref{prop:doublehookbound}).
As above, one can show that the obstruction to double hook circuits is some set of edge identifications. 
See \Cref{sec:doublehook} for details.
We conjecture that double hook circuits always exist (\Cref{conj:doublehookexist}).

To show \Cref{thm:intro16tet}, we combine cases where we can show double hook circuits exist, along with cases where we can improve the estimates in the single hook circuit argument. 
The technical complication is that for the latter, the sources of improvement don't tend to exist at the same time. For example, certain arguments only work when $B$ has many branch cycles, and others only work when $B$ has few branch cycles. 
What we managed is an elaborate patchwork of ideas that cover all the cases and lowers the bound on the number of tetrahedra from $23$ to $16$. 
See \Cref{sec:oneboundary} for details.

{\bf Outline of paper.} In \Cref{sec:background} we provide some background for (layered) veering triangulations and set up some terminology. In \Cref{sec:singlehook} we explain the proof of \Cref{thm:introsinglehook}.

In \Cref{sec:doublehook} and \Cref{sec:oneboundary} we explain many ways to sharpen the argument in \Cref{sec:singlehook}, which will result in an improved bound in the one boundary case, as recorded in \Cref{thm:EIIRP}. There are some calculus computations that needs to be performed in the course of proving \Cref{thm:EIIRP} and deducing \Cref{thm:intro16tet} from \Cref{thm:EIIRP}. We defer these to \Cref{sec:calculus} so that the reader can focus on the main ideas.

In \Cref{sec:dilatation} we explain how to use results from computations to arrive at \Cref{thm:introdilthm}. Explanations for the code used for the computation are deferred to \Cref{sec:code}. Finally, in \Cref{sec:questions}, we discuss some future directions.

{\bf Acknowledgements.} A very special thanks to my PhD advisor Ian Agol for his guidance, support, and patience with me during my time as his student. 
We thank Eriko Hironaka for her collaboration on \cite{HT22}, and for many helpful conversations about pseudo-Anosov maps and dynamics in general.
We thank Anna Parlak, Saul Schleimer, and Henry Segerman for writing and maintaining the Veering code, as well as their advice on coding.
We thank Chris Leininger for comments on an earlier version of this paper.
Finally, we thank the anonymous referee for their suggestions on improving the organization of the paper.

{\bf Notational conventions.} Throughout this paper, 
\begin{itemize}
    \item $X \cut Y$ will denote the metric completion of $X \backslash Y$ with respect to the induced path metric from $X$. In addition, we will call the components of $X \cut Y$ the \textit{complementary regions} of $Y$ in $X$.
    \item If $G$ is a directed graph, we will denote an edge path as the sequence of edges $(e_1,...,e_n)$ that it traverses.
    \item Suppose $\alpha$ is a path, then $-\alpha$ will denote the path traversed in the opposite direction.
\end{itemize}

\section{Background} \label{sec:background}

\subsection{Veering triangulations} \label{subsec:vt}

We recall some basic definitions and facts about veering triangulations. For more details, see \cite[Chapter 1]{Tsathesis}. 

\begin{defn} \label{defn:vt}
A \textit{veering triangulation} is a transverse taut ideal triangulation with a coloring of the edges by red or blue, so that the four side edges of each tetrahedron starting from an endpoint of the front edge and going counterclockwise, are colored red, blue, red, blue, respectively. See \Cref{fig:veertet}.
\end{defn}

\begin{figure}[ht]
    \centering
    \fontsize{14pt}{14pt}\selectfont
    \resizebox{!}{3cm}{
\begingroup%
  \makeatletter%
  \providecommand\color[2][]{%
    \errmessage{(Inkscape) Color is used for the text in Inkscape, but the package 'color.sty' is not loaded}%
    \renewcommand\color[2][]{}%
  }%
  \providecommand\transparent[1]{%
    \errmessage{(Inkscape) Transparency is used (non-zero) for the text in Inkscape, but the package 'transparent.sty' is not loaded}%
    \renewcommand\transparent[1]{}%
  }%
  \providecommand\rotatebox[2]{#2}%
  \newcommand*\fsize{\dimexpr\f@size pt\relax}%
  \newcommand*\lineheight[1]{\fontsize{\fsize}{#1\fsize}\selectfont}%
  \ifx\svgwidth\undefined%
    \setlength{\unitlength}{325.82965819bp}%
    \ifx\svgscale\undefined%
      \relax%
    \else%
      \setlength{\unitlength}{\unitlength * \real{\svgscale}}%
    \fi%
  \else%
    \setlength{\unitlength}{\svgwidth}%
  \fi%
  \global\let\svgwidth\undefined%
  \global\let\svgscale\undefined%
  \makeatother%
  \begin{picture}(1,0.40856788)%
    \lineheight{1}%
    \setlength\tabcolsep{0pt}%
    \put(0,0){\includegraphics[width=\unitlength,page=1]{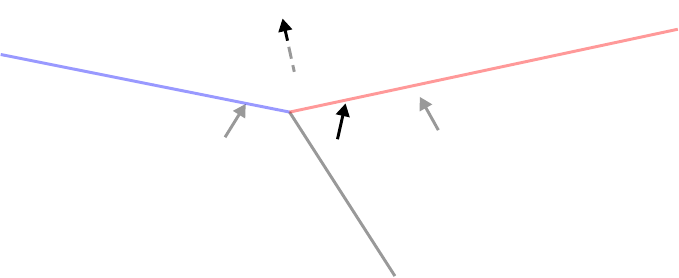}}%
    \put(0.45702311,0.36605622){\color[rgb]{0,0,0}\makebox(0,0)[lt]{\lineheight{1.25}\smash{\begin{tabular}[t]{l}$\pi$\end{tabular}}}}%
    \put(0.49904816,0.14324942){\color[rgb]{0,0,0}\transparent{0.40000001}\makebox(0,0)[lt]{\lineheight{1.25}\smash{\begin{tabular}[t]{l}$\pi$\end{tabular}}}}%
    \put(0.76168485,0.11769163){\color[rgb]{0,0,1}\makebox(0,0)[lt]{\lineheight{1.25}\smash{\begin{tabular}[t]{l}$0$\end{tabular}}}}%
    \put(0.31075696,0.09433859){\color[rgb]{1,0,0}\makebox(0,0)[lt]{\lineheight{1.25}\smash{\begin{tabular}[t]{l}$0$\end{tabular}}}}%
    \put(0.25420744,0.29133508){\color[rgb]{0,0,1}\transparent{0.40000001}\makebox(0,0)[lt]{\lineheight{1.25}\smash{\begin{tabular}[t]{l}$0$\end{tabular}}}}%
    \put(0.62556736,0.30455855){\color[rgb]{1,0,0}\transparent{0.40000001}\makebox(0,0)[lt]{\lineheight{1.25}\smash{\begin{tabular}[t]{l}$0$\end{tabular}}}}%
    \put(0,0){\includegraphics[width=\unitlength,page=2]{veertet.pdf}}%
  \end{picture}%
\endgroup%
}
    \caption{A tetrahedron in a veering triangulation. There are no restrictions on the colors of the top and bottom edges.} 
    \label{fig:veertet}
\end{figure}

The local structure of a veering triangulation is fairly restricted. To be more precise, we have \Cref{prop:togglefan} below, which describes the local combinatorics around each edge.

\begin{defn} \label{defn:togglefan}
Let $\Delta$ be a veering triangulation. A tetrahedron in $\Delta$ is called a \textit{toggle tetrahedron} if the colors on its top and bottom edges differ. It is called a \textit{red/blue fan tetrahedron} if both its top and bottom edges are red/blue respectively.
\end{defn}

\begin{prop}[{\cite[Observation 2.6]{FG13}}] \label{prop:togglefan}
Every edge $e$ in $\Delta$ has one tetrahedron above it, one tetrahedron below it, and two stacks of tetrahedra, in between the tetrahedra above and below, on either of its sides.

Each stack must be nonempty. Suppose $e$ is blue/red. If there is exactly one tetrahedron in one stack, then that tetrahedron is a blue/red fan tetrahedron respectively. If there are $n>1$ tetrahedron in one stack, then going from bottom to top, the tetrahedra in that stack are: one toggle tetrahedron, $n-2$ red/blue fan tetrahedra, and one toggle tetrahedron, respectively.
\end{prop}

\begin{defn} \label{defn:fans}
A side of an edge in $\Delta$ is said to be \textit{short} if the stack of tetrahedra to that side has exactly one tetrahedron, otherwise it is said to be \textit{long}.

In \Cref{fig:fans}, we illustrate a blue edge whose left side is short and right side is long. 
\end{defn}

\begin{figure}[ht]
    \centering
    \resizebox{!}{3.5cm}{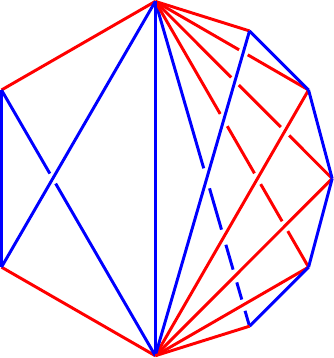}
    \caption{A blue edge whose left side is short and right side is long.}
    \label{fig:fans}
\end{figure}

Associated to a veering triangulation are its \textit{stable branched surface} and \textit{dual graph}. We recall their definitions in \Cref{defn:dualgraph} below.

\begin{defn} \label{defn:branchsurf}
Let $M$ be a 3-manifold. A \textit{branched surface} $B$ is a subset of $M$ locally of the form of one of the pictures in \Cref{fig:branchsurflocal}.
The set of points where $B$ is locally of the form of \Cref{fig:branchsurflocal} middle and right is called the \textit{branch locus} of $B$. The points where $B$ is locally of the form of \Cref{fig:branchsurflocal} right are called the \textit{vertices} of $B$. The complementary regions of the branch locus in $B$ are called the \textit{sectors}. 

\begin{figure}[ht]
    \centering
    \resizebox{!}{1.8cm}{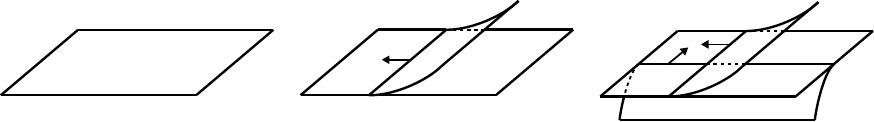} 
    \caption{The local models for branched surfaces. The arrows indicate the maw co-orientation of the branch locus.}
    \label{fig:branchsurflocal}
\end{figure}

The branch locus can be naturally considered as a 4-valent graph by taking the set of vertices to be the set of vertices of $B$ and taking the set of edges to be the complementary regions of the vertices. Each edge of the branch locus has a canonical co-orientation on $B$, which we call the \textit{maw co-orientation}, given locally by the direction from the side with more sectors to the side with less sectors.

Near a vertex $v$, $B$ can be considered as a disc with two sectors attached along two smooth arcs that intersect once at $v$. We call the two attached sectors the \textit{fins} of $B$ at $v$. In \Cref{fig:branchsurflocal} right these would be the topmost and bottommost sectors.
\end{defn}

\begin{defn} \label{defn:dualgraph}
Let $\Delta$ be a veering triangulation on a 3-manifold $M$. We define the \textit{stable branched surface} $B$ of $\Delta$ to be the branched surface which, in each tetrahedron $t$ in $\Delta$, consists of a quadrilateral with its 4 vertices on the top and bottom edges and the two side edges of the same color as the top edge of $t$, and a triangular fin for each side edge of the opposite color to the top edge, as in \Cref{fig:branchsurf} left.

\begin{figure}[ht]
    \centering
    \resizebox{!}{3.5cm}{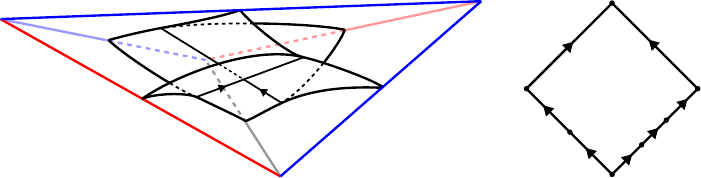}
    \caption{Left: The portion of the stable branched surface and the dual graph within each tetrahedron. Right: The form of each sector of the stable branched surface.}
    \label{fig:branchsurf}
\end{figure}

Consider the branch locus of the stable branched surface. We orient its edges to be positively transverse to the faces of $\Delta$. This defines a directed graph, which we call the \textit{dual graph} of $\Delta$ and denote by $\Gamma$.
\end{defn}

\begin{defn} \label{defn:turns}
Suppose $c$ is an edge path of $\Gamma$, then at a vertex $v$ of $c$, we say that $c$ takes a \textit{branching turn} at $v$ if it is smooth near $v$, otherwise we say that $c$ takes an \textit{anti-branching turn} at $v$. A cycle of $\Gamma$ that only takes branching turns is called a \textit{branch cycle}. A cycle of $\Gamma$ that only takes anti-branching turns is called an \textit{AB cycle}.
\end{defn}

We will in fact deal with the stable branched surface more than the triangulation itself in this paper, so we spend some time describing its combinatorics.

Each sector of the stable branched surface $B$ is dual to an edge of the triangulation $\Delta$, namely the one that passes through it. We define the \textit{color} of a sector to be the color of its dual edge. Also, each vertex of $\Gamma$ is dual to a tetrahedron of $\Delta$, namely the one that it sits inside of. In particular, the number of vertices of $B$ is equal to the number of tetrahedra in $\Delta$.
Under these dual correspondences, the first part of \Cref{prop:togglefan} translates to the following. See also \cite[Section 6.13]{SS19}.

\begin{prop} \label{prop:vbssector}
Each sector $s$ in $B$ is a diamond with the two top sides each being an edge of $\Gamma$ and the two bottom sides divided into edges of $\Gamma$. The number of edges a bottom side is divided into is equal to the number of tetrahedra to that side of the dual edge of $\Delta$. See \Cref{fig:branchsurf} right.
\end{prop}

In particular we can talk about the \textit{top vertex}, the \textit{bottom vertex}, and the two \textit{side vertices} of a sector, the last term meaning the two vertices where a top side meets a bottom side.

\begin{defn} \label{defn:vertexcolor}
A vertex of $B$ is said to be \textit{blue} if $B$ is locally of the form as in \Cref{fig:vbslocal} left, and is said to be \textit{red} if $B$ is locally of the form as in \Cref{fig:vbslocal} right. Note that here we use the fact that the 3-manifold is oriented in order to distinguish the two pictures.
A useful mnemonic for this definition is to look at the fins of $v$. If they protrude out in a \textbf{L}eft/\textbf{R}ight-handed fashion then $v$ is b\textbf{L}ue/\textbf{R}ed, respectively. 
\end{defn}

\begin{figure}[ht]
    \centering
    \fontsize{12pt}{12pt}\selectfont
    \resizebox{!}{4cm}{
\begingroup%
  \makeatletter%
  \providecommand\color[2][]{%
    \errmessage{(Inkscape) Color is used for the text in Inkscape, but the package 'color.sty' is not loaded}%
    \renewcommand\color[2][]{}%
  }%
  \providecommand\transparent[1]{%
    \errmessage{(Inkscape) Transparency is used (non-zero) for the text in Inkscape, but the package 'transparent.sty' is not loaded}%
    \renewcommand\transparent[1]{}%
  }%
  \providecommand\rotatebox[2]{#2}%
  \newcommand*\fsize{\dimexpr\f@size pt\relax}%
  \newcommand*\lineheight[1]{\fontsize{\fsize}{#1\fsize}\selectfont}%
  \ifx\svgwidth\undefined%
    \setlength{\unitlength}{440.91404712bp}%
    \ifx\svgscale\undefined%
      \relax%
    \else%
      \setlength{\unitlength}{\unitlength * \real{\svgscale}}%
    \fi%
  \else%
    \setlength{\unitlength}{\svgwidth}%
  \fi%
  \global\let\svgwidth\undefined%
  \global\let\svgscale\undefined%
  \makeatother%
  \begin{picture}(1,0.4960678)%
    \lineheight{1}%
    \setlength\tabcolsep{0pt}%
    \put(0,0){\includegraphics[width=\unitlength,page=1]{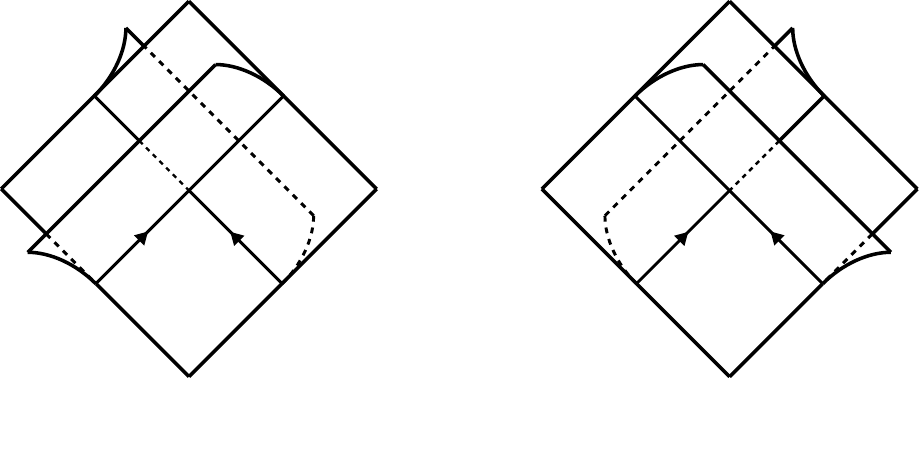}}%
    \put(0.19230895,0.008432){\color[rgb]{0,0,1}\makebox(0,0)[lt]{\lineheight{1.25}\smash{\begin{tabular}[t]{l}\Huge $L$\end{tabular}}}}%
    \put(0.78059007,0.008432){\color[rgb]{1,0,0}\makebox(0,0)[lt]{\lineheight{1.25}\smash{\begin{tabular}[t]{l}\Huge $R$\end{tabular}}}}%
  \end{picture}%
\endgroup%
}
    \caption{Defining the color of a vertex.}
    \label{fig:vbslocal}
\end{figure}

\begin{defn} \label{defn:vbstogglefan}
A sector of $B$ is said to be a \textit{toggle sector} if the colors on its top and bottom vertices differ. Otherwise it is called a \textit{fan sector}.
\end{defn}

With these definitions, an edge of $\Delta$ is blue/red iff its dual sector is blue/red iff the bottom vertex of this dual sector is blue/red, respectively. A sector is toggle/fan iff the tetrahedron dual to its top vertex is toggle/fan respectively. Hence the second part of \Cref{prop:togglefan} translates to the following.

\begin{prop} \label{prop:vbstogglefan}
Let $s$ be a blue/red sector of $B$. On a bottom side of $s$, the vertices at the endpoints are colored blue/red while the vertices in the interior are colored red/blue, respectively. 

Suppose a bottom side of $s$ is divided into edges $e_1,...,e_\delta$, listed from bottom to top. Let $s_i$ be the sector that has $e_i$ as a top side. If $\delta=1$, equivalently, that side of the dual edge in $\Delta$ is short, then $s_1$ is fan. If $\delta \geq 2$, equivalently, that side of the dual edge in $\Delta$ is long, then $s_1$ and $s_\delta$ are toggle while $s_i$ for $i=2,...,\delta-1$ are fan, moreover a top side of $s_i$ is equal to a bottom side of $s_{i+1}$ for $i=2,...,\delta-1$ and a top side of $s_1$ is a proper subset of a bottom side of $s_2$. 
\end{prop}

\begin{prop} \label{prop:vbscontradictions}
A branch cycle in $B$ must meet vertices of both colors.
\end{prop}
\begin{proof}
Suppose otherwise, then there will be some sector that is a fin to every vertex on the branch cycle. That sector will contain the branch cycle as a boundary component. But this contradicts the fact that each sector is a diamond.
\end{proof}

Finally, we recall the notion of a descending path. These were introduced in \cite{LMT23}.

\begin{defn} \label{defn:descendingpath}
A \textit{descending path} on the stable branched surface $B$ is a path that intersects the branch locus of $B$ transversely and induces the maw co-orientation at each intersection, that is, it goes from a side with more sectors to a side with less sectors.
\end{defn}

Suppose $c$ is an edge path of the dual graph $\Gamma$. By pushing $c$ downwards slightly in $B$ and reversing its orientation, we obtain a descending path $\alpha$ on $B$. We illustrate a local picture of this procedure in \Cref{fig:descendingpath} left. In \Cref{fig:descendingpath} right, we show a bigger example, drawn in the style of the figures used in the rest of this paper.  

The intersection points of $\alpha$ with the branch locus of $B$ are in one-to-one correspondence with the vertices of $c$. At such an intersection point $x$, $\alpha$ meets the boundary of exactly one sector $s$ which it does not locally meets the interior of. We say that $s$ \textit{merges into} $\alpha$ at $x$. In particular $s$ is a fin of the vertex of $c$ corresponding to $x$.

\begin{figure}[ht]
    \centering
    \resizebox{!}{4.5cm}{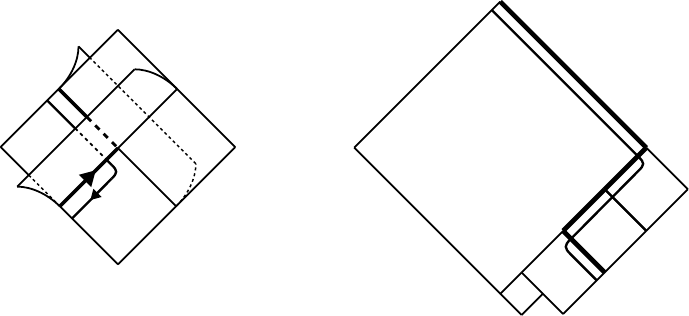}
    \caption{Pushing an edge path of the dual graph (thickened) downwards and reversing its orientation to get a descending path.}
    \label{fig:descendingpath}
\end{figure}

\subsection{Pseudo-Anosov mapping tori} \label{subsec:layeredvt}

We set up some terminology on pseudo-Anosov mapping tori and layered veering triangulations.

\begin{defn} \label{defn:pamap}
A homeomorphism $f$ on a finite-type surface $S$ is said to be \textit{pseudo-Anosov} if there exists a pair of transverse singular measured foliations $(l^s,\mu^s)$ and $(l^u,\mu^u)$ such that $f_* (l^s,\mu^s) = (l^s,\lambda^{-1} \mu^s)$ and $f_* (l^u,\mu^u) = (l^u,\lambda \mu^u)$ for some $\lambda >1$. 
See \cite[Exposé 9]{FLP79} or \cite[Definition 2.1]{HT22} for more details.

We say that $f$ is \textit{fully punctured} if the set of singular points of $l^s$ and $l^u$ is equal to the set of punctures and is nonempty.
\end{defn}

Let $f:S \to S$ be a fully-punctured pseudo-Anosov map. By blowing air into the leaves of the stable measured foliation that contain the punctures, we obtain a measured lamination, which we still denote by $(l^s,\mu^s)$. 

Let $T_f$ be the mapping torus of $f$. Recall this is constructed by taking $S \times [0,1]$ and gluing $S \times \{1\}$ to $S \times \{0\}$ by $f$. The suspension of the lamination $l^s$, that is, the image of $l^s \times [0,1]$, is a lamination on $T_f$, which we call the \textit{stable lamination} and denote by $\Lambda^s$.

\begin{thm}[{\cite{Ago11}, \cite[Theorem 9.1]{LT23}}] \label{thm:layeredvt}
There is a unique veering triangulation $\Delta_f$ on $T_f$ whose stable branched surface carries $\Lambda^s$. 
Moreover, one can find surfaces in $T_f$ in the same isotopy class as $S \times \{0\}$ that are positively transverse to the edges and does not pass through the vertices of the dual graph $\Gamma$ of $\Delta_f$.
\end{thm}

We refer to $\Delta_f$ as the \textit{layered veering triangulation} associated to $f$.
We refer to a surface in the second sentence of the theorem as a \textit{fiber surface}.
We will denote fiber surfaces by $S$ in this paper. This slight abuse of notation is justified because a fiber surface must be homeomorphic to the surface $f$ is defined on. 

\begin{prop} \label{prop:nonullcycles}
Let $S$ be a fiber surface and let $c$ be a cycle of the dual graph $\Gamma$, then $S$ and $c$ must intersect.
\end{prop}
\begin{proof}
By \cite[Theorem 5.1]{LMT24}, the cone over the fibered face that $[S] \in H_2(T_f, \partial T_f)$ lies in is dual to the cone of all $\Gamma$-cycles in $H_1(T_f)$. By \cite[Theorem 5.15]{LMT24}, none of the $\Gamma$-cycles are null-homologous, so since $[S]$ lies in the interior of the cone, $[S]$ has positive intersection number with all $\Gamma$-cycles.

Alternatively, this also follows from the construction of $\Delta_f$ in \cite{Ago11}.
\end{proof}

For the rest of this section, we fix a fully-punctured pseudo-Anosov map $f:S \to S$. Let $\Delta_f$ be the associated veering triangulation. Let $B$ be the stable branched surface and $\Gamma$ be the dual graph of $\Delta_f$.  

The homology class of $S \times \{0\}$ determines a $\mathbb{Z}$-covering of $T_f$, which we denote by $\widehat{T_f}$. $\Lambda^s$ lifts to a lamination $\widehat{\Lambda^s}$ on $\widehat{T_f}$. Unlike $\Lambda^s$, $\widehat{\Lambda^s}$ has a natural transverse measure induced by $\mu^s$. We denote this measure as $\widehat{\mu^s}$. Let $g$ be the generator of the deck transformation group of $\widehat{T_f}$ for which $g_* \widehat{\mu^s} = \lambda \widehat{\mu^s}$, that is, $g$ shifts $\widehat{T_f}$ upwards.

Let $\widehat{B}$ be the lift of $B$ to $\widehat{T_f}$. $\widehat{B}$ carries $\widehat{\Lambda^s}$, hence we can define the \textit{weight} of a sector to be the total $\widehat{\mu^s}$-measure of the leaves that are carried by the sector.

Fix a fiber surface $S$.  
The intersection of $S$ with $B$ determines a train track $\tau$ on $S$. Meanwhile, $S$ intersects each sector $s$ of $B$ in disjoint \textit{arcs}. Each arc has endpoints on edges of $\Gamma$ that lie on different sides of $s$. We order the arcs from bottom to top, using terms such as `the bottommost arc on $s$'. Notice that these arcs, over all sectors, are exactly the branches of $\tau$. 

Similarly, the intersection points of $S$ with $B$ are exactly the switches of $\tau$.
Now, each complementary region of $\tau$ in $S$ is a once-punctured polygon. Recall that the \textit{index} of such a polygon with $n$-cusps on its boundary is $-\frac{n}{2}$. By the Poincaré-Hopf theorem, the sum of indices of the complementary regions is equal to the Euler characteristic of the entire surface $S$. From this, one can compute that the number of intersection points of $S$ with $B$ is $-2\chi(S)$.

Fix a homeomorphic lift of $S$ to $\widehat{T_f}$, which we denote by $S_0$. Write $S_r=g^r \cdot S_0$. Each $S_r$ is a separating surface in $\widehat{T_f}$, hence it makes sense to say that a set lies above or below some $S_r$. For each sector $\widehat{s}$ of $\widehat{B}$, suppose $r$ is the smallest integer such that $\widehat{s}$ lies above $S_{r-1}$. If $\widehat{s}$ intersects any $S_i$ at all, then $r$ is also the smallest integer such that $S_r$ intersects $\widehat{s}$. In this case we say that $\widehat{s}$ is \textit{at height $r$}. Notice that each sector $s$ of $B$ has a collection of lifts $\{\widehat{s_r}\}_{r \in \mathbb{Z}}$ where each $\widehat{s_r}$ is at height $r$. We define the \textit{weight} of $s$ to be the weight of $\widehat{s_0}$. We caution the reader that this does NOT make $B$ a measured branched surface, since this measure may not be additive across the branch locus.

Finally, let $\widehat{\Gamma}$ be the lift of $\Gamma$ to $\widehat{T_f}$. We say that a vertex $\widehat{v}$ of $\widehat{\Gamma}$ is \textit{at height $r$} if $r$ is the smallest integer such that $\widehat{v}$ lies above $S_{r-1}$. Notice that if a sector $\widehat{s}$ of $\widehat{B}$ has a vertex at height $r$ on its boundary, then $\widehat{s}$ is of height $\leq r$.

\section{Single hook circuits} \label{sec:singlehook}

In this section we will show \Cref{thm:introsinglehook}. The proof uses a special type of circuit in the dual graph, which we call hook circuits. The title of the section comes from the fact that we only make use of a single hook circuit in this section, as opposed to \Cref{sec:doublehook} where we use up to two hook circuits. 

In \Cref{subsec:singlehookbound}, we show how the existence of single hook circuits can be used to bound the number of tetrahedra. 
The rest of the section is then devoted to analyzing when single hook circuits exist. 
In \Cref{subsec:Alexanderduality}, we build up the machinery of A-quads and B-quads. The key proposition is \Cref{prop:Alexanderduality}, which states that via Alexander duality, the existence of certain circuits is related to the second homology map of the inclusion of certain 2-complexes built from these quads into the 3-manifold. 
In \Cref{subsec:singlehookexist}, we use this machinery to show that the obstruction to Eulerian hook circuits are certain sector identifications of the branched surface. If these obstructions exist in many places, then the global combinatorics is highly constrained. From this, we will show that the desired single hook circuits exist for all but one triangulation (\Cref{prop:singlehookexist}).

\subsection{Bounding the number of vertices using single hook circuits} \label{subsec:singlehookbound}

We fix the following setting. 
Let $T_f$ be the mapping torus of a fully-punctured pseudo-Anosov map $f:S \to S$. 
Let $\chi$ be the Euler characteristic of $S$ and $\lambda$ be the dilatation of $f$.
Let $\Delta$ be the veering triangulation on $T_f$ associated to $f$, let $B$ be the stable branched surface of $\Delta$, and let $\Gamma$ be the dual graph of $\Delta$. Finally, let $S$ be a fiber surface. We use notation as in \Cref{subsec:layeredvt}.

\begin{defn} \label{defn:circuit}
An edge path in a directed graph is \textit{simple} if it does not repeat edges (but it is allowed to repeat vertices). A \textit{circuit} is a simple cycle.  An \textit{Eulerian circuit} is a circuit that traverses each edge exactly once.
\end{defn}

\begin{defn} \label{defn:hookcircuit}
Let $s$ be a sector of $B$.  
Suppose the two bottom sides of $s$ are divided into $\delta_1$ and $\delta_2$ edges of $\Gamma$ respectively. We label these, from bottom to top, as $e^1_1,...,e^1_{\delta_1}$ and $e^2_1,...,e^2_{\delta_2}$ respectively. 
We also label the edge of $\Gamma$ that is the top side of $s$ above $e^\beta_{\delta_\beta}$ to be $e^\beta_{\delta_\beta+1}$, for $\beta = 1,2$. 
Suppose the bottommost arc on $s$ has endpoints on $e^1_{k_1}$ and $e^2_{k_2}$. The \textit{hook} of $s$ on the side $\beta$, which we denote by $h_\beta$, is the path $(e^\beta_{k_\beta}, e^\beta_{k_\beta+1},..., e^\beta_{\delta_\beta+1})$. If $k_\beta=1$ we say that the hook $h_\beta$ is \textit{deep}. We illustrate an example of a sector in \Cref{fig:hooks} whose left hook is deep but right hook is not. If $s$ does not have any arcs, we leave its hooks undefined.

\begin{figure}[ht]
    \centering
    \resizebox{!}{5cm}{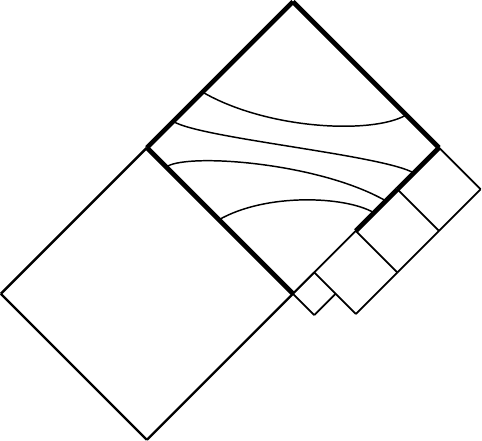}
    \caption{The two hooks (thickened) of a sector with the given placement of arcs. Here the left hook is deep but the right hook is not.}
    \label{fig:hooks}
\end{figure}

A circuit $c$ in $\Gamma$ is said to \textit{hook around $s$ on side $\beta$} if it contains $h_\beta$ as a sub-path. When $h_\beta$ is understood, we will refer to the vertices of $c$ in the interior of $h_\beta$ as the \textit{hook vertices} of $c$. 

There is a degenerate case here when $k_\beta \neq \delta_\beta+1$ and $e^\beta_{\delta_\beta+1}=e^\beta_{k_\beta}$ (which forces $k_\beta=1$ by \Cref{prop:vbstogglefan}). In this case, $h_\beta$ is not simple; we will say that the circuit $(e^\beta_i)_{i \in \mathbb{Z}/\delta_\beta}$ \textit{hooks around $s$ on side $\beta$}. In other words, we allow non-simple paths to be sub-paths of circuits if they traverse the same sequence of edges in the same order. In this case, all the vertices on $(e^\beta_i)_{i \in \mathbb{Z}/\delta_\beta}$ are hook vertices.

We refer to circuits that hook around a sector as \textit{hook circuits} in general.
\end{defn}

We will essentially only be interested in the case when $s$ is a sector of minimum weight. In this case, $s$ must contains arcs. One way to see this is to consider the lift $\widehat{s_0}$ of $s$ in $\widehat{B}$ that is at height 0. If $s$ does not contain arcs, the same is true for $\widehat{s_0}$, so the sector on top of $\widehat{s_0}$ is at height 0 as well, but of smaller weight, implying that the sector of $B$ it covers is of smaller weight than $s$. 

Similarly, in this case one can see that $k_1$ and $k_2$ in \Cref{defn:hookcircuit} must be at most $\delta_1$ and $\delta_2$ respectively. Otherwise, say $j_1 = \delta_1 + 1$, then the two sectors with $e_{\delta_1+1}$ along their bottom sides must have smaller weight by the same reasoning.

\begin{prop} \label{prop:singlehookbound}
Let $s$ be a sector of $B$ of minimum weight. Suppose there exists an Eulerian circuit $c$ of $\Gamma$ which hooks around $s$. Then the number of tetrahedra in the veering triangulation is $\leq \frac{1}{2} \lambda^{-2\chi}$.
\end{prop}
\begin{proof}
Let the weight of $s$ be $w$. Let $\widehat{s_0}$ be the lift of $s$ in $\widehat{B}$ that is of height 0. Recall that the weight of $\widehat{s_0}$ is $w$. Take the basepoint of $c$ to be the top vertex of $s$. Lift $c$ to a path $\widehat{c}$ ending at the top vertex of $\widehat{s_0}$, reverse its orientation, then push it downwards to get a descending path $\alpha$.

Since $c$ is Eulerian, it passes through each vertex of $B$ two times, hence its length is two times the number of tetrahedra $N$ in the triangulation. Meanwhile, $\Gamma$ intersects the fiber surface $S$ for $-2\chi(S)$ times, hence $c$ intersects $S$ for $-2\chi(S)$ times as well. These facts imply that the starting point of $\alpha$ lies on a sector of weight $w$, the ending point of $\alpha$ lies on a sector of weight $\lambda^{-2\chi} w$, and $\alpha$ passes through the branch locus of $\widehat{B}$ for $2N$ times. Hence we get an equation
\begin{equation} \label{eq:singlehook}
    \lambda^{-2\chi} w = w + \sum_{j=1}^{2N} w_j
\end{equation}
Here $w_j$ is the weight of the $j^{\text{th}}$ sector merging into $\alpha$. Our task is to show a lower bound for $\sum_{j=1}^{2N} w_j$ in terms of $w$ and $N$.

We first set up some notation. For each vertex $v$, $c$ either takes an anti-branching turn or a branching turn at $v$ for the both times it visits $v$. We say that $v$ is \textit{A-resolved} or \textit{B-resolved} respectively in those cases. This terminology is motivated from the perspective that we are `resolving' the directed graph $\Gamma$ at each vertex to produce a 1-manifold, and will be consistent with our future discussion on existence of hook circuits.

Recall that each intersection point of $\alpha$ with $\widehat{\Gamma}$ corresponds naturally to a vertex of $\widehat{c}$. Under this correspondence, the term in $W=\sum_{j=1}^{2N} w_j$ that arises from an intersection point of $\alpha$ with $\widehat{\Gamma}$ is the weight of one of the fins of the corresponding vertex $v$ of $\widehat{c}$. The height of this sector is bounded above by the height $r_v$ of $v$, hence the weight of the sector is bounded below by $\lambda^{-r_v} w$. If $v$ does not cover a hook vertex of $c$, then $v$ will be at non-positive height, since $-\widehat{c}$ will have gone through $S_0$ after having traversed the hook at the start, hence the corresponding term is $\geq w$. This argument does not work if $v$ does cover a hook vertex of $c$ however, so we will need to modify our strategy a little bit.

What we will do is that for each B-resolved hook vertex $u$, we pair up the two terms in $W$ that correspond to the two vertices of $\widehat{c}$ covering the vertex of $\Gamma$ that $u$ lies at. See \Cref{fig:hookpairing}. Consider a lift $\widehat{u}$ of $u$. Let $d_u$ be the difference between the heights of the two fins of $\widehat{u}$. This quantity only depends on $u$ since any other lift is $g^r \cdot \widehat{u}$ for some $r$, and $g^r$ changes the heights of both fins by $r$. For the two terms of $W$ in such a pair, we can now bound their sum from below by $\lambda^{-r_1+d_u} w + \lambda^{-r_2-d_u} w$ where $r_i$ is the height of the fin opposite to the one whose weight is added (up to relabeling $r_1, r_2$ or changing the sign of $d_u$). 

\begin{figure}[ht]
    \centering
    \resizebox{!}{7cm}{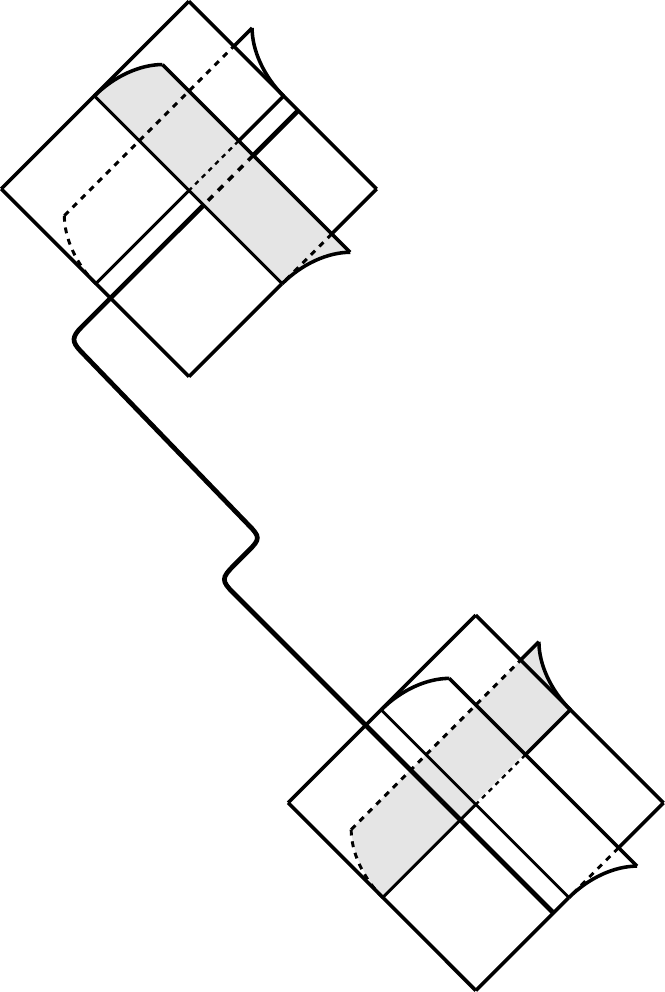}
    \caption{For each B-resolved hook vertex $u$, we pair up the two vertices of $\widehat{c}$ covering the vertex of $\Gamma$ that $u$ lies at. The two corresponding terms are the weights of sectors (shaded) covering the two fins at $u$.}
    \label{fig:hookpairing}
\end{figure}

We claim that $r_1$ and $r_2$ must be non-positive. If the corresponding vertex of $\widehat{c}$ covers a hook vertex, then this follows from the fact that $\widehat{s_0}$ is at height 0, since in this case the fin opposite to the one whose weight is added is $\widehat{s_0}$. Otherwise, this follows from the fact that the corresponding vertex of $\widehat{c}$ has non-positive height. Hence $\lambda^{-r_1+d_u} w + \lambda^{-r_2-d_u} w \geq \lambda^{d_u} w + \lambda^{-d_u} w \geq 2w$. 

Now, there is only one A-resolved hook vertex, so by doing this grouping for each B-resolved hook vertex and using the original argument for the other vertices of $\widehat{c}$ not covering a hook vertex, we see that $W=\sum_{j=1}^{2N} w_j \geq (2N-1)w$. Hence from \Cref{eq:singlehook}, we get $\lambda^{-2\chi} \geq 2N$ which implies the bound in the proposition.
\end{proof}

\subsection{Resolutions and Alexander duality} \label{subsec:Alexanderduality}

To utilize \Cref{prop:singlehookbound}, we need to be able to find hook circuits. The way we will do so is to change our perspective and study the connectedness of the dual graph after resolutions. In turn, the way we study this connectedness will be to use a version of Alexander duality. In this subsection we will explain these ideas.

\begin{defn} \label{defn:resolution}
A \textit{(2,2)-valent directed graph} is a directed graph for which every vertex has exactly two incoming edges and exactly two outgoing edges.

Let $G$ be a (2,2)-valent directed graph. Let $v$ be a vertex of $G$. Let $i_1, i_2$ be the incoming edges and $o_1, o_2$ be the outgoing edges at $v$. We define \textit{a resolution at $v$} to be the operation of deleting $v$ from $G$ then joining $i_\beta$ to $o_{\sigma(\beta)}$ for some permutation $\sigma$, which returns another (2,2)-valent directed graph. Notice that any vertex can be resolved in two ways.

Suppose $G'$ is obtained from $G$ by repeated resolutions of vertices. Then $G'$ is determined by a collection $I=\{(v_i,\sigma_i)\}$ where the $v_i$ are distinct vertices of $G$ and $\sigma_i$ are permutations dictating how the resolutions are performed at $v_i$. In this case we write $G'=G(I)$. Notice that there is a natural map $G' \to G$ defined by gluing back the resolved vertices. Also, notice that the notation is set up such that $G(I_1 \cup I_2)$ makes sense as long as $I_1$ and $I_2$ do not contain contradicting instructions on how to resolve a vertex.

A simple edge path or cycle $c$ of $G$ determines a resolution of $G$ by resolving every vertex in the interior of $c$ in a manner such that $c$ lifts to a subset of an edge in the resolved graph $G'$. In this case we use $c$ to denote the collection that determines the resolution, that is, we write $G'=G(c)$.
\end{defn}

For a general (2,2)-valent directed graph, there is no good way of referring to the two ways of resolving a vertex. But for dual graphs of veering triangulations, we can use the following terminology, which was already hinted at in the proof of \Cref{prop:singlehookbound}.

\begin{defn} \label{defn:ABresolution}
Let $v$ be a vertex of the dual graph of a veering triangulation. If in \Cref{defn:resolution}, $(i_\beta,o_{\sigma(\beta)})$ takes an anti-branching turn at $v$, then we call the resolution an \textit{A-resolution}. Otherwise, $(i_\beta,o_{\sigma(\beta)})$ takes a branching turn at $v$, and we call the resolution a \textit{B-resolution}.
\end{defn}

With this notation, the existence of an Eulerian hook circuit can be rephrased as

\begin{prop} \label{prop:singlehookresolution}
There exists an Eulerian circuit that hooks around a sector $s$ on side $\beta$ if and only if $\Gamma(h_\beta)$ is connected.
\end{prop}
\begin{proof}
For the forward direction, the Eulerian hook circuit $c$ in $\Gamma$ lifts to an Eulerian circuit in $\Gamma(h_\beta)$, hence $\Gamma(h_\beta)$ is connected.
For the backward direction, since $\Gamma(h_\beta)$ is a connected (2,2)-valent directed graph, by a classical theorem of Euler, $\Gamma(h_\beta)$ has a Eulerian circuit $c$. The image of $c$ in $\Gamma$ is a Eulerian circuit that contains $h_\beta$ as a sub-path.
\end{proof}

This shift in perspective to resolutions transfers the question that one must answer to apply \Cref{prop:singlehookbound} into: how can one study the connectedness of resolutions of the dual graph? Here is where Alexander duality comes in.

Consider the general setting of having a cell complex $X$. For our purposes, this means that $X$ is obtained by gluing $n$-balls along faces on their boundary that are homeomorphic to $n-1$-balls by homeomorphisms. The \textit{dual graph} of $X$ is defined to be the graph with set of vertices equal to the set of $n$-balls and an edge between $e_1$ and $e_2$ for every pair of faces on $e_1$ and $e_2$ that are glued together. Then a general position argument shows that the number of components of $X$ is equal to the number of components of its dual graph.

When $\Delta$ is a veering triangulation, this definition of dual graph agrees with the definition in \Cref{defn:dualgraph} that we have used so far, with the edge orientations forgotten. In particular this shows that the dual graph of a veering triangulation is always connected.

Suppose $v$ is a vertex of the dual graph $\Gamma$ of a veering triangulation $\Delta$. Recall that $v$ is dual to a tetrahedron $t$ of $\Delta$. We define the \textit{A-quad} of $t$ to be a properly embedded quadrilateral-with-4-ideal-vertices with edges along the top edge, the 2 side edges of the opposite color as the top edge, and the bottom edge of $t$. We define the \textit{B-quad} of $t$ to be a properly embedded quadrilateral-with-4-ideal-vertices with edges along the top edge, the 2 side edges of the same color as the top edge, and the bottom edge of $t$. We will only consider A/B-quads up to isotopy hence these are uniquely defined for each tetrahedron.

If we now define a new cell complex $\Delta'$ by cutting $\Delta$ along the A/B-quad of $t$, then the dual graph of $\Delta'$ is exactly the A/B-resolution of $\Gamma$ at $v$, respectively (again, with the edge orientations forgotten). 

More generally, the graph $\Gamma(I)$ obtained by resolving according to some collection $I$ is the dual graph of the cell complex obtained by cutting $\Delta$ along the union of A/B-quads in the corresponding collection of tetrahedra, where we take the A/B-quad in a tetrahedron if we A/B-resolve the corresponding vertex respectively.

At this point, it is convenient for our discussion to compactify the triangulation. That is, we cut away a neighborhood of each end of $T_f$, so that it has torus boundary components instead of torus ends. Correspondingly, $\Delta$ becomes a triangulation by truncated tetrahedra, and each quad becomes an octagon with every other side lying along $\partial T_f$, however we will still refer to them as quads. 

In the above scenario, we let $Q=Q(I)$ be the union of A/B-quads, and let $\partial_v Q = Q \cap \partial T_f$. Here $Q$ is a 2-complex and via a version of Alexander duality, we can compute the number of components of $\Delta \cut Q$ using the homology of $Q$. For concreteness, let us use homology with $\mathbb{R}$-coefficients.

\begin{prop} \label{prop:Alexanderduality}
We have $\widetilde{H}^0(\Delta \cut Q) \cong \ker(H_2(Q, \partial_v Q) \to H_2(T_f, \partial T_f))$. In particular, the number of components of $\Gamma(I)$ is equal to $1+ \dim \ker(H_2(Q,\partial_v Q) \to H_2(T_f, \partial T_f))$
\end{prop}
\begin{proof}
Let $U$ be a regular neighborhood of $Q$ in $T_f$. Let $\partial_v U = U \cap \partial T_f$ and $\partial_h U = \partial U \cut \partial_v U$. $\partial_v U$ is a regular neighborhood of $\partial_v Q$ on $\partial T_f$.

Consider the commutative diagram of long exact sequences

\begin{center}
\begin{tikzcd}[
  column sep=small,row sep=small,
  ar symbol/.style = {draw=none,"\textstyle#1" description,sloped},
  cong/.style = {ar symbol={\cong}},
  ]
0 \ar[d,cong] & \mathbb{R} \ar[d,cong] & & H_2(Q,\partial_v Q) \ar[d,cong] & \\
H_3(U, \partial_v U) \ar[r] \ar[dd] & H_3(T_f,\partial T_f) \ar[r] \ar[dd] & H_3(T_f, U \cup \partial T_f) \ar[r] \ar[d,cong] & H_2(U,\partial_v U) \ar[r] \ar[dd] & H_2(T_f,\partial T_f) \ar[dd] \\
 & & H_3(T_f \cut U, \partial (T_f \cut U)) \ar[d] & & \\
H^0(U, \partial_h U) \ar[r] \ar[d,cong] & H^0(T_f) \ar[r] \ar[d,cong] & H^0(T_f \cut U) \ar[r] \ar[d,cong] & {H^1(U, \partial_h U)} \ar[r] & H^1(T_f) \\
0 & \mathbb{R} & H^0(T_f \cut Q) & & 
\end{tikzcd}
\end{center}

The vertical arrows are isomorphisms given by Poincaré (or Poincaré-Lefschetz) duality. Hence this shows that $$\widetilde{H}^0(T_f \cut Q) \cong \ker(H_2(Q, \partial_v Q) \to H_2(T_f, \partial T_f)) \eqno\qedhere$$
\end{proof}

In practice, we will need to know how the isomorphism in \Cref{prop:Alexanderduality} actually operates. For a given component $C$ of $\Delta \cut Q$, its boundary in $\Delta$ will consist of a collection of quads. If we co-orient each of these quads to point out of $C$, their sum will determine the corresponding 2-cycle in $\ker(H_2(Q,\partial_v Q) \to H_2(T_f, \partial T_f))$. Note that if both sides of a quad lie in $C$, then the quad will be included twice in the sum with opposite co-orientations, hence end up not appearing in the 2-cycle. In general, the coefficient of each quad that appears in the cycle will be $\pm 1$.

Conversely, given a nonzero 2-cycle in $\ker(H_2(Q,\partial_v Q) \to H_2(T_f, \partial T_f))$ with coefficients $\pm 1$, we can first co-orient the quads appropriately so that the coefficients are all 1. Then the quads will bound a component of $\Delta \cut Q$ for which they are co-oriented out of. 
The fact that we have $\widetilde{H}^0$ instead of $H^0$ on the left hand side of \Cref{prop:Alexanderduality} means that if we add together the 2-cycles corresponding to each of the components of $\Delta \cut Q$, the quads all cancel each other out and we get the zero sum.

\subsection{Existence of single hook circuits}  \label{subsec:singlehookexist}

The main goal of this subsection is to prove the following proposition.

\begin{prop} \label{prop:singlehookexist}
There exists a choice of fiber surface $S$ for which there is a Eulerian circuit that hooks around a sector of minimum weight, unless $\Delta=$ \texttt{cPcbbbdxm\_10}.
\end{prop}

We remark that \texttt{cPcbbbdxm\_10} is a triangulation of the figure eight knot complement, and is in fact the first member of the veering triangulation census \cite{GSS}.

The proof of \Cref{prop:singlehookexist} is quite elaborate. It will first involve a multi-step analysis of the connectedness of resolutions of the dual graph determined by hooks, then go through an elimination process making use of a result on the flexibility of fiber surfaces. The proof will occupy the rest of this section.

We first set up some notation. Suppose $s$ is a sector of $B$ of minimum weight. Without loss of generality suppose that $s$ is blue. Label the edges of $\Gamma$ on the sides of $s$ as in \Cref{defn:hookcircuit}. Let $s^\beta_k$ be the sector that has $e^\beta_k$ as a top side. Let $v^\beta_k$ be the top vertex of $s^\beta_k$. In other words, $v^\beta_k$ is the vertex shared by $e^\beta_k$ and $e^\beta_{k+1}$. By extension, we also let $v^\beta_0$ be the bottom vertex of $s$ and write $v_0=v^\beta_0$ for both sides $\beta$. See \Cref{fig:sectorsetup}. For the figures in this section, we will take $\beta$ to be the left hand side of $s$.

\begin{figure}
    \centering
    \fontsize{8pt}{8pt}\selectfont
    \resizebox{!}{5cm}{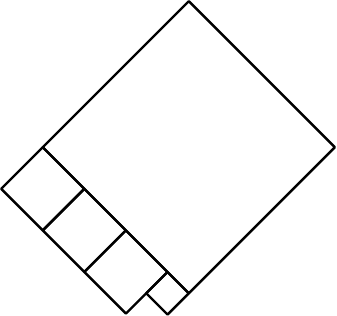}
    \caption{Setting up notation for the edges, sectors, and vertices adjacent to $s$. In this example, $\delta_\beta=4$.}
    \label{fig:sectorsetup}
\end{figure}

Notice that the vertices $v^\beta_k$ are not necessarily distinct, but if we have $v^\beta_{k_1}=v^\beta_{k_2}$ for distinct $k_1, k_2 = 1,...,\delta_\beta$, then the identification cannot be such that $e^\beta_{k_1} = e^\beta_{k_2}$, for otherwise the sector $s$ will not be embedded in its interior. Also, we cannot have $v^\beta_1=v^\beta_k$ for $2 \leq k \leq \delta_\beta$, since $v^\beta_1$ is the top vertex of a toggle sector, whereas $v^\beta_k$ are the top vertices of fan sectors for $2 \leq k \leq \delta_\beta-1$, and $v^\beta_{\delta_\beta}$ is of the opposite color as $v^\beta_1$, by \Cref{prop:vbstogglefan}. For the same reason, we cannot have $v^\beta_{\delta_\beta}=v^\beta_k$ for $1 \leq k \leq \delta_\beta-1$.

Let $h'_\beta$ be the edge path $(e^\beta_2,...,e^\beta_{\delta_\beta})$ and consider the resolution $\Gamma(h'_\beta)$ of $\Gamma$ determined by $h'_\beta$, for some chosen side $\beta$. Notice that all the resolutions in $h'_\beta$ are B-resolutions of red vertices. 

\begin{lemma} \label{lemma:singleprehook}
$\Gamma(h'_\beta)$ is always connected.
\end{lemma}
\begin{proof}
Notice that the lemma is vacuously true if $\delta_\beta=1$, so we can assume that $\delta_\beta \geq 2$.

Suppose that $\Gamma(h'_\beta)$ is not connected. Let $C_0$ be the component that contains $h'_\beta$ and let $C_1$ be a component that is not $C_0$. Let $J$ be the collection of indices $k$ such that $v^\beta_k$ meets the image of $C_1$ under the map $\Gamma(h'_\beta) \to \Gamma$. $J$ must be nonempty since the only resolutions we perform are at $v^\beta_k$.

Consider the union of sectors $s^\beta_2 \cup ... \cup s^\beta_{\delta_\beta}$, which is a rectangle. Since the only resolutions we perform are B-resolutions on red vertices, by following along the top sides of the rectangle, we see that the side vertex of $s^\beta_{\delta_\beta}$ other than $v^\beta_{\delta_\beta-1}$, which we denote by $v$, meets $C_0$. If this vertex does not meet $C_1$, then again by the fact that the only resolutions we perform are B-resolutions of red vertices, by following along the bottom side of the rectangle that meets $v$, we see that none of the sides of the rectangle meet $C_1$. This implies that for any $k \in J$, the bottom vertex of $s^\beta_{k+1}$ must be resolved and meets $C_1$, hence is equal to $v^\beta_{k'}$ for some $k' \in J$.

The assignment $k \mapsto k'$ thus defines a permutation $\sigma: J \to J$, under which $C_1$ is the union of the top sides of $s^\beta_k$ in the interior of the rectangle for $k \in J$, joined together according to $\sigma$. But then $C_1$ will be a branch cycle containing only red vertices, contradicting \Cref{prop:vbscontradictions}. 

Hence $v$ must be resolved and meets $C_1$, which implies that $v=v^\beta_k$ for some $k \in J$. Moreover, as seen above, the top sides of the rectangle lie in $C_0$, hence the identification of vertices must be in the manner such that $s^\beta_{\delta_\beta}=s^\beta_{k+1}$.

However, since the top vertex of $s^\beta_{\delta_\beta}$ is blue, we must have $k=\delta_\beta-1$. But this cannot be since it implies that $s^\beta_{\delta_\beta}$ is not embedded in its interior.
\end{proof}

\Cref{lemma:singleprehook} forms the basis of our first application of \Cref{prop:Alexanderduality} in the following \Cref{lemma:singlenondeephook}. 
Before that, we set up yet more notation. Let $q^\beta_k$ be the B-quad in the tetrahedron dual to $v^\beta_k$, for $k=1,...,\delta_\beta-1$, and let $q^\beta_{\delta_\beta}$ be the A-quad in the tetrahedron dual to $v^\beta_{\delta_\beta}$. 

We illustrate these quads in \Cref{fig:fanquad} right, in the case when $\beta$ is short at the top and in the case when $\beta$ is long at the bottom. In the figure, we lay out the quads from left to right by the tetrahedra in which they lie from top to bottom in the stack. For each quad, the top left edge is the top edge of the tetrahedron it lies in, while the bottom right edge is the bottom edge of the tetrahedron it lies in. In particular, for each $i$, the two edges on the right of $q^\beta_i$ are the same two edges on the left of $q^\beta_{i-1}$, so if we take the union $\bigcup_{i=1}^{\delta_\beta} q^\beta_i$ of all the quads, they fit together to give a big diamond with two blue edges and two red edges in the boundary, see \Cref{fig:fanquadunion}.

\begin{figure}
    \centering
    \fontsize{12pt}{12pt}\selectfont    
    \resizebox{!}{9cm}{
\begingroup%
  \makeatletter%
  \providecommand\color[2][]{%
    \errmessage{(Inkscape) Color is used for the text in Inkscape, but the package 'color.sty' is not loaded}%
    \renewcommand\color[2][]{}%
  }%
  \providecommand\transparent[1]{%
    \errmessage{(Inkscape) Transparency is used (non-zero) for the text in Inkscape, but the package 'transparent.sty' is not loaded}%
    \renewcommand\transparent[1]{}%
  }%
  \providecommand\rotatebox[2]{#2}%
  \newcommand*\fsize{\dimexpr\f@size pt\relax}%
  \newcommand*\lineheight[1]{\fontsize{\fsize}{#1\fsize}\selectfont}%
  \ifx\svgwidth\undefined%
    \setlength{\unitlength}{574.57228773bp}%
    \ifx\svgscale\undefined%
      \relax%
    \else%
      \setlength{\unitlength}{\unitlength * \real{\svgscale}}%
    \fi%
  \else%
    \setlength{\unitlength}{\svgwidth}%
  \fi%
  \global\let\svgwidth\undefined%
  \global\let\svgscale\undefined%
  \makeatother%
  \begin{picture}(1,0.68150071)%
    \lineheight{1}%
    \setlength\tabcolsep{0pt}%
    \put(0,0){\includegraphics[width=\unitlength,page=1]{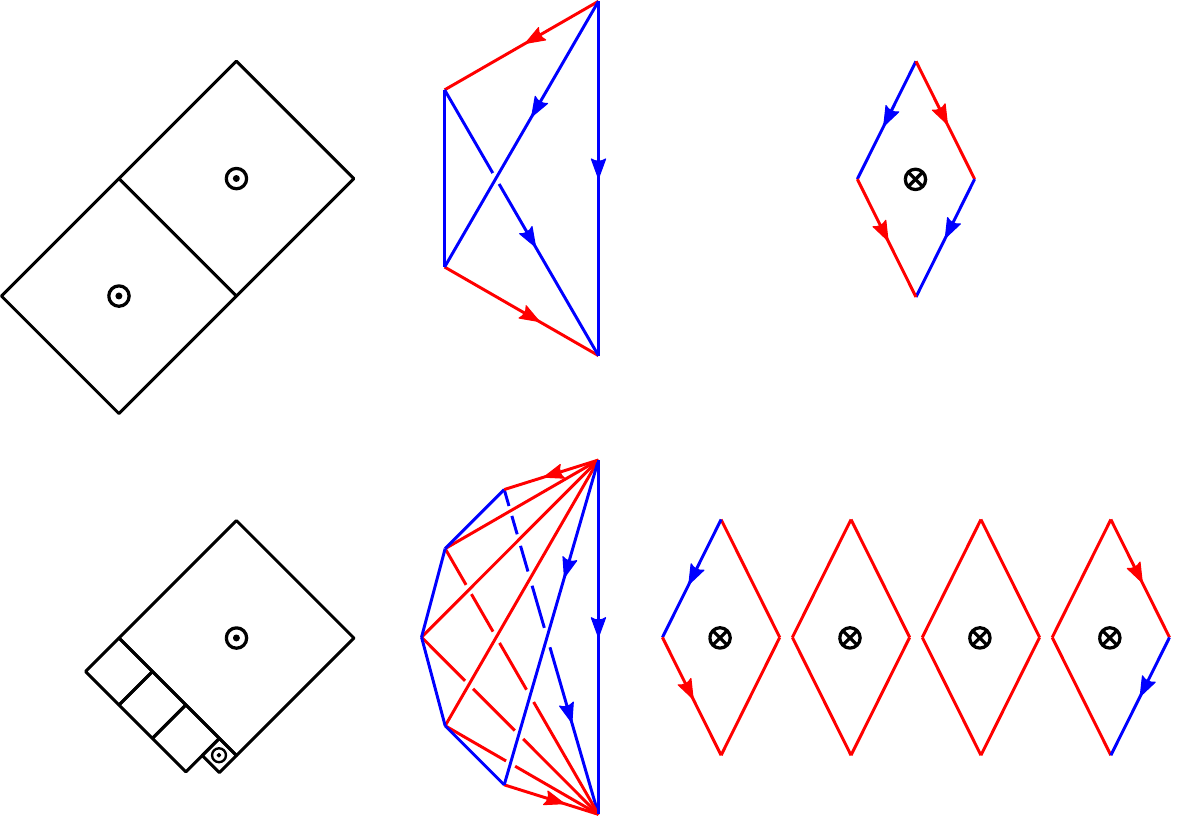}}%
    \put(0.74713634,0.39804067){\color[rgb]{0,0,0}\makebox(0,0)[lt]{\lineheight{1.25}\smash{\begin{tabular}[t]{l}\La $q^\beta_1$\end{tabular}}}}%
    \put(0.5851538,0.01478076){\color[rgb]{0,0,0}\makebox(0,0)[lt]{\lineheight{1.25}\smash{\begin{tabular}[t]{l}\La $q^\beta_4$\end{tabular}}}}%
    \put(0.69369055,0.01478076){\color[rgb]{0,0,0}\makebox(0,0)[lt]{\lineheight{1.25}\smash{\begin{tabular}[t]{l}\La $q^\beta_3$\end{tabular}}}}%
    \put(0.80222729,0.01478076){\color[rgb]{0,0,0}\makebox(0,0)[lt]{\lineheight{1.25}\smash{\begin{tabular}[t]{l}\La $q^\beta_2$\end{tabular}}}}%
    \put(0.91076403,0.01478076){\color[rgb]{0,0,0}\makebox(0,0)[lt]{\lineheight{1.25}\smash{\begin{tabular}[t]{l}\La $q^\beta_1$\end{tabular}}}}%
  \end{picture}%
\endgroup%
}
    \caption{The quads that can appear in $Q(h_\beta)$. Top: When $\beta$ is short. Bottom: When $\beta$ is long. We lay out the quads from left to right by the tetrahedra in which they lie from top to bottom in the stack. For each quad, the top left edge is the top edge of the tetrahedron it lies in, while the bottom right edge is the bottom edge of the tetrahedron it lies in. Moreover, we made a choice of orientation on some of the edges for ease of reference. We co-orient the quads as indicated whenever consistent.}
    \label{fig:fanquad}
\end{figure}

\begin{figure}[ht]
    \centering
    \resizebox{!}{6cm}{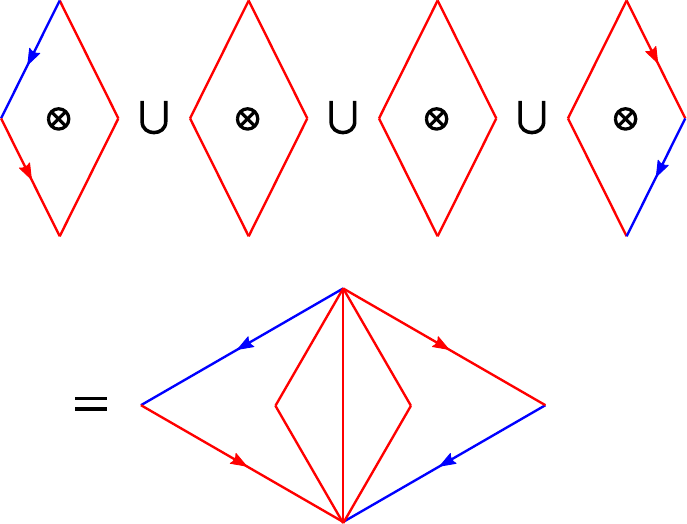}
    \caption{If we take the union $\bigcup_{i=1}^{\delta_\beta} q^\beta_i$, the quads fit together to give a big diamond with two blue edges and two red edges in the boundary. Here we illustrate the situation for \Cref{fig:fanquad} bottom.}
    \label{fig:fanquadunion}
\end{figure}

In the rest of this section, we will study the connectivity of $\Delta \cut \bigcup_{k \in J} q^\beta_k$ for various sets $J \subset \{1,...,\delta_\beta\}$. In this setting, each $q^\beta_k$, $k \in J$, will meet one or two components of $\Delta \cut \bigcup_{k \in J} q^\beta_k$, and at least one of the components it meets will contain the hook $h_\beta$. If some $q^\beta_k$ meets exactly two components, then we co-orient it to point out of the component containing $h_\beta$. On the other hand, if $q^\beta_k$ meets only one component, then there is no canonical way to co-orient it. More succinctly, this means that we co-orient the quads $q^\beta_k$, $k=1,...,\delta_\beta$, as indicated in \Cref{fig:fanquad} (that is, into the page) whenever consistent.

\begin{lemma} \label{lemma:singlenondeephook}
If $h_\beta$ is not deep, then $\Gamma(h_\beta)$ is connected.
\end{lemma}
\begin{proof}
Recall that the assumption means that $h_\beta$ does not contain $v^\beta_1$ in its interior.
Consider the 2-complex $Q(h_\beta)$ as in \Cref{subsec:Alexanderduality}. In our notation, we have $Q(h_\beta) = \bigcup_{i=k_\beta}^{\delta_\beta} q^\beta_i$ for some $k_\beta \geq 2$. In particular, notice that since $v^\beta_1$ is not resolved, there is at most one quad with a blue edge, namely the A-quad $q_{\delta_\beta}$ which has exactly one blue edge (or no such quad exists if $\beta$ is short).

This implies that if we have a 2-cycle of $Q(h_\beta)$ that is homologically trivial in $T_f$ and where all the coefficients are $\pm 1$, then such a cycle cannot contain $q_{\delta_\beta}$, for otherwise the boundary of the cycle will contain a single blue edge with no other edge to cancel it out. But then the 2-cycle will in fact be a 2-cycle of $Q(h'_\beta)$ that is homologically trivial in $T_f$, which contradicts \Cref{lemma:singleprehook} by \Cref{prop:Alexanderduality}.
\end{proof}

Thus if we have a minimum weight sector that has a non-deep hook, then we are already done. Such a scenario, however, is not always true, so we have to analyze how \Cref{lemma:singlenondeephook} fails when $h_\beta$ is deep.

\begin{lemma} \label{lemma:singledeephookfan}
If $h_\beta$ is deep, then $\Gamma(h_\beta)$ is connected unless $v^\beta_{\delta_\beta}$ is equal to the bottom vertex of $s^\beta_1$, in the manner such that the bottom side of $s^\beta_1$ containing $e^\beta_{\delta_\beta+1}$ meets $v_0$ as well. See \Cref{fig:tbt} left.

In particular if $s$ is fan, then $v^\beta_{\delta_\beta}=v_0$, in the manner such that $(e^\beta_{\delta_\beta+1},e^\beta_1)$ takes an anti-branching turn at $v_0$. See \Cref{fig:tbt} right.
\end{lemma}

\begin{figure}[ht]
    \centering
    \fontsize{8pt}{8pt}\selectfont    
    \resizebox{!}{4.5cm}{
\begingroup%
  \makeatletter%
  \providecommand\color[2][]{%
    \errmessage{(Inkscape) Color is used for the text in Inkscape, but the package 'color.sty' is not loaded}%
    \renewcommand\color[2][]{}%
  }%
  \providecommand\transparent[1]{%
    \errmessage{(Inkscape) Transparency is used (non-zero) for the text in Inkscape, but the package 'transparent.sty' is not loaded}%
    \renewcommand\transparent[1]{}%
  }%
  \providecommand\rotatebox[2]{#2}%
  \newcommand*\fsize{\dimexpr\f@size pt\relax}%
  \newcommand*\lineheight[1]{\fontsize{\fsize}{#1\fsize}\selectfont}%
  \ifx\svgwidth\undefined%
    \setlength{\unitlength}{371.83383251bp}%
    \ifx\svgscale\undefined%
      \relax%
    \else%
      \setlength{\unitlength}{\unitlength * \real{\svgscale}}%
    \fi%
  \else%
    \setlength{\unitlength}{\svgwidth}%
  \fi%
  \global\let\svgwidth\undefined%
  \global\let\svgscale\undefined%
  \makeatother%
  \begin{picture}(1,0.43661924)%
    \lineheight{1}%
    \setlength\tabcolsep{0pt}%
    \put(0.22732591,0.23952257){\color[rgb]{0,0,0}\makebox(0,0)[lt]{\lineheight{1.25}\smash{\begin{tabular}[t]{l}\Large $s$\end{tabular}}}}%
    \put(0,0){\includegraphics[width=\unitlength,page=1]{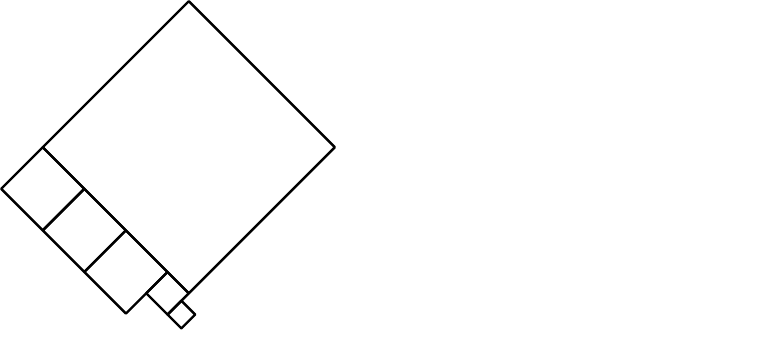}}%
    \put(0.7419068,0.23952257){\color[rgb]{0,0,0}\makebox(0,0)[lt]{\lineheight{1.25}\smash{\begin{tabular}[t]{l}\Large $s$\end{tabular}}}}%
    \put(0,0){\includegraphics[width=\unitlength,page=2]{tbt.pdf}}%
    \put(0.22871942,0.0252495){\color[rgb]{0,0,0}\makebox(0,0)[lt]{\lineheight{1.25}\smash{\begin{tabular}[t]{l}$s$\end{tabular}}}}%
    \put(0.75089435,0.02279611){\color[rgb]{0,0,0}\makebox(0,0)[lt]{\lineheight{1.25}\smash{\begin{tabular}[t]{l}\ns $s$\end{tabular}}}}%
    \put(0,0){\includegraphics[width=\unitlength,page=3]{tbt.pdf}}%
  \end{picture}%
\endgroup%
}
    \caption{$\Gamma(h_\beta)$ is connected unless the scenario depicted on the left occurs. If $s$ is fan, then the left figure specializes to the right figure, and we say that $s$ satisfies condition (TBT).}
    \label{fig:tbt}
\end{figure}

\begin{defn}
We say that a fan sector $s$ satisfies condition (TBT) (abbreviating \textbf{T}op = \textbf{B}ottom by \textbf{T}ranslation) if it satisfies the condition in the second paragraph of \Cref{lemma:singledeephookfan}.
\end{defn}

\begin{proof}[Proof of \Cref{lemma:singledeephookfan}]
Since $h_\beta$ is deep, the 2-complex $Q(h_\beta)=\bigcup_{i=1}^{\delta_\beta} q^\beta_i$. 
Suppose $\Gamma(h_\beta)$ is not connected. Consider the component containing $h_\beta$. This corresponds to some component $C$ of $\Delta \cut Q(h_\beta)$. Consider the 2-cycle $x$ of $Q(h_\beta)$ corresponding to $C$. Since we have co-oriented all the quads out of $C$ whenever possible, we can assume that all the coefficients in $x$ are 1. 

As opposed to the situation in \Cref{lemma:singlenondeephook}, if $\delta_\beta \geq 2$, we now have 2 quads that have a blue edge, namely the B-quad $q^\beta_1$ and the A-quad $q^\beta_{\delta_\beta}$, which each have exactly one blue edge, or if $\delta_\beta=1$, the single A-quad has exactly two blue edges.

The 2-cycle $x$ must contain both of the quads that have blue edges and the two blue edges must be the same edge in $\Delta$, for otherwise, as in \Cref{lemma:singlenondeephook}, there would be a contradiction to \Cref{lemma:singleprehook}. Moreover, the two blue edges have to be identified in such a way that they cancel each other out in the boundary of $x$.

Thus the first statement of the lemma follows from the fact that the 2 blue edges are dual to the sector with $v^\beta_{\delta_\beta}$ as the bottom vertex and $s^\beta_1$ respectively, and the manner of identification is as described in the lemma.

When $s$ is fan, $e^\beta_{\delta_\beta+1}$ will be a bottom side of $s^\beta_1$ by \Cref{prop:vbstogglefan}, hence the second statement follows.
\end{proof}

The intuition at this point is that $\Gamma(h_\beta)$ should be connected for `generic' sides $\beta$, since it is unlikely, for a `generic' triangulation, to have many edge identifications as described in \Cref{lemma:singledeephookfan}. However, notice that we require $\beta$ to be a side of a minimum weight sector in \Cref{prop:singlehookbound}, and at a first glance, minimum weight sectors might not seem `generic'. Fortunately, if we allow the fiber surface to vary, we have the following proposition that allows us to make an argument based on this intuition.

To simplify the language, let us call an arc on a sector \textit{deep} if both of its endpoints lie on $e^\beta_1$ for the two sides $\beta$, and let us call a sector \textit{deep} if it contains a deep arc. In other words, both hooks of a sector $s$ are deep if and only if $s$ is deep.

\begin{prop} \label{prop:manyfibersurfaces}
For any sector $s$, there exists a fiber surface $S$ such that $s$ is the only deep sector.
\end{prop}
\begin{proof}
Fix some fiber surface $S'$. Consider the $\mathbb{Z}$-cover $\widehat{T_f}$. Recall that $S'_r$ is the lift of $S'$ at height $r$. Let $\widehat{s_0}$ be the lift of $s$ that is at height 0. Let $R'$ be the union of sectors of $\widehat{B}$ that can be reached from $\widehat{s_0}$ via a path passing through $\leq -3\chi(S)$ sectors. $R'$ is a compact set, so there exists $r_1<r_2$ so that the region $R$ bounded by $S'_{r_1}$ and $S'_{r_2}$ contains $R'$.

Now starting at $\widehat{F_0}=S'_{r_2}$, inductively perform the following procedure: The image $F_i$ of $\widehat{F_i}$ in $T_f$ is some fiber surface. If some sector $t$ other than $s$ has a deep arc $a$, push $F_i$ downwards near $a$ through a vertex of $B$ to get another fiber surface $F_{i+1}$. In $\widehat{T_f}$, we have pushed $\widehat{F_i}$ across a vertex of $\widehat{B}$. Since the intersection of $\widehat{F_i}$ with $\widehat{B}$ is a train track with $-3 \chi(S)$ branches, and since we never push across $\widehat{s_0}$, $\widehat{F_i}$ must stay within $R$. Since there are finitely many vertices in $R$, this procedure ends eventually and we have the desired fiber surface.
\end{proof}

We pause for a moment to examine the progress we have made towards \Cref{prop:singlehookexist}:
For each sector $s$ of $B$, we take a fiber surface $S$ as in \Cref{prop:manyfibersurfaces} and see if there is a minimum weight sector that is not deep. If so, \Cref{lemma:singlenondeephook} and \Cref{prop:singlehookresolution} implies \Cref{prop:singlehookexist}. If not, that is, every minimum weight sector is deep, then $s$ is the only minimum weight sector. Consider the two hooks $h_\beta$ of $s$. If $s$ is fan, then either some $\Gamma(h_\beta)$ is connected, in which case \Cref{prop:singlehookresolution} implies \Cref{prop:singlehookexist}, or \Cref{lemma:singledeephookfan} shows that $s$ satisfies (TBT). By applying this argument to all fan sectors $s$, we see that either \Cref{prop:singlehookexist} is true, or all fan sectors satisfy (TBT). The next ingredient we need in the proof is the following proposition which concerns the toggle sectors in the latter case.

\begin{lemma} \label{lemma:singledeephooktoggle}
Suppose condition (TBT) is satisfied for all fan sectors. Suppose $s$ is a toggle sector and let $h_\beta$ be a hook of $s$. Then $\Gamma(h_\beta)$ is connected unless $e^\beta_{\delta_\beta+1}=e^\beta_1$. See \Cref{fig:sbf} left.
\end{lemma}

\begin{figure}[ht]
    \centering
    \fontsize{12pt}{12pt}\selectfont    
    \resizebox{!}{4cm}{
\begingroup%
  \makeatletter%
  \providecommand\color[2][]{%
    \errmessage{(Inkscape) Color is used for the text in Inkscape, but the package 'color.sty' is not loaded}%
    \renewcommand\color[2][]{}%
  }%
  \providecommand\transparent[1]{%
    \errmessage{(Inkscape) Transparency is used (non-zero) for the text in Inkscape, but the package 'transparent.sty' is not loaded}%
    \renewcommand\transparent[1]{}%
  }%
  \providecommand\rotatebox[2]{#2}%
  \newcommand*\fsize{\dimexpr\f@size pt\relax}%
  \newcommand*\lineheight[1]{\fontsize{\fsize}{#1\fsize}\selectfont}%
  \ifx\svgwidth\undefined%
    \setlength{\unitlength}{313.33245705bp}%
    \ifx\svgscale\undefined%
      \relax%
    \else%
      \setlength{\unitlength}{\unitlength * \real{\svgscale}}%
    \fi%
  \else%
    \setlength{\unitlength}{\svgwidth}%
  \fi%
  \global\let\svgwidth\undefined%
  \global\let\svgscale\undefined%
  \makeatother%
  \begin{picture}(1,0.37747226)%
    \lineheight{1}%
    \setlength\tabcolsep{0pt}%
    \put(0,0){\includegraphics[width=\unitlength,page=1]{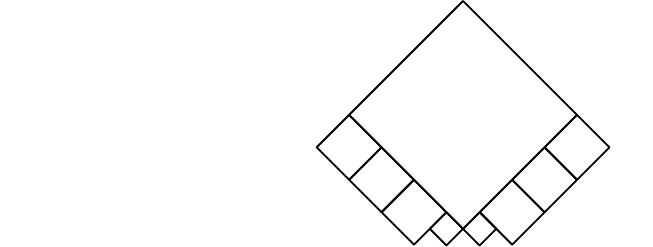}}%
    \put(0.21005702,0.18550792){\color[rgb]{0,0,0}\makebox(0,0)[lt]{\lineheight{1.25}\smash{\begin{tabular}[t]{l}\Large $s$\end{tabular}}}}%
    \put(0,0){\includegraphics[width=\unitlength,page=2]{sbf.pdf}}%
    \put(0.69371918,0.18551075){\color[rgb]{0,0,0}\makebox(0,0)[lt]{\lineheight{1.25}\smash{\begin{tabular}[t]{l}\Large $s$\end{tabular}}}}%
    \put(0,0){\includegraphics[width=\unitlength,page=3]{sbf.pdf}}%
    \put(0.19060149,0.01800013){\color[rgb]{0,0,0}\makebox(0,0)[lt]{\lineheight{1.25}\smash{\begin{tabular}[t]{l}$\s$\end{tabular}}}}%
    \put(0.6742484,0.01801829){\color[rgb]{0,0,0}\makebox(0,0)[lt]{\lineheight{1.25}\smash{\begin{tabular}[t]{l}$\s$\end{tabular}}}}%
    \put(0.72521294,0.01801835){\color[rgb]{0,0,0}\makebox(0,0)[lt]{\lineheight{1.25}\smash{\begin{tabular}[t]{l}$\s$\end{tabular}}}}%
  \end{picture}%
\endgroup%
}
    \caption{If condition (TBT) is satisfied for all fan sectors, $\Gamma(h_\beta)$ is connected unless the scenario depicted on the left occurs, in which case we say that $s$ satisfies (SBF) on the side $\beta$. If $s$ satisfies (SBF) on both sides, then we have the right figure, and we say that $s$ satisfies the condition (BSBF).}
    \label{fig:sbf}
\end{figure}

\begin{defn}
We say that a sector $s$ satisfies condition (SBF) (abbreviating \textbf{S}ide = \textbf{B}ottom by \textbf{F}lip) on the side $\beta$ if it satisfies the condition in \Cref{lemma:singledeephooktoggle}.

We say that a sector $s$ satisfies condition (BSBF) (abbreviating \textbf{B}oth \textbf{S}ides = \textbf{B}ottom by \textbf{F}lips) if it satisfies (SBF) on both sides. See \Cref{fig:sbf} right.
\end{defn}

\begin{proof}[Proof of \Cref{lemma:singledeephooktoggle}]
Suppose $\Gamma(h_\beta)$ is not connected. Consider the component containing $h_\beta$. This corresponds to some component $C$ of $\Delta \cut Q(h_\beta)$, which in turn corresponds under \Cref{prop:Alexanderduality} to some 2-cycle $x$ of $Q(h_\beta)$.

Recall the construction of $x$. We look at the union of quads in the boundary of $C$ in $\Delta$ and take their sum. Since $C$ contains $h_\beta$, every quad $q^\beta_k$ appears in the union. After taking the sum, the quad $q^\beta_k$ does not appear in $x$ if and only if it appears twice in the union, necessarily with opposite co-orientations. There is a dichotomy here. A quad $q^\beta_k$ could appear twice because $q^\beta_k = q^\beta_{k'}$ for some $k' \neq k$, or it could appear twice even if it is distinct from all other $q^\beta_{k'}$. In the former case, we will say that $q^\beta_k$ is \textit{overlapped}, while in the latter case, we will say that $q^\beta_k$ is \textit{omitted}. 

We claim that under the assumption of the lemma, no $q^\beta_k$ is omitted.
First notice that as reasoned in the proof of \Cref{lemma:singledeephookfan}, $q^\beta_1$ and $q^\beta_{\delta_\beta}$ must appear in $x$, so they are not omitted. We show that $q^\beta_k$, for $k=2,...,\delta_\beta-1$, are not omitted by downwards induction. 

Consider the sector $s^\beta_{\delta_\beta-1}$. It is fan, so by assumption it satisfies (TBT). Hence it is also the sector having $v^\beta_{\delta_\beta-1}$ as its bottom vertex, and contains $v^\beta_{\delta_\beta}$ on its bottom side that does not meet $s^\beta_{\delta_\beta-2}$. Now if $q^\beta_{\delta_\beta-1}$ is omitted, the bottom vertex of $s^\beta_{\delta_\beta}$ must be resolved, otherwise by following along the sides of $s^\beta_{\delta_\beta-1}$, we see that $v^\beta_{\delta_\beta}$ does not meet any component of $\Gamma(h_\beta)$ other than $C$, hence $q^\beta_{\delta_\beta}$ must be omitted in $x$, contradicting what we established above. So we suppose that this bottom vertex is equal to $v^\beta_j$ for some $j$. This identification cannot be so that $v^\beta_{\delta_\beta-1}=v^\beta_{j+1}$, for otherwise since $v^\beta_{\delta_\beta-1}$ is red, we must have $j \leq \delta_\beta-2$, and $q^\beta_{\delta_\beta-1}$ would be overlapped (by $q^\beta_{j+1}$), not omitted. But in the other manner of identification we also have that $v^\beta_{\delta_\beta}$ does not meet any component of $\Gamma(h_\beta)$ other than $C$, giving us a contradiction as above.

Inductively, suppose we have shown that $q^\beta_k$ is not omitted for some $3 \leq k \leq \delta_\beta-1$. Consider the sector $s^\beta_k$. Since it satisfies (TBT), both its top and bottom vertices are $v^\beta_k$. If $v^\beta_{k-1}$ is omitted, then by following along the sides of $s^\beta_k$ containing $v^\beta_{k-1}$, we see that $v^\beta_k$ does not meet any component of $\Gamma(h_\beta)$ other than $C$. So $q^\beta_k$ must be overlapped, say $q^\beta_k=q^\beta_j$. But then since $v^\beta_{k-1}$ is red, we must have $j \leq \delta_\beta-2$, and $q^\beta_{k-1}$ would be overlapped (by $q^\beta_{j+1}$), not omitted. By induction this proves our claim that no $q^\beta_k$ is omitted.

The upshot here is that the edges in the boundary of the union $\bigcup_{k=1}^{\delta_\beta} q^\beta_k$ must cancel themselves out. This is because the same is true for the 2-cycle $x$, and the edges in the boundary of quads not appearing in $x$ cancel out in pairs of overlapped quads. 

As pointed out before \Cref{lemma:singlenondeephook}, the boundary of $\bigcup_{k=1}^{\delta_\beta} q^\beta_k$ consists of 4 edges, two of which are blue and two are red. The blue edges cancelling each other out is the content of \Cref{lemma:singledeephookfan}. Now we also know that the two red edges cancel each other out, which by \Cref{prop:vbstogglefan} and the fact that $s$ is toggle, implies the current lemma.
\end{proof}

Returning to our proof of \Cref{prop:singlehookexist}, which was suspended before \Cref{lemma:singledeephooktoggle}, by now taking $s$ to be each toggle sector, we see that either some $\Gamma(h_\beta)$ is connected, in which case \Cref{prop:singlehookresolution} implies \Cref{prop:singlehookexist}, or all toggle sectors $s$ satisfies (BSBF). Hence our proof of \Cref{prop:singlehookexist} is finally concluded by the following proposition.

\begin{prop} \label{prop:m003}
Let $\Delta$ be a veering triangulation and $B$ be its stable branched surface. If every fan sector of $B$ satisfies (TBT) and every toggle sector of $B$ satisfies (BSBF), then $\Delta = \texttt{cPcbbbdxm\_10}$.
\end{prop}
\begin{proof}
Let $t$ be a toggle tetrahedron in $\Delta$. Suppose the top edge of $t$ is blue. Then the vertex $v$ of $B$ dual to $t$ is the bottom vertex of some blue sector $s$. $s$ must be toggle, otherwise since every fan sector satisfies (TBT), the sector below $s$ must be fan, hence blue, contradicting the fact that the bottom edge of $t$ is red. 

Hence $s$ satisfies (BSBF), so the two blue side edges of $t$ must be equal (and in fact equals the top edge of $t$) and be identified in a parallel way.

Meanwhile, let $s'$ be the sector that has $v$ as the top vertex. Similarly as above, since $s'$ satisfies (BSBF), the two red side edges of $t$ must be equal (and in fact equals to the bottom edge of $t$) and be identified in a parallel way.

Now consider a quadrilateral-with-4-ideal-vertices properly embedded in $t$ with its 4 sides along the side edges of $t$. By our reasoning above, the sides match up in pairs to form a once-punctured torus $T$ in $T_f$. The quadrilateral can be chosen so that the intersection of $B$ with $T$ is a train track $\tau$ of the form illustrated in \Cref{fig:m003}. In particular, using the language of \cite{Lan19}, no splitting of $\tau$ can contain a \textit{stable loop}. Thus by \cite[Proposition 4.4 and Proposition 4.6]{Lan19}, $T$ is a fiber surface. 

\begin{figure}[ht]
    \centering
    \resizebox{!}{2.5cm}{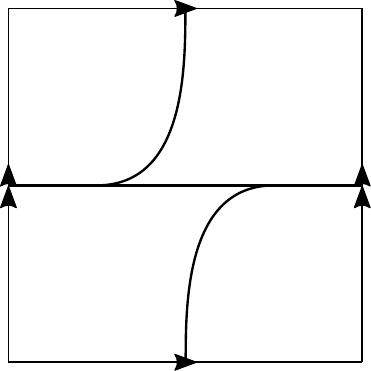}
    \caption{If every fan sector satisfies (TBT) and every toggle sector satisfies (BSBF), then we can find an equatorial quadrilateral that meets $B$ in the depicted train track.}
    \label{fig:m003}
\end{figure}

Once-punctured torus bundles and their veering triangulations are well-studied. These triangulations are encoded by a periodic path in the Farey tessellation, with each edge corresponding to a turn in the path; blue edges for left turns and red edges for right turns. We refer to \cite{Gue06} for details. In particular, it can be checked that no fan sector of these can satisfy (TBT). Moreover, (BSBF) for a single toggle sector implies that the path in the Farey tessellation has period 2, and that the triangulation is \texttt{cPcbbbdxm\_10}.
\end{proof}

\begin{rmk} \label{rmk:m003}
\Cref{prop:singlehookexist} is sharp in the sense that when $\Delta=$ \texttt{cPcbbbdxm\_10}, $\Gamma(h_\beta)$ is disconnected for every deep hook $h_\beta$.
\end{rmk}

\Cref{prop:singlehookbound} and \Cref{prop:singlehookexist} imply \Cref{thm:introsinglehook}, which we restate as follows.

\begin{thm} \label{thm:singlehook}
Let $f:S \to S$ be a fully-punctured pseudo-Anosov map with normalized dilatation $\lambda^{-\chi}$. Then the mapping torus of $f$ admits a veering triangulation with less than or equal to $\frac{1}{2} \lambda^{-2\chi}$ tetrahedra.
\end{thm}
\begin{proof}
If $\Delta \neq$ \texttt{cPcbbbdxm\_10}, then by \Cref{prop:singlehookexist}, there exists an Eulerian circuit that hooks around a sector of minimum weight, thus the theorem follows from \Cref{prop:singlehookbound}.

If $\Delta=$ \texttt{cPcbbbdxm\_10}, then $T_f$ fibers in a unique way as a once-punctured torus bundle. In this case, it is straightforward to calculate that $\lambda=\mu^2$ and $\chi=-1$, so we can check that the number of tetrahedra, which is 2, is less than $\frac{1}{2} \lambda^{-2\chi} = \frac{1}{2} \mu^4 \approx 3.43$.
\end{proof}

For non-fully-punctured pseudo-Anosov maps, we have the following corollary.

\begin{cor} \label{cor:singlehook}
Let $f:S_{g,s} \to S_{g,s}$ be a pseudo-Anosov map on the surface $S_{g,s}$ with genus $g$ and $s$ punctures with dilatation $\lambda$. Suppose $\lambda^{2g-2+\frac{2}{3}s} \leq P$, then the fully-punctured mapping torus of $f$ admits a veering triangulation with less than or equal to $\frac{1}{2}P^6$ tetrahedra.
\end{cor}
\begin{proof}
We can fully puncture $f$ at its singularities to get $f^\circ:S^\circ \to S^\circ$. Since each singularity is at least 3-pronged, and each puncture on $S_{g,n}$ is at least 1-pronged, we at most puncture at $2(-\chi(S_{g,n})-\frac{n}{2})=4g-4+n$ points. Hence $\chi(S^\circ) \leq -\chi(S_{g,n})+(4g-4+n)=3(2g-2+\frac{2}{3}n)$ and the corollary follows from \Cref{thm:singlehook}.
\end{proof}

\section{Double hook circuits} \label{sec:doublehook}

As explained in the introduction, we must further improve the bound of \Cref{thm:introsinglehook} (at least in the one boundary component case) in order to prove \Cref{thm:introdilthm}. In this section we present one possible approach for obtaining an improvement. The idea is to use two hook circuits or a circuit that hooks twice. In contrast with using single hook circuits in \Cref{sec:singlehook}, we refer to this approach as using double hook circuits. The resulting improved bound is recorded as \Cref{prop:doublehookbound}, which we will explain in \Cref{subsec:doublehookbound}.

Unfortunately, we are unable to show that such double hook circuits always exist. In \Cref{subsec:doublehookobstruct}, we describe the obstructions to their existence. This understanding of obstructions will be used in \Cref{sec:oneboundary} to show that when $T_f$ has only one boundary component, we can sometimes bypass the obstructions and use the sharper bound in \Cref{prop:doublehookbound}.

\subsection{Bounding the number of vertices using double hook circuits} \label{subsec:doublehookbound}

We continue the notation from \Cref{sec:singlehook}. In addition, we introduce two more pieces of terminology.

\begin{defn} \label{defn:Euleriancollection}
Let $\{c_1,...,c_n\}$ be a collection of circuits in a directed graph. The collection is said to be \textit{Eulerian} if every edge is traversed exactly once by some $c_i$.
\end{defn}

\begin{defn} \label{defn:doublehook}
Let $s$ be a sector of $B$. A circuit of $\Gamma$ is said to \textit{hook around $s$ twice} if it contains both hooks of $s$. If the sector $s$ is understood, we say that a vertex of $c$ is a hook vertex if it lies in the interior of $h_1$ or $h_2$.
\end{defn}

\begin{prop} \label{prop:doublehookbound}
Let $s$ be a sector of $B$ of minimum weight. Suppose one of the following two scenarios is true:
\begin{enumerate}[label=(\roman*)]
    \item There exists an Eulerian collection of two circuits $\{c_1,c_2\}$ in $\Gamma$ where both $c_1$ and $c_2$ hook around $s$
    \item There exists an Eulerian circuit $c$ in $\Gamma$ that hooks around $s$ twice
\end{enumerate}
Then the number of tetrahedra in the veering triangulation is $\leq \frac{1}{4} \lambda^{-2\chi} + 1$.
\end{prop}
\begin{proof}
We first show the bound in case (i). The argument reuses a lot of the ideas from \Cref{prop:singlehookbound}. 

Suppose $c_i$ contains the hook $h_i$ of $s$, for $i=1,2$. Let $\widehat{s_0}$ be the lift of $s$ in $\widehat{B}$ that is of height 0. For each $i$, take the basepoint of $c_i$ to be the top vertex of $s$. Lift $c_i$ to a path $\widehat{c_i}$ ending at the top vertex of $\widehat{s_0}$, reverse its orientation, then push it downwards to get a descending path $\alpha_i$.

Let the length of $c_i$ be $n_i$, and suppose it intersects the fiber surface $S$ for $p_i$ times. Then since $\{c_1,c_2\}$ is Eulerian, $n_1+n_2$ is equal to two times the number of tetrahedra $N$ in the triangulation, while $p_1+p_2$ is equal to $-2\chi(S)$. 

As in the proof of \Cref{prop:singlehookbound}, we get two equations
\begin{equation} \label{eq:doublehook}
    \lambda^{p_i} w = w + \sum_{j=1}^{n_i} w_{i,j}
\end{equation}
where $w_{i,j}$ is the weight of the $j^{\text{th}}$ sector merging into $\alpha_i$.
We again say that a vertex $v$ in $\Gamma$ is \textit{A-resolved} if the $c_i$ take an anti-branching turn the two times they visit $v$, and say that $v$ is \textit{B-resolved} otherwise. 

As in \Cref{prop:singlehookbound}, for each term in $W = \sum_{j=1}^{n_1} w_{1,j} + \sum_{j=1}^{n_2} w_{2,j}$ corresponding to a vertex $v$ of $c$ not covering a hook vertex of $c_1$ or $c_2$, we can use the lower bound $\lambda^{-r_v} w \geq w$ where $r_v$ is the height of $v$. 
For each B-resolved hook vertex $u$, we group together the two terms in $W$ that correspond to the two vertices of $\widehat{c}$ covering the vertex of $\Gamma$ that $u$ lies at. The same argument shows that the sum of the pair of terms is bounded below by $\lambda^{-r_1+d_u}w + \lambda^{-r_2-d_u}w \geq 2w$. 

In this setting, there are two A-resolved hook vertices, so the argument implies that $W \geq (2N-2)w$. Meanwhile, we can multiply \Cref{eq:doublehook} for $i=1,2$ to get $$\lambda^{-2\chi} = (1+\frac{1}{w} \sum_{j=1}^{n_1} w_{1,j}) (1+\frac{1}{w} \sum_{j=1}^{n_2} w_{2,j})$$
where $(1+\frac{1}{w} \sum_{j=1}^{n_1} w_{1,j}) + (1+\frac{1}{w} \sum_{j=1}^{n_2} w_{2,j}) = 2+\frac{1}{w}W \geq 2N$.

Here is where we bring in a new ingredient: We claim that $\frac{1}{w} \sum_{j=1}^{n_i} w_{i,j} \geq 1$ for each $i$. If the hook $h_i$ is simple, that is, the degenerate case in \Cref{defn:hookcircuit} does not occur, then this is because the term corresponding to the first non-hook vertex is the weight of a sector at non-positive height, hence is $\geq w$. If $h_i$ is not simple, then observe that the starting point and ending point of $\alpha_i$ belong to adjacent sectors. We can take a descending path from the former to the latter that intersects the branch locus of $\widehat{B}$ exactly once, and the weight added along that intersection is the weight of a sector at non-positive height, hence is $\geq w$. See \Cref{fig:doublehookbound1} where we highlight this descending path in orange. Hence $\sum_{j=1}^{n_i} w_{i,j} = \lambda^{p_i}w - w \geq w$.

\begin{figure}[ht]
    \centering
    \fontsize{12pt}{12pt}\selectfont
    \resizebox{!}{4.5cm}{
\begingroup%
  \makeatletter%
  \providecommand\color[2][]{%
    \errmessage{(Inkscape) Color is used for the text in Inkscape, but the package 'color.sty' is not loaded}%
    \renewcommand\color[2][]{}%
  }%
  \providecommand\transparent[1]{%
    \errmessage{(Inkscape) Transparency is used (non-zero) for the text in Inkscape, but the package 'transparent.sty' is not loaded}%
    \renewcommand\transparent[1]{}%
  }%
  \providecommand\rotatebox[2]{#2}%
  \newcommand*\fsize{\dimexpr\f@size pt\relax}%
  \newcommand*\lineheight[1]{\fontsize{\fsize}{#1\fsize}\selectfont}%
  \ifx\svgwidth\undefined%
    \setlength{\unitlength}{119.98707448bp}%
    \ifx\svgscale\undefined%
      \relax%
    \else%
      \setlength{\unitlength}{\unitlength * \real{\svgscale}}%
    \fi%
  \else%
    \setlength{\unitlength}{\svgwidth}%
  \fi%
  \global\let\svgwidth\undefined%
  \global\let\svgscale\undefined%
  \makeatother%
  \begin{picture}(1,1.19230921)%
    \lineheight{1}%
    \setlength\tabcolsep{0pt}%
    \put(0,0){\includegraphics[width=\unitlength,page=1]{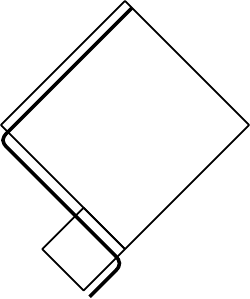}}%
    \put(0.44604661,0.65538456){\color[rgb]{0,0,0}\makebox(0,0)[lt]{\lineheight{1.25}\smash{\begin{tabular}[t]{l}\La $s$\end{tabular}}}}%
    \put(0.3133551,0.1775091){\color[rgb]{0,0,0}\makebox(0,0)[lt]{\lineheight{1.25}\smash{\begin{tabular}[t]{l}$\s$\end{tabular}}}}%
    \put(0,0){\includegraphics[width=\unitlength,page=2]{doublehookbound1.pdf}}%
  \end{picture}%
\endgroup%
}
    \caption{If $h_i$ is not simple, we can use a descending path (in orange) to show that $\frac{1}{w} \sum_{j=1}^{n_i} w_{i,j} \geq 1$.}
    \label{fig:doublehookbound1}
\end{figure}

The claim implies that $$\lambda^{-2\chi} = (1+\frac{1}{w} \sum_{j=1}^{n_1} w_{1,j}) (1+\frac{1}{w} \sum_{j=1}^{n_2} w_{2,j}) \geq 2(2N-2)$$ which implies the bound in the proposition.

Now for case (ii), if the top vertex of $s$ is not a hook vertex of $c$, then we can cut and paste $c$ at that point to get an Eulerian collection of circuits $\{c_1,c_2\}$, where each $c_i$ hooks around $s$, so we reduce to case (i). 

If the top vertex of $s$ is a hook vertex, say it lies in the interior of $h_1$, then we can do the cut and paste operation anyway to get an Eulerian collection of circuits $\{c_1,c_2\}$. See \Cref{fig:doublehookbound2}. The difference here is that the component containing the vertices in the interior of $h_1$, say $c_1$, may no longer contain the entirety of $h_1$ hence not hook around $s$ by definition. However the above argument still works since the only place we use the fact that $c_1$ contains the entirety of $h_1$ is to say that the terms $w_{1,j}$ corresponding to vertices of $\widehat{c_1}$ not covering hook vertices are weights of sectors of non-positive height (hence are $\geq w$). In this case there are simply no such vertices and so we do not have to worry about establishing the lower bounds for them. Thus the argument in case (i) carries through to show the same bound.
\end{proof}

\begin{figure}[ht]
    \centering
    \fontsize{12pt}{12pt}\selectfont
    \resizebox{!}{4.5cm}{
\begingroup%
  \makeatletter%
  \providecommand\color[2][]{%
    \errmessage{(Inkscape) Color is used for the text in Inkscape, but the package 'color.sty' is not loaded}%
    \renewcommand\color[2][]{}%
  }%
  \providecommand\transparent[1]{%
    \errmessage{(Inkscape) Transparency is used (non-zero) for the text in Inkscape, but the package 'transparent.sty' is not loaded}%
    \renewcommand\transparent[1]{}%
  }%
  \providecommand\rotatebox[2]{#2}%
  \newcommand*\fsize{\dimexpr\f@size pt\relax}%
  \newcommand*\lineheight[1]{\fontsize{\fsize}{#1\fsize}\selectfont}%
  \ifx\svgwidth\undefined%
    \setlength{\unitlength}{119.00710699bp}%
    \ifx\svgscale\undefined%
      \relax%
    \else%
      \setlength{\unitlength}{\unitlength * \real{\svgscale}}%
    \fi%
  \else%
    \setlength{\unitlength}{\svgwidth}%
  \fi%
  \global\let\svgwidth\undefined%
  \global\let\svgscale\undefined%
  \makeatother%
  \begin{picture}(1,1.19511825)%
    \lineheight{1}%
    \setlength\tabcolsep{0pt}%
    \put(0,0){\includegraphics[width=\unitlength,page=1]{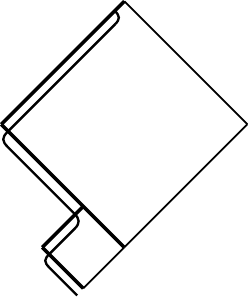}}%
    \put(0.45843079,0.67492285){\color[rgb]{0,0,0}\makebox(0,0)[lt]{\lineheight{1.25}\smash{\begin{tabular}[t]{l}\La $s$\end{tabular}}}}%
    \put(0.30668107,0.18350811){\color[rgb]{0,0,0}\makebox(0,0)[lt]{\lineheight{1.25}\smash{\begin{tabular}[t]{l}$s$\end{tabular}}}}%
  \end{picture}%
\endgroup%
}
    \caption{In case (ii), if $e^2_{\delta_2+1}=e^1_1$, then we can cut and paste $c$ to get an Eulerian collection of circuits $\{c_1,c_2\}$. $c_1$ will not contain the hook $h_1$ (thickened) but the argument goes through.}
    \label{fig:doublehookbound2}
\end{figure}

\subsection{Obstructions to double hook circuits} \label{subsec:doublehookobstruct}

We have the following analogue of \Cref{prop:singlehookresolution} for applying \Cref{prop:doublehookbound}.

\begin{prop} \label{prop:doublehookresolution}
Let $s$ be a sector of $B$. Suppose one of the following two statements about the resolved dual graph $\Gamma(h_1 \cup h_2)$ is true:
\begin{enumerate}[label=(\Roman*)]
    \item $\Gamma(h_1 \cup h_2)$ is connected
    \item $\Gamma(h_1 \cup h_2)$ has two components, with one component containing $h_1$ and the other containing $h_2$
\end{enumerate}
Then the hypothesis of \Cref{prop:doublehookbound} is true.
\end{prop}
\begin{proof}
In case (I), we can find an Eulerian circuit in $\Gamma(h_1 \cup h_2)$ whose image in $\Gamma$ is an Eulerian circuit that hooks around $s$ twice, so (ii) is true.
In case (II), we can find an Eulerian circuit in each component of $\Gamma(h_1 \cup h_2)$ whose images in $\Gamma$ form an Eulerian collection of circuits that each hook around $s$, so (i) is true.
\end{proof}

For the rest of this section, we will state and prove some conditions on $s$ under which we can show that (I) or (II) in \Cref{prop:doublehookresolution} is true. Some of these will be analogues of statements in \Cref{subsec:singlehookexist}, and our approach in fact will closely mirror that in \Cref{subsec:singlehookexist}. 

We fix a sector $s$ of $B$. Suppose without loss of generality that $s$ is blue. We will use the same notation as in \Cref{subsec:singlehookexist} for the edges, sectors, and vertices adjacent to $s$.
Recall that $h'_\beta = (e^\beta_2,...,e^\beta_{\delta_\beta})$. Consider the resolution $\Gamma(h'_1 \cup h'_2)$ determined by the two paths $h'_1$ and $h'_2$. As in \Cref{lemma:singleprehook}, all the resolutions are B-resolutions of red vertices.

\begin{lemma} \label{lemma:doubleprehook}
$\Gamma(h'_1 \cup h'_2)$ is either connected or has two components. In the latter case, $v^1_{\delta_1}=v^2_{\delta_2}$ and the component not containing $h'_1 \cup h'_2$ contains at least one branch cycle.
\end{lemma}

\begin{defn} \label{defn:frc}
We say that a sector $s$ satisfies condition (FRC) (abbreviating \textbf{F}an \textbf{R}esolution \textbf{C}onnected) if the resolution $\Gamma(h'_1 \cup h'_2)$ is connected for $s$.
\end{defn}

\begin{proof}[Proof of \Cref{lemma:doubleprehook}]
If some $\delta_\beta=1$, then this reduces to \Cref{lemma:singlenondeephook}, so we can assume that $\delta_1,\delta_2 \geq 2$. 
Suppose that $\Gamma(h'_1 \cup h'_2)$ is not connected. By following the bottom sides of $s$, we see that $h'_1$ and $h'_2$ lie in the same component $C_0$. Let $C_1$ be a component that is not $C_0$. 
As in \Cref{lemma:singleprehook}, let $J$ be the collection of $(\beta,k)$ such that $v^\beta_k$ meets $C_1$.

Let $v^\beta$ be the side vertex of $s^\beta_{\delta_\beta}$ other than $v^\beta_{\delta_\beta-1}$. The same argument as in \Cref{lemma:singleprehook} shows that at least one of $v^\beta$, say, $v^1$, must be resolved and meets $C_1$, otherwise a permutation on $J$ would give rise to a branch cycle only meeting red vertices, contradicting \Cref{prop:vbscontradictions}.
Suppose $v^1 = v^\sigma_k$. The same argument as in \Cref{lemma:singleprehook} shows that, since $v^1_{\delta_1}$ is blue, $k$ must be $\delta_{\sigma}-1$. Thus $\sigma \neq 1$ otherwise $s^1_{\delta_1}$ is not embedded in its interior. This shows that $s^1_{\delta_1}=s^2_{\delta_2}$ and $v^1_{\delta_1}=v^2_{\delta_2}$. 

If $v^2$ only meets $C_0$ and no other component of $\Gamma(h'_1 \cup h'_2)$, then $C_1$ is uniquely determined to be the component which $v^1$ meets other than $C_0$. Since we picked $C_1$ arbitrarily at the start, this shows that $\Gamma(h'_1 \cup h'_2)$ has exactly two components in this case.

If $v^2$ meets some component $C_2 \neq C_0$, then the bottom vertex of $s^1_{\delta_1}=s^2_{\delta_2}$ does not meet $C_0$, hence is not resolved. Thus $C_1=C_2$, and $C_1$ is again uniquely determined to be the component which $v^1$ meets other than $C_0$, showing that $\Gamma(h'_1 \cup h'_2)$ has exactly two components.

Since the only resolutions we perform are B-resolutions, any branch cycle either lies completely in $C_0$ or completely in $C_1$, so $C_1$ contains at least one branch cycle.
\end{proof}

In the one boundary component case, we will often be able to bypass the obstruction described in \Cref{lemma:doubleprehook} and have (FRC). With this in mind, for the rest of this section, we prove some lemmas providing some sets of conditions, which when taken together with (FRC), guarantee that the assumptions of \Cref{prop:doublehookresolution} are satisfied.

To prove these lemmas, we use the A/B-quads as described in \Cref{subsec:Alexanderduality} and \Cref{subsec:singlehookexist}. We continue the notation from there. Since we consider both sides of $s$ now, we have two groups of quads, one group for each side of $s$ as laid out in \Cref{fig:fanquad}.

We have the following analogue of \Cref{lemma:singledeephookfan}.

\begin{lemma} \label{lemma:doubledeephookfan}
Suppose $s$ satisfies (FRC), then $\Gamma(h_1 \cup h_2)$ is connected unless some $h_\beta$ is deep and $v^{\beta'}_{\delta_{\beta'}}$ is equal to the bottom vertex of $s^\beta_1$ for some $\beta'$, in the manner such that the bottom side of $s^\beta_1$ containing $e^{\beta'}_{\delta_{\beta'}+1}$ meets $v_0$ as well. See \Cref{fig:doubledeephookfan}.

In particular if $s$ is fan, then the top vertex of $s$ is equal to the bottom vertex of $s$.
\end{lemma}

\begin{figure}[ht]
    \centering
    \fontsize{12pt}{12pt}\selectfont
    \resizebox{!}{4.5cm}{
\begingroup%
  \makeatletter%
  \providecommand\color[2][]{%
    \errmessage{(Inkscape) Color is used for the text in Inkscape, but the package 'color.sty' is not loaded}%
    \renewcommand\color[2][]{}%
  }%
  \providecommand\transparent[1]{%
    \errmessage{(Inkscape) Transparency is used (non-zero) for the text in Inkscape, but the package 'transparent.sty' is not loaded}%
    \renewcommand\transparent[1]{}%
  }%
  \providecommand\rotatebox[2]{#2}%
  \newcommand*\fsize{\dimexpr\f@size pt\relax}%
  \newcommand*\lineheight[1]{\fontsize{\fsize}{#1\fsize}\selectfont}%
  \ifx\svgwidth\undefined%
    \setlength{\unitlength}{234.45150781bp}%
    \ifx\svgscale\undefined%
      \relax%
    \else%
      \setlength{\unitlength}{\unitlength * \real{\svgscale}}%
    \fi%
  \else%
    \setlength{\unitlength}{\svgwidth}%
  \fi%
  \global\let\svgwidth\undefined%
  \global\let\svgscale\undefined%
  \makeatother%
  \begin{picture}(1,0.60590819)%
    \lineheight{1}%
    \setlength\tabcolsep{0pt}%
    \put(0,0){\includegraphics[width=\unitlength,page=1]{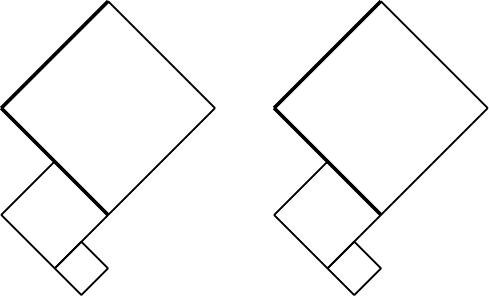}}%
    \put(0.2029977,0.3671753){\color[rgb]{0,0,0}\makebox(0,0)[lt]{\lineheight{1.25}\smash{\begin{tabular}[t]{l}\La $s$\end{tabular}}}}%
    \put(0.76167421,0.3671753){\color[rgb]{0,0,0}\makebox(0,0)[lt]{\lineheight{1.25}\smash{\begin{tabular}[t]{l}\La $s$\end{tabular}}}}%
    \put(0.15154007,0.04094201){\color[rgb]{0,0,0}\makebox(0,0)[lt]{\lineheight{1.25}\smash{\begin{tabular}[t]{l}$s$\end{tabular}}}}%
    \put(0.71021654,0.04094201){\color[rgb]{0,0,0}\makebox(0,0)[lt]{\lineheight{1.25}\smash{\begin{tabular}[t]{l}$\s$\end{tabular}}}}%
  \end{picture}%
\endgroup%
}
    \caption{If $s$ satisfies (FRC) but $\Gamma(h_1 \cup h_2)$ is not connected, then some $h_\beta$ (thickened) is deep and $v^{\beta'}_{\delta_{\beta'}}$ is equal to the bottom vertex of $s^\beta_1$ for some $\beta'$, in a manner such that the bottom side of $s^\beta_1$ containing $e^{\beta'}_{\delta_{\beta'}+1}$ meets $v_0$ as well. In these examples $\beta$ is the left side, and $\beta'$ is the left/right side in the left/right figure respectively.}
    \label{fig:doubledeephookfan}
\end{figure}

\begin{proof}
The 2-complex $Q(h_1 \cup h_2)$ corresponding to the resolution is $\bigcup_{i=k_1}^{\delta_1} q^1_i \cup \bigcup_{i=k_2}^{\delta_2} q^2_i$. There are at most four blue sides in $Q(h_1 \cup h_2)$, belonging to at most four quads, namely, $q^1_1$, $q^1_{\delta_1}$, $q^2_1$, and $q^2_{\delta_2}$. For simplicity let us assume that $\delta_1,\delta_2 \geq 2$, so that these four quads are distinct. The same argument will work for $\delta_1$ or $\delta_2=1$ with some more careful wording. We let the reader fill this out themselves.

Since we assumed that $s$ satisfies (FRC), if $\Gamma(h_1 \cup h_2)$ is not connected, then the 2-cycle $x$ of $Q(h_1 \cup h_2)$ corresponding to a component must involve at least two of these four quads and identify two of their blue edges. 

If the blue edge of $q^1_{\delta_1}$ is identified with that of $q^2_{\delta_2}$, then $q^1_{\delta_1}$ and $q^2_{\delta_2}$ must be the same quad. Hence the 2-cycle $x$ needs to contain at least one of $q^1_1$ and $q^2_1$ and identify one of their blue edges to the blue edge of $q^1_{\delta_1}=q^2_{\delta_2}$ in order to cancel it out.

Similarly, if the blue edge of $q^1_1$ is identified with that of $q^2_1$, then $q^1_1$ and $q^2_1$ must be the same quad. Hence the 2-cycle $x$ needs to contain at least one of $q^1_{\delta_1}$ and $q^2_{\delta_2}$ and identify one of their blue edges to the blue edge of $q^1_1=q^2_1$.

Hence we conclude that the blue edge of some $q^\beta_1$ must be identified with the blue edge of some $q^{\beta'}_{\delta_{\beta'}}$. This in particular implies that $q^\beta_1$ is contained in $Q(h_1 \cup h_2)$, hence $h_\beta$ is deep.

The edge identification implies that $v^{\beta'}_{\delta_{\beta'}}$ is equal to the bottom vertex of $s^\beta_1$. There are two manners of identification here. One is such that $e^{\beta'}_{\delta_{\beta'}+1}$ and $v_0$ lie on the same side of $s^\beta_1$, which is the statement of the lemma, while the other is such that they lie on different sides of $s^\beta_1$, which we tackle for the rest of the proof.

If $h_\beta$ and $h_{\beta'}$ lie in the same component $C_0$ of $\Gamma(h_1 \cup h_2)$, then by following along the sides of $s^\beta_1$, we see that $v^{\beta'}_{\delta_{\beta'}}$ only meets $C_0$ and no other component, contradicting the fact that $q^{\beta'}_{\delta_{\beta'}}$ is included in the 2-cycle. 

The only case where $h_\beta$ and $h_{\beta'}$ do not lie in the same component is if $\beta \neq \beta'$, and the bottom vertex of $s$ is resolved, which implies that it is equal to $v^\beta_{\delta_\beta}$ or $v^{\beta'}_{\delta_{\beta'}}$, since these are the only two blue vertices we resolve.

In the former case, the identification must be so that $s=s^\beta_1$, otherwise $h_\beta$ and $h_{\beta'}$ lie in the same component, but then the statement of the lemma is still true. See \Cref{fig:doubledeephookfan2} left (where we take $\beta$ to be the left side). Similarly, in the latter case, the identification must be so that $s=s^{\beta'}_1$. But since we are assuming that $e^{\beta'}_{\delta_{\beta'}+1}$ lies on a bottom side of $s^\beta_1$, we must have $s^\beta_1=s$ so we are in fact the former case again. See \Cref{fig:doubledeephookfan2} right. 

\begin{figure}[ht]
    \centering
    \fontsize{12pt}{12pt}\selectfont
    \resizebox{!}{4.5cm}{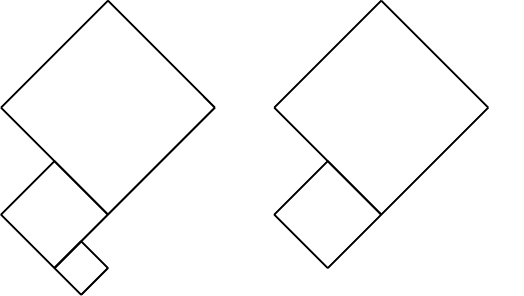}
    \caption{Reasoning that \Cref{lemma:doubledeephookfan} holds when $h_\beta$ and $h_{\beta'}$ do not lie in the same component.}
    \label{fig:doubledeephookfan2}
\end{figure}

Finally, the second statement follows from \Cref{prop:vbstogglefan} as in \Cref{lemma:singleprehook}.
\end{proof}

Here are two more situations where we can show that $\Gamma(h_1 \cup h_2)$ is connected.

\begin{lemma} \label{lemma:doublehooktophingesequal}
Suppose that:
\begin{itemize}
    \item $s$ satisfies (FRC)
    \item $v^1_{\delta_1} = v^2_{\delta_2}$
    \item $e^1_{\delta_1+1} \neq e^1_1$ and $e^2_{\delta_2+1} \neq e^2_1$
\end{itemize}
Then $\Gamma(h_1 \cup h_2)$ is connected, i.e. (I) of \Cref{prop:doublehookresolution} is true.
\end{lemma}
\begin{proof}
Let $C_i$ be the component of $\Gamma(h_1 \cup h_2)$ containing $h_i$, for $i=1,2$. We first claim that $C_1=C_2$.
If the bottom vertex of $s$ is not resolved, then this is clear. Otherwise this bottom vertex must be $v^1_{\delta_1} = v^2_{\delta_2}$. The third condition in the assumption implies that the identification must be such that $e^2_{\delta_2+1} = e^1_1$, so by following along the sides of $s$, we see that $C_1=C_2$ in this case as well.

Now suppose $\Gamma(h_1 \cup h_2)$ is not connected. Consider the 2-cycle of $Q(h_1 \cup h_2)$ corresponding to $C_1$. Since $v^1_{\delta_1} = v^2_{\delta_2}$, $C_1$ is the only component of $\Gamma(h_1 \cup h_2)$ that meets this vertex, so the quad $q^1_{\delta_1}=q^2_{\delta_2}$ will not be included in the 2-cycle. Since $s$ satisfies (FRC), the 2-cycle must then contain the quads $q^1_1$ and $q^2_1$ and identify their blue edges. But this implies that $v^1_1=v^2_1$, and that this common vertex only meets $C_1$. So these quads are not in the 2-cycle as well, giving us a contradiction.
\end{proof}

\begin{lemma} \label{lemma:doublehookbottomhingesequal}
Suppose that:
\begin{itemize}
    \item $s$ satisfies (FRC)
    \item $v^1_1 = v^2_1$
    \item $e^1_{\delta_1+1} \neq e^1_1$ and $e^2_{\delta_2+1} \neq e^2_1$
\end{itemize}
Then $\Gamma(h_1 \cup h_2)$ is connected, i.e. (I) of \Cref{prop:doublehookresolution} is true.
\end{lemma}
\begin{proof}
This lemma is similar to \Cref{lemma:doublehooktophingesequal}. 
The same argument as in \Cref{lemma:doublehooktophingesequal} shows that $h_1$ and $h_2$ lie in a common component $C$. 
If $\Gamma(h_1 \cup h_2)$ is not connected, consider the 2-cycle of $Q(h_1 \cup h_2)$ corresponding to $C$. Since $v^1_1 = v^2_1$, $C$ is the only component of $\Gamma(h_1 \cup h_2)$ that meets this vertex, so the quad $q^1_1=q^2_1$ will not be included in the 2-cycle. Since $s$ satisfies (FRC), the 2-cycle must then contain the quads $q^1_{\delta_1}$ and $q^2_{\delta_2}$ and identify their blue edges. But this implies that $v^1_{\delta_1}=v^2_{\delta_2}$, and that this common vertex only meets $C$. So these quads are not in the 2-cycle as well, giving us a contradiction.
\end{proof}

Now we state two sets of conditions under which we can show the other scenario in \Cref{prop:doublehookresolution} is true. 

\begin{lemma} \label{lemma:doublehookoneshortcycle}
Suppose that:
\begin{itemize}
    \item $\Gamma(h_1 \cup h'_2)$ is connected (in particular $s$ satisfies (FRC))
    \item $e^2_{\delta_2+1}=e^2_1$
\end{itemize}
Then (II) in \Cref{prop:doublehookresolution} is true.
\end{lemma}
\begin{proof}
We can assume that $h_2$ is deep here, since resolving less vertices cannot create more components for the resolved dual graph.

Since $\Gamma(h_1 \cup h'_2)$ is connected, by resolving two more points $v^2_{\delta_2}$ and $v^2_1$, we get at most 3 components in $\Gamma(h_1 \cup h_2)$. Here $(e^2_i)_{i \in \mathbb{Z}/{\delta_2}}$ is clearly a component on its own, so it remains to show that $\Gamma(h_1 \cup h_2)$ does not have 3 components.

Suppose otherwise. Let $C_0$ be the component that is $(e^2_i)_{i \in \mathbb{Z}/{\delta_2}}$. In this case, $v^2_{\delta_2}$ must meet components $C_0$ and $C_1$, while $v^2_1$ must meet components $C_0$ and $C_2$, where $C_1 \neq C_2$. 

In particular $e^1_1$ lies in $C_1$ and $e^1_{\delta_1+1}$ lies in $C_2$. But then following along the sides of $s$, we see that $C_1=C_2$, giving us a contradiction. 
\end{proof}

\begin{lemma} \label{lemma:doublehookonecycleothernondeep}
Suppose that:
\begin{itemize}
    \item $s$ satisfies (FRC)
    \item $h_1$ is not deep
    \item $e^2_{\delta_2+1}=e^2_1$
\end{itemize}
Then (II) in \Cref{prop:doublehookresolution} is true.
\end{lemma}
\begin{proof}
By assumption, $\Gamma(h'_1 \cup h'_2)$ is connected. So if we resolve one more point $v^1_{\delta_1}$, $\Gamma(h_1 \cup h'_2)$ has at most 2 components. However, $Q(h_1 \cup h'_2)$ only has one blue edge, namely that of $q^1_{\delta_1}$, so there cannot be any 2-cycle containing $q^1_{\delta_1}$. Hence $\Gamma(h_1 \cup h'_2)$ is connected. Now the lemma follows from \Cref{lemma:doublehookoneshortcycle}.
\end{proof}

\section{One boundary component case} \label{sec:oneboundary}

The goal of this section is to prove \Cref{thm:intro16tet}. 
From \Cref{subsec:manybranchcycle} to \Cref{subsec:oneboundarytoggle}, we will devise an arsenal of propositions, each of which provides an adequate bound in a specific circumstance.
With these in place, we will then run an elimination process in \Cref{subsec:EIIRP}, similar to that in \Cref{subsec:singlehookexist} (but even more elaborate), in order to obtain an overall improved bound (\Cref{thm:EIIRP}).
For the reader's convenience, we have provided a flow chart in \Cref{fig:flowchart} summarizing the elimination process, intended to be consulted after a first reading of this section.
Finally, one can deduce \Cref{thm:intro16tet} by substituting $\lambda^{-\chi} = 6.86$.

Each proposition will incorporate the results of \Cref{sec:doublehook}. The recurring idea is that either one of the propositions in \Cref{subsec:doublehookobstruct} applies, or the combinatorics of the triangulation is locally constrained in specific ways. Together with the assumption that the mapping torus has only one boundary component, this will allow us to locate sources of improvement for the estimates in the proof of \Cref{prop:singlehookbound}.

To be more specific about these improvements for estimates, let us summarize the argument of \Cref{prop:singlehookbound} by \Cref{tab:singlehookbound}. Here we think of each vertex of $-c$ as contributing a term in the sum on the right hand side of \Cref{eq:singlehook}, and we table up the quantity and contributions of each type of vertex. What we will do is, with certain knowledge of the stable branched surface, show that the contribution from certain vertices can be improved. When we do so, we will summarize the proof in tables like \Cref{tab:singlehookbound} but with more rows.

\begin{table}[ht]
    \centering
    \caption{The argument in \Cref{prop:singlehookbound}.}
    \begin{tabular}{|C{6cm}|C{6cm}|C{3cm}|}
    \hline
    Vertices of $-c$ & Quantity & Contribution \\
    \hline \hline
    Pairs of vertices that meet a B-resolved hook vertex & $\text{\# hook vertices}-1$ & $2w$ \\
    \hline   
    Remaining non-hook vertices & $2N-2\text{(\# hook vertices)}+1$ & $w$ \\
    \hline
    \end{tabular}
    \label{tab:singlehookbound}
\end{table}

We also outline the flavor of the propositions in each individual subsection:
In \Cref{subsec:manybranchcycle}, \Cref{prop:manybranchcycle} will tackle the case with at least three branch cycles. 
This will allow us to reduce to the case of having one or two branch cycles.
In \Cref{subsec:fewbranchcycle}, \Cref{prop:fewbranchcyclenotranslate} and \Cref{prop:fewbranchcyclefrc} will allow us to further reduce to cases with specific combinatorics at minimum weight sectors.
The improvement for estimates in these propositions will come from the last vertex on each branch cycle visited by the hook circuit.

In \Cref{subsec:oneboundaryfan}, the propositions will address various cases with the aforementioned combinatorial restriction and with a minimum weight fan sector.
Similarly, in \Cref{subsec:oneboundarytoggle}, the propositions will address cases with a minimum weight toggle sector.
The improvement for estimates here will come from pass-throughs of the hook circuit at certain vertices close to the minimum weight sector, which we can locate because of the very specific combinatorial hypotheses.

For the rest of this section, we fix the following setting. Let $T_f$ be the mapping torus of a fully-punctured pseudo-Anosov map $f$, where $T_f$ has only one boundary component. Let $\Delta$ be the veering triangulation on $T_f$ that carries $\Lambda^s$, let $B$ be the stable branched surface of $\Delta$, and let $\Gamma$ be the dual graph of $\Delta$. Finally, let $S$ be a fiber surface. We use the notation as in \Cref{subsec:layeredvt}.

\subsection{When the number of branch cycles is large} \label{subsec:manybranchcycle}

The crucial fact about $T_f$ having only one boundary component is if there are $l$ branch cycles in the stable branched surface, then each branch cycle intersects the fiber surface $-\frac{2}{l} \chi(S)$ times. This is because all the branch cycles are homotopic to the same slope on the boundary of $T_f$ (sometimes known as the \textit{degeneracy slope} in the literature), hence are homotopic to each other. 

Now consider a sector $s$ of $B$ of minimum weight. Let $c$ be a branch cycle that meets a bottom side of $s$. Take the bottom vertex of $s$ to be the basepoint of $c$. Let $\widehat{s_0}$ be the lift of $s$ in $\widehat{B}$ that is of height 0. Lift $c$ to a path $\widehat{c}$ ending at the bottom vertex of $\widehat{s_0}$, then push $\widehat{c}$ \emph{upwards} on the side of $\widehat{s_0}$ and reverse its orientation to get a descending path $\alpha$. The starting point of $\alpha$ lies on $\widehat{s_0}$ which is of weight $w$, while the ending point of $\alpha$ lies on $g^{\frac{2}{l}\chi} \widehat{s}$ which is of weight $\lambda^{-\frac{2}{l}\chi}w$, since $c$ intersects the fiber surface $-\frac{2}{l}\chi$ times. Moreover, at each intersection point of $\alpha$ with the branch locus of $\widehat{B}$, the sector that merges in is of non-positive height, hence has weight $\geq w$.

Therefore if $\lambda^{-\frac{2}{l}\chi} < n+2$ for integer $n$, then $\alpha$ must intersect the branch locus of $\widehat{B}$ at most $n$ times. This implies that the branch cycle $c$ meets at most $n$ vertices of the same color as $s$.
Notice that this immediately implies that $\lambda^{-\frac{2}{l}\chi} < 2$ is impossible. We will not need to use this fact in full but we record it as a proposition anyway.

\begin{prop} \label{prop:verymanybranchcycle}
Let $f$ be a fully-punctured pseudo-Anosov map with normalized dilatation $\lambda^{-\chi}$. Suppose the mapping torus $T_f$ of $f$ only has one boundary component, then the stable branched surface $B$ has less than or equal to $\frac{2 \log \lambda^{-\chi}}{\log 2}$ branch cycles.
\end{prop}

We will focus on the case when $\lambda^{-\frac{2}{l}\chi} < 4$ in this subsection. This occurs when $l$ is large enough (relative to $\lambda^{-\chi}$), hence the name of the subsection.

Let $s$ be a sector of $B$ of minimum weight and let $c$ be an Eulerian circuit that hooks around $s$, say $c$ contains the hook $h_1$. Such a pair $(s,c)$ exists unless $\Delta =$ \texttt{cPcbbbdxm\_10} by \Cref{prop:singlehookexist}. But for $\Delta =$ \texttt{cPcbbbdxm\_10}, $l=1$ and $\lambda^{-\chi}=\mu^2$ so $\lambda^{-\frac{2}{l}\chi} < 4$ does not hold anyway, thus we can ignore this exceptional case in this subsection. We use the same notation on the edges, vertices, and sectors adjacent to $s$ as in the previous sections.

If $\lambda^{-\frac{2}{l}\chi} < 4$, each branch cycle meeting a bottom side of $s$ meets at most $2$ vertices of the same color as $s$. This implies that $v^1_{\delta_1}$ is equal to the bottom vertex of $s^2_1$ in a way such that $e^1_{\delta_1+1}$ and $v_0$ lie on different sides of $s^2_1$, and similarly, $v^2_{\delta_1}$ is equal to the bottom vertex of $s^1_1$ in a way such that $e^2_{\delta_1+1}$ and $v_0$ lie on different sides of $s^1_1$. See \Cref{fig:manybranchcycle}.

\begin{figure}[ht]
    \centering
    \fontsize{12pt}{12pt}\selectfont
    \resizebox{!}{4.5cm}{
\begingroup%
  \makeatletter%
  \providecommand\color[2][]{%
    \errmessage{(Inkscape) Color is used for the text in Inkscape, but the package 'color.sty' is not loaded}%
    \renewcommand\color[2][]{}%
  }%
  \providecommand\transparent[1]{%
    \errmessage{(Inkscape) Transparency is used (non-zero) for the text in Inkscape, but the package 'transparent.sty' is not loaded}%
    \renewcommand\transparent[1]{}%
  }%
  \providecommand\rotatebox[2]{#2}%
  \newcommand*\fsize{\dimexpr\f@size pt\relax}%
  \newcommand*\lineheight[1]{\fontsize{\fsize}{#1\fsize}\selectfont}%
  \ifx\svgwidth\undefined%
    \setlength{\unitlength}{119.86879844bp}%
    \ifx\svgscale\undefined%
      \relax%
    \else%
      \setlength{\unitlength}{\unitlength * \real{\svgscale}}%
    \fi%
  \else%
    \setlength{\unitlength}{\svgwidth}%
  \fi%
  \global\let\svgwidth\undefined%
  \global\let\svgscale\undefined%
  \makeatother%
  \begin{picture}(1,1.18429918)%
    \lineheight{1}%
    \setlength\tabcolsep{0pt}%
    \put(0,0){\includegraphics[width=\unitlength,page=1]{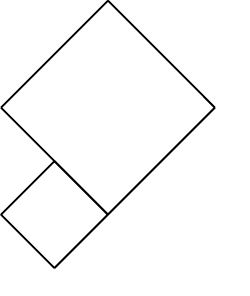}}%
    \put(0.38364879,0.7145134){\color[rgb]{0,0,0}\makebox(0,0)[lt]{\lineheight{1.25}\smash{\begin{tabular}[t]{l}\Huge $s$\end{tabular}}}}%
    \put(0,0){\includegraphics[width=\unitlength,page=2]{manybranchcycle.pdf}}%
    \put(0.72680652,0.08545649){\color[rgb]{0,0,0}\makebox(0,0)[lt]{\lineheight{1.25}\smash{\begin{tabular}[t]{l}$s$\end{tabular}}}}%
    \put(0,0){\includegraphics[width=\unitlength,page=3]{manybranchcycle.pdf}}%
    \put(0.08412792,0.08545649){\color[rgb]{0,0,0}\makebox(0,0)[lt]{\lineheight{1.25}\smash{\begin{tabular}[t]{l}$s$\end{tabular}}}}%
  \end{picture}%
\endgroup%
}
    \caption{If $\lambda^{-\frac{2}{l}\chi} < 4$, each branch cycle meeting a bottom side of $s$ meets at most $2$ vertices of the same color as $s$.}
    \label{fig:manybranchcycle}
\end{figure}

If $s$ fails (FRC), then by \Cref{lemma:doubleprehook}, $v^1_{\delta_1}=v^2_{\delta_2}$. If $s$ satisfies (FRC) but $\Gamma(h_1 \cup h_2)$ is not connected, then $v^{\beta'}_{\delta_{\beta'}}$ is equal to the bottom vertex of $s^\beta_1$ for some $\beta, \beta'$ in a way such that $e^{\beta'}_{\delta_{\beta'}+1}$ and $v_0$ lie on the same side of $s^\beta_1$, by \Cref{lemma:doubledeephookfan}. But in the current setting, this implies that $v^1_{\delta_1}=v^2_{\delta_2}$. So we see that unless $\Gamma(h_1 \cup h_2)$ is connected, we must have $v^1_{\delta_1}=v^2_{\delta_2}$.

\begin{prop} \label{prop:manybranchcycle}
Suppose that $l > \frac{\log \lambda^{-\chi}}{\log 2}$. Then the number of tetrahedra in $\Delta$ is $\leq \max \{\frac{1}{4} \lambda^{-2\chi} +1, \frac{1}{2} (\lambda^{-2\chi} - \frac{\lambda^{-2\chi} - \lambda^{-\frac{2}{l}\chi}}{\lambda^{-\frac{2}{l}\chi} - 1} + l) \}$
\end{prop}
\begin{proof}
If $\Gamma(h_1 \cup h_2)$ is connected, then the proposition follows from \Cref{prop:doublehookresolution} and \Cref{prop:doublehookbound}.
If $\Gamma(h_1 \cup h_2)$ is not connected, then by the reasoning before the proposition, we have $v^1_{\delta_1}=v^2_{\delta_2}$. The way we will prove the proposition is to improve the estimates made in \Cref{prop:singlehookbound} when that argument is applied to the Eulerian hook circuit $c$.

First notice that $v^1_{\delta_1}=v^2_{\delta_2}$ implies that the sides of $s$ lie along two branch cycles, say $c_1$ and $c_2$ among the $l$ branch cycles $c_1,...,c_l$. Let $u_i$, $i=3,...,l$ be the last vertex of $-c$ that lies on $c_i$. By definition, $c$ takes an anti-branching turn at each $u_i$, and each $u_i$ does not lie on the sides of $s$. Without loss of generality, suppose that $-c$ meets $u_i$ in the order of increasing $i$.

In the argument of \Cref{prop:singlehookbound}, the contribution of the terms corresponding to each $u_i$ is $w$, since each $u_i$ lies on $c_i$ hence does not meet a hook vertex. But between the last vertex on the hook and $u_i$, $-c$ will have passed through $c_3,...,c_i$ in their entirety, hence will have intersected the fiber surface at least $-\frac{2(i-2)}{l}\chi$ times. So the sector in the term corresponding to $u_i$ is in fact at height $\leq \frac{2(i-2)}{l}\chi$, hence of weight $\geq \lambda^{-\frac{2(i-2)}{l}\chi} w$. This implies that we can obtain a sharper bound if we replace the $w$ contributed by $u_i$ to $\lambda^{-\frac{2(i-2)}{l}\chi} w$.

We can make one more improvement, the source of the improvement differing in two cases. 
Case 1 is if $v^1_{\delta_1}=v^2_{\delta_2}$ is not the bottom vertex of $s$. In this case, the edges of $\Gamma$ that lie on the sides of $s$ are all distinct. Let $p^\beta_k$ be the number of intersection points between $e^\beta_k$ and the fiber surface. Then $\sum_{k=1}^{\delta_1+1} p^1_k + \sum_{k=1}^{\delta_2+1} p^2_k$ is at most the number of intersection points of $c_1$ and $c_2$ with the fiber surface, which is $\frac{2}{l} \cdot (-2\chi)$. But $\sum_{k=1}^{\delta_1+1} p^1_k = \sum_{k=1}^{\delta_2+1} p^2_k$ since every arc has endpoints on different sides of $s$, so $\sum_{k=1}^{\delta_1+1} p^1_k = \sum_{k=1}^{\delta_2+1} p^2_k \leq -\frac{2}{l}\chi$.

Now consider the last vertex of $-c$, which we denote by $u$. $u$ sits at the top vertex of $s$. Between the last vertex on the hook and $u$, $c$ will have intersected the fiber surface for $-2\chi - \sum_{k=1}^{\delta_1+1} p^1_k \geq -\frac{2(l-1)}{l}\chi$ times. Since $v^1_{\delta_1}=v^2_{\delta_2}$ is not the bottom vertex of $s$, $u$ is not a hook vertex. As a result, in the argument of \Cref{prop:singlehookbound}, the contribution of $u$ is $w$, and we can replace this contribution by $\lambda^{-\frac{2(l-1)}{l}\chi} w$ to get a better bound.

Case 2 is if $v^1_{\delta_1}=v^2_{\delta_2}$ is the bottom vertex of $s$. Notice that in this case $s$ must be toggle, otherwise the top sides of $s$ form one or two branch cycles meeting vertices of only one color, contradicting \Cref{prop:vbscontradictions}. See \Cref{fig:manybranchcyclecontradiction}.

\begin{figure}
    \centering
    \fontsize{12pt}{12pt}\selectfont
    \resizebox{!}{3cm}{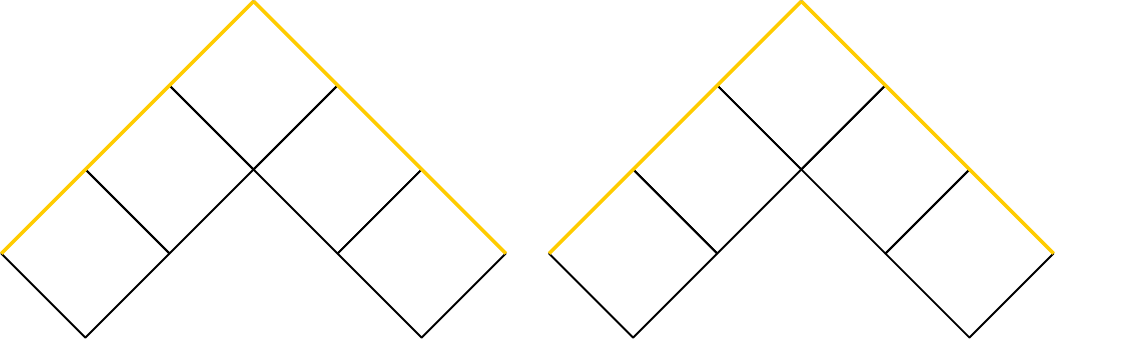}
    \caption{If $v^1_{\delta_1}=v^2_{\delta_2}$, then $s$ has to be toggle, otherwise the top sides of $s$ form one or two branch cycles (in yellow) meeting vertices of only one color.}
    \label{fig:manybranchcyclecontradiction}
\end{figure}

In this case, we consider the last vertex $u$ of $-c$ as well. The term corresponding to $u$ is the weight of a sector of $\widehat{B}$ covering $s$. If we let $p^\beta_k$ be the number of intersection points between $e^\beta_k$ and the fiber surface as above, then the height of this sector will be $2\chi + \sum_{k=1}^{\delta_\beta} p^\beta_k$ for $\beta=1$ or 2 depending on the manner of identification of $v^1_{\delta_1}=v^2_{\delta_2}$ with the bottom vertex of $s$, and whether $c$ takes an anti-branching or branching turn at $u$. In either case, $\sum_{k=1}^{\delta_\beta} p^\beta_k$ is at most the number of intersection points of one branch cycle with the fiber surface, which is $-\frac{2}{l}\chi$. 

We can now replace the contribution of $u$ by $\lambda^{-\frac{2(l-1)}{l}\chi} w$. However, there is a difference with the previous case here. The original contribution of $u$ is not necessarily $w$. If $u$ is paired up with a B-resolved hook vertex, then the contribution of the pair is $2w$, while the contribution of $u$ itself is greater than $w$. So what we do here is remove the contributions of both terms in the pair, and replace them by $\lambda^{-\frac{2(l-1)}{l}\chi} w$.

With these optimizations, \Cref{eq:singlehook} becomes
\begin{align*}
    \lambda^{-2\chi} w &\geq 2Nw + (\lambda^{-\frac{2}{l}\chi} w - w) + ... + (\lambda^{-\frac{2(l-2)}{l}\chi} w - w) + (\lambda^{-\frac{2(l-1)}{l}\chi} w - w) \\
    &= 2Nw + \frac{\lambda^{-2\chi} - \lambda^{-\frac{2}{l}\chi}}{\lambda^{-\frac{2}{l}\chi} - 1}w - (l-1)w
\end{align*}
in case 1 and 
\begin{align*}
    \lambda^{-2\chi} w &\geq 2Nw + (\lambda^{-\frac{2}{l}\chi} w - w) + ... + (\lambda^{-\frac{2(l-2)}{l}\chi} w - w) + (\lambda^{-\frac{2(l-1)}{l}\chi} w - 2w) \\
    &= 2Nw + \frac{\lambda^{-2\chi} - \lambda^{-\frac{2}{l}\chi}}{\lambda^{-\frac{2}{l}\chi} - 1}w - lw
\end{align*}
in case 2.
So in both cases, we have 
\begin{align*}
    \lambda^{-2\chi} &\geq 2N + \frac{\lambda^{-2\chi} - \lambda^{-\frac{2}{l}\chi}}{\lambda^{-\frac{2}{l}\chi} - 1} - l \\
    N &\leq \frac{1}{2} (\lambda^{-2\chi} - \frac{\lambda^{-2\chi} - \lambda^{-\frac{2}{l}\chi}}{\lambda^{-\frac{2}{l}\chi} - 1} + l)
    \qedhere
\end{align*}
\end{proof}

We summarize the part of the proof of \Cref{prop:manybranchcycle} where we improve \Cref{prop:singlehookbound} in \Cref{tab:manybranchcycle}.

\begin{table}[ht]
    \centering
    \caption{The argument in \Cref{prop:manybranchcycle}.}
    \begin{tabular}{|C{6cm}|C{6cm}|C{3cm}|}
    \hline
    Vertices of $-c$ & Quantity & Contribution \\
    \hline \hline
    Last vertex on $c_i$, $i=3,...,l$ & $l-2$ &  $\lambda^{-\frac{2(i-2)}{l}\chi} w$ \\
    \hline
    Last vertex & 1 & $\lambda^{-\frac{2(l-1)}{l}\chi} w$ \\
    \hline
    \multirow{2}{6cm}{\centering Remaining pairs of vertices that meet a B-resolved hook vertex} & Case 1: $\text{\# hook vertices}-1$ & \multirow{2}{3cm}{\centering $2w$} \\
     & Case 2: $\text{\# hook vertices}-2$ & \\
    \hline   
    \multirow{3}{6cm}{\centering Remaining non-hook vertices} & Case 1: $2N-2\text{(\# hook vertices)}-l+2$ & \multirow{3}{3cm}{\centering $w$} \\
     & Case 2: $2N-2\text{(\# hook vertices)}-l+3$ & \\
    \hline
    \end{tabular}
    \label{tab:manybranchcycle}
\end{table}

We note that when $\lambda^{-\chi}=\mu^4$, $l > \frac{\log \lambda^{-\chi}}{\log 2}$ reads $l > 2.78$, that is, $l \geq 3$. So for our goal of proving \Cref{thm:introdilthm}, we must now turn our attention to the case when $l=1$ or $2$.

\subsection{When the number of branch cycles is small} \label{subsec:fewbranchcycle}

As in the last subsection, let $s$ be a sector of minimum weight. Without loss of generality assume that $s$ is blue. Let $c$ be an Eulerian circuit that hooks around $s$, say $c$ contains the hook $h_1$. We use the same notation on the edges, vertices, and sectors adjacent to $s$ as always. Suppose $B$ has $l$ branch cycles, where $l=1$ or $2$.

The key property that $l=1$ or $2$ buys us is \Cref{prop:fewbranchcyclefrc}, which says that we can assume that $s$ satisfies (FRC) hence apply the lemmas in \Cref{subsec:doublehookobstruct}, for otherwise we have an improved bound by other means. Before proving \Cref{prop:fewbranchcyclefrc}, it will be in our advantage to prove \Cref{prop:fewbranchcyclenotranslate} first.

\begin{prop} \label{prop:fewbranchcyclenotranslate}
When $l=1$ or $2$, either $e^1_{\delta_1+1} \neq e^2_1$ and $e^2_{\delta_2+1} \neq e^1_1$, or the number of tetrahedra in $\Delta$ is $\leq \frac{1}{4} \lambda^{-2\chi} +1$.
\end{prop}
\begin{proof}
If $l=1$ and, say, $e^1_{\delta_1+1} = e^2_1$, then $(e^1_i)_{i \in \mathbb{Z}/{\delta_1}}$ is the sole branch cycle of $B$. But this branch cycle only meets one blue vertex, contradicting the fact that any Eulerian circuit must meet each vertex twice.

If $l=2$ and, say, $e^1_{\delta_1+1} = e^2_1$, then as above, $(e^1_i)_{i \in \mathbb{Z}/{\delta_1}}$ is a branch cycle of $B$. We denote this branch cycle by $c_1$, and denote the other branch cycle of $B$ by $c_2$. 
If $e^2_{\delta_2+1}=e^1_1$, then $c_2=(e^2_i)_{i \in \mathbb{Z}/{\delta_2}}$. In this case, $(e^1_1,...,e^1_{\delta_1},e^2_1,...,e^2_{\delta_2})$ is an Eulerian circuit that hooks around $s$ twice, implying that the number of tetrahedra in $\Delta$ is $\leq \frac{1}{4} \lambda^{-2\chi}+1$ by \Cref{prop:doublehookbound}.

On the other hand, if $e^2_{\delta_2+1} \neq e^1_1$, then all sides of $s$ other than $(e^1_1,...,e^1_{\delta_1})$ lie on $c_2$. See \Cref{fig:fewbranchcyclenotranslate}. Here we use the fact that if $e^1_{\delta_1+1}=e^1_1$, then $s$ will not be embedded in its interior near $v_0$, which is a contradiction. Moreover, the edges $e^2_k$, $k=1,...,\delta_2+1$ are distinct. Here we use the fact that if $e^2_{\delta_2+1}=e^2_1$, then again $s$ will not be embedded in its interior near $v_0$.

\begin{figure}
    \centering
    \fontsize{12pt}{12pt}\selectfont
    \resizebox{!}{4cm}{
\begingroup%
  \makeatletter%
  \providecommand\color[2][]{%
    \errmessage{(Inkscape) Color is used for the text in Inkscape, but the package 'color.sty' is not loaded}%
    \renewcommand\color[2][]{}%
  }%
  \providecommand\transparent[1]{%
    \errmessage{(Inkscape) Transparency is used (non-zero) for the text in Inkscape, but the package 'transparent.sty' is not loaded}%
    \renewcommand\transparent[1]{}%
  }%
  \providecommand\rotatebox[2]{#2}%
  \newcommand*\fsize{\dimexpr\f@size pt\relax}%
  \newcommand*\lineheight[1]{\fontsize{\fsize}{#1\fsize}\selectfont}%
  \ifx\svgwidth\undefined%
    \setlength{\unitlength}{135.41789137bp}%
    \ifx\svgscale\undefined%
      \relax%
    \else%
      \setlength{\unitlength}{\unitlength * \real{\svgscale}}%
    \fi%
  \else%
    \setlength{\unitlength}{\svgwidth}%
  \fi%
  \global\let\svgwidth\undefined%
  \global\let\svgscale\undefined%
  \makeatother%
  \begin{picture}(1,1.04662902)%
    \lineheight{1}%
    \setlength\tabcolsep{0pt}%
    \put(0,0){\includegraphics[width=\unitlength,page=1]{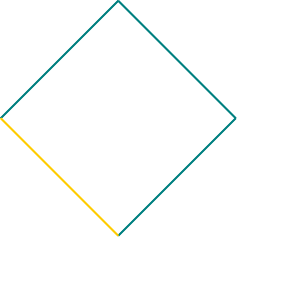}}%
    \put(0.37431883,0.59649542){\color[rgb]{0,0,0}\makebox(0,0)[lt]{\lineheight{1.25}\smash{\begin{tabular}[t]{l}\Hu $s$\end{tabular}}}}%
    \put(0,0){\includegraphics[width=\unitlength,page=2]{fewbranchcyclenotranslate.pdf}}%
    \put(0.59252952,0.17152241){\color[rgb]{0,0,0}\makebox(0,0)[lt]{\lineheight{1.25}\smash{\begin{tabular}[t]{l}\La $s$\end{tabular}}}}%
    \put(0.1194401,0.327242){\color[rgb]{1,0.8,0}\makebox(0,0)[lt]{\lineheight{1.25}\smash{\begin{tabular}[t]{l}\La $c_1$\end{tabular}}}}%
    \put(0.62778154,0.87565266){\color[rgb]{0,0.50196078,0.50196078}\makebox(0,0)[lt]{\lineheight{1.25}\smash{\begin{tabular}[t]{l}\La $c_2$\end{tabular}}}}%
  \end{picture}%
\endgroup%
}
    \caption{If $e^1_{\delta_1+1} = e^2_1$ and $e^2_{\delta_2+1} \neq e^1_1$, then one bottom side of $s$ lies along the branch cycle $c_1=(e^1_i)_{i \in \mathbb{Z}/{\delta_1}}$ (in yellow) while all other sides of $s$ lie on $c_2$ (in teal).}
    \label{fig:fewbranchcyclenotranslate}
\end{figure}

Let $p^\beta_k$ be the number of intersections of $e^\beta_k$ with the fiber surface. We have the following equations:
\begin{itemize}
    \item $\sum_{k=1}^{\delta_1+1} p^1_k = \sum_{k=1}^{\delta_2+1} p^2_k$
    \item $\sum_{k=1}^{\delta_1} p^1_k = -\chi$
    \item $p^1_{\delta_1+1}=p^2_1$
    \item $\sum_{k=1}^{\delta_2+1} p^2_k \leq -\chi$
\end{itemize}
which imply that $p^1_{\delta_1+1}=p^2_1=0$ and $\sum_{k=2}^{\delta_2+1} p^2_k = -\chi$. That is, the intersection points of $c_2$ with the fiber surface are all lie on $e^2_k$, $k=2,...,\delta_2+1$.

Consider cutting and pasting the branch cycle $c_2$ at $v^2_{\delta_2}$. This gives two cycles of $\Gamma$, one containing $e^2_k$, $k=2,...,\delta_2+1$, and the other not containing any of these edges, thus does not intersect the fiber surface. But this latter cycle would then contradict \Cref{prop:nonullcycles}. 
\end{proof}

\begin{prop} \label{prop:fewbranchcyclefrc}
Suppose $l=1$ or $2$ and suppose there is an Eulerian circuit $c$ that hooks around $s$. Then either the minimum weight sector $s$ satisfies (FRC) or the number of tetrahedra in $\Delta$ is $\leq \max \{ \frac{1}{4} \lambda^{-2\chi} +1, \frac{1}{2} (\lambda^{-2\chi}-\lambda^{-\frac{4}{3}\chi}-\lambda^{-\chi})+1\}$.
\end{prop}
\begin{proof}
By \Cref{lemma:doubleprehook}, if $s$ does not satisfy (FRC), then $v^1_{\delta_1}=v^2_{\delta_2}$ and the component of $\Gamma(h'_1 \cup h'_2)$ not containing $h'_1 \cup h'_2$ contains at least one branch cycle. This immediately implies that if $l=1$, then $s$ must satisfy (FRC), so we assume that $l=2$ and $s$ does not satisfy (FRC) in the rest of the proof.

In this case, the component not containing $h'_1 \cup h'_2$ contains exactly one branch cycle, which we denote by $c_2$. All the sides of $s$ are contained in the other component of $\Gamma(h'_1 \cup h'_2)$, hence contained in the branch cycle other than $c_2$, which we denote by $c_1$. 
If $v^1_{\delta_1}=v^2_{\delta_2}$ is equal to the bottom vertex of $s$ and $e^1_{\delta_1+1}=e^2_1$ (equivalently, $e^2_{\delta_2+1}=e^1_1$), then the number of tetrahedra in $\Delta$ is $\leq \frac{1}{4} \lambda^{-2\chi} +1$ by \Cref{prop:fewbranchcyclenotranslate}.

If $v^1_{\delta_1}=v^2_{\delta_2}$ is equal to the bottom vertex of $s$ and $e^1_{\delta_1+1}=e^1_1$ (equivalently, $e^2_{\delta_2+1}=e^2_1$), first notice that the hook $h_1$ is not deep, otherwise $\Gamma(h_1)$ is not connected, for $(e^1_i)_{i \in \mathbb{Z}/{\delta_1}}$ would be its own component, contradicting the hypothesis that there is a hook circuit $c$ containing $h_1$. Let $p^\beta_k$ be the number of intersections of $e^\beta_k$ with the fiber surface. Then we have the equations:
\begin{itemize}
    \item $\sum_{k=1}^{\delta_1+1} p^1_k = \sum_{k=1}^{\delta_2+1} p^2_k$
    \item $p^1_{\delta_1+1}=p^1_1=0$
    \item $p^2_{\delta_2+1}=p^2_1$
    \item $\sum_{k=1}^{\delta_1} p^1_k + \sum_{k=1}^{\delta_2} p^2_k \leq -\chi$
\end{itemize}
which imply that the number of times the hook $h_1$ intersects the fiber surface, which is $\sum_{k=2}^{\delta_1+1} p^1_k$, is at most $-\frac{2}{3} \chi$. 

What we can do now is to replace the contribution of the last vertex of $-c$ on $c_2$, which we denote by $u_2$, from $w$ to $\lambda^{-\chi}w$. As in \Cref{prop:manybranchcycle}, the new estimate arises from the fact that $-c$ would have traversed $c_2$ all the way by the time it reached $u_2$, hence intersected the fiber surface at least $-\chi$ times, while the old estimate is $w$ because $u_2$ lies on $c_2$ (and $c$ takes an anti-branching turn at $u_2$) hence cannot meet a hook vertex. 

We also replace the contribution of the last vertex of $-c$, which we denote by $u$, from $w$ to $\lambda^{-\frac{4}{3}\chi}w$. The original contribution is $w$ because $h_1$ is not deep, so $u$ does not meet a hook vertex. The new contribution arises from the fact that between the last vertex on the hook and $u$, $-c$ meets the fiber surface for $\geq -2\chi + \frac{2}{3} \chi = -\frac{4}{3} \chi$ times.

With these improvements, \Cref{eq:singlehook} becomes
\begin{align*}
    \lambda^{-2\chi} w &\geq w + (2N-1)w + (\lambda^{-\chi}w-w) + (\lambda^{-\frac{4}{3}\chi}w-w) \\
    N &\leq \frac{1}{2} (\lambda^{-2\chi}-\lambda^{-\frac{4}{3}\chi}-\lambda^{-\chi})+1
\end{align*}

If $v^1_{\delta_1}=v^2_{\delta_2}$ does not equal to the bottom vertex of $s$, then we can improve the estimates for the same two vertices. Namely, we first replace the contribution of $u_2$, from $w$ to $\lambda^{-\chi}w$, with the same justification as in the last case. For $u$, we can actually replace $w$ by $\lambda^{-\frac{3}{2}\chi}w$. This is because our equations regarding the $p^\beta_k$ now becomes
\begin{itemize}
    \item $\sum_{k=1}^{\delta_1+1} p^1_k = \sum_{k=1}^{\delta_2+1} p^2_k$
    \item $\sum_{k=1}^{\delta_1+1} p^1_k + \sum_{k=1}^{\delta_2+1} p^2_k \leq -\chi$
\end{itemize}
so the number of times the hook $h_1$ intersects the fiber surface is $\leq \sum_{k=1}^{\delta_1+1} p^1_k \leq -\frac{1}{2} \chi$. 

With these improvements, \Cref{eq:singlehook} becomes
\begin{align*}
    \lambda^{-2\chi} w &\geq w + (2N-1)w + (\lambda^{-\chi}w-w) + (\lambda^{-\frac{3}{2}\chi}w-w) \\
    N &\leq \frac{1}{2} (\lambda^{-2\chi}-\lambda^{-\frac{3}{2}\chi}-\lambda^{-\chi})+1 \leq \frac{1}{2} (\lambda^{-2\chi}-\lambda^{-\frac{4}{3}\chi}-\lambda^{-\chi})+1
    \qedhere
\end{align*}
\end{proof}

As before, we summarize the argument in \Cref{prop:fewbranchcyclefrc} using \Cref{tab:fewbranchcyclefrc}.

\begin{table}[ht]
    \centering
    \caption{The argument in \Cref{prop:fewbranchcyclefrc}}
    \begin{tabular}{|C{6cm}|C{6cm}|C{3cm}|}
    \hline
    Vertices of $-c$ & Quantity & Contribution \\
    \hline \hline
    Last vertex on $c_2$ & 1 & $\lambda^{-\chi}w$ \\
    \hline
    Last vertex & 1 & $\lambda^{-\frac{4}{3}\chi}w$ \\
    \hline
    Pairs of vertices that meet a B-resolved hook vertex & $\text{\# hook vertices}-1$ & $2w$ \\
    \hline   
    Remaining non-hook vertices & $2N-2\text{(\# hook vertices)}-1$ & $w$ \\
    \hline
    \end{tabular}
    \label{tab:fewbranchcyclefrc}
\end{table}

\subsection{When the sector is fan} \label{subsec:oneboundaryfan}

In this subsection, we explain some arguments that work when $s$ is fan.
As in the last subsection, let $s$ be a sector of minimum weight. Suppose $s$ is fan. Let $c$ be an Eulerian circuit that hooks around $s$, say $c$ contains the hook $h_1$. 
We use the same notation on the edges, vertices, and sectors adjacent to $s$ as always.

Suppose that $s=s^1_1$, then by \Cref{prop:vbscontradictions}, $e^1_1=e^1_2$. If $e^1_1=e^1_2$ does not intersect the fiber surface, then $(e^1_1)$ is a cycle of $\Gamma$ that does not intersect the fiber surface, contradicting \Cref{prop:nonullcycles}. So $e^1_1=e^1_2$ intersects the fiber surface, and the hook $h_1$ is deep. But then $\Gamma(h_1)$ will not be connected, since $(e^1_1)$ is its own component, contradicting our assumption that the hook circuit $c$ exists.

Now suppose that $s=s^2_1$, then by \Cref{prop:vbscontradictions}, $e^2_1=e^2_2$. Notice that in this case, $h'_2$ is empty, so by \Cref{lemma:doublehookoneshortcycle}, \Cref{prop:doublehookresolution}, and \Cref{prop:doublehookbound}, we know that the number of tetrahedra in $\Delta$ is $\leq \frac{1}{4} \lambda^{-2\chi}+1$.
We record this reasoning as a proposition.

\begin{prop} \label{prop:oneboundaryfannonemb}
Let $s$ be a minimum weight sector. Suppose $s$ is fan and suppose there is an Eulerian circuit $c$ that hooks around $s$. If $s=s^1_1$ or $s^2_1$, then the number of tetrahedra in $\Delta$ is $\leq \frac{1}{4} \lambda^{-2\chi}+1$.
\end{prop}

We then have the following argument.

\begin{prop} \label{prop:oneboundaryfanemb}
Let $s$ be a sector of minimum weight. Suppose that:
\begin{itemize}
    \item $s$ is fan
    \item $s \neq s^1_1$ or $s^2_1$
    \item $s$ satisfies (FRC) but does not satisfy (TBT)
\end{itemize}
Then the number of tetrahedra in $\Delta$ is $\leq \max \{ \frac{1}{4} \lambda^{-2\chi}+1, \frac{1}{3} \lambda^{-2\chi}+\frac{1}{2} \}$.
\end{prop}
\begin{proof}
By \Cref{lemma:doubledeephookfan}, $\Gamma(h_1 \cup h_2)$ is connected unless the bottom vertex of $s$ is equal to the top vertex of $s$. When $\Gamma(h_1 \cup h_2)$ is connected, \Cref{prop:doublehookresolution} and \Cref{prop:doublehookbound} imply that the number of tetrahedra in $\Delta$ is $\leq \frac{1}{4} \lambda^{-2\chi}+1$, so we assume in the rest of the proof that the bottom vertex of $s$ is equal to the top vertex of $s$. Since $s$ does not satisfy (TBT), the identification must be such that $(e^1_{\delta_1}, e^1_1)$ takes a branching turn at $v_0$.

Let $\widehat{s_0}$ be the lift of $s$ in $\widehat{B}$ that is at height 0. Let $b$ be the number of arcs on $s$. Then the sector having the bottom vertex of $\widehat{s_0}$, which we denote by $\widehat{v_0}$, as its top vertex is $g^{-b} \widehat{s_0}$. Suppose the weight of the two fins at $\widehat{v_0}$ are $aw$ and $a'w$, with $a \geq a'$. Then we have $\lambda^b w = w + aw + a'w \leq w + 2aw$, which implies that $a \geq \frac{\lambda^b-1}{2}$. Meanwhile, the fins at $\widehat{v_0}$ are at height $\leq 0$, thus $a,a' \geq 1$, and $\lambda^b w = w + aw + a'w \geq 3w$, which implies that $\lambda^b \geq 3$.

Now notice that both $\Gamma(h_1)$ and $\Gamma(h_2)$ are connected, since otherwise $s$ satisfies (TBT) by \Cref{lemma:singledeephookfan}. We claim that both $\Gamma(h_1 \cup \{(v_0,\text{A-resolution})\})$ and $\Gamma(h_2 \cup \{(v_0,\text{A-resolution})\})$ are connected as well. This is because by following along the sides of $s$, we see that the $v_0$ only meets one component of $\Gamma(h_\beta \cup \{(v_0,\text{A-resolution})\})$, so the additional A-resolution at $v_0$ cannot disconnect $\Gamma(h_\beta)$.

Let $c_\beta$ be an Eulerian circuit of $\Gamma$ which is the image of an Eulerian circuit of $\Gamma(h_\beta \cup \{(v_0,\text{A-resolution})\})$. For one of the $c_\beta$, without loss of generality say $c_1$, when we apply the argument of \Cref{prop:singlehookbound}, the last term in \Cref{eq:singlehook} is the weight of a sector which is a translate of the one with weight $aw$ (as opposed to the one with weight $a'w$). See \Cref{fig:oneboundaryfanemb}.

\begin{figure}[ht]
    \centering
    \fontsize{12pt}{12pt}\selectfont
    \resizebox{!}{4cm}{
\begingroup%
  \makeatletter%
  \providecommand\color[2][]{%
    \errmessage{(Inkscape) Color is used for the text in Inkscape, but the package 'color.sty' is not loaded}%
    \renewcommand\color[2][]{}%
  }%
  \providecommand\transparent[1]{%
    \errmessage{(Inkscape) Transparency is used (non-zero) for the text in Inkscape, but the package 'transparent.sty' is not loaded}%
    \renewcommand\transparent[1]{}%
  }%
  \providecommand\rotatebox[2]{#2}%
  \newcommand*\fsize{\dimexpr\f@size pt\relax}%
  \newcommand*\lineheight[1]{\fontsize{\fsize}{#1\fsize}\selectfont}%
  \ifx\svgwidth\undefined%
    \setlength{\unitlength}{141.73228346bp}%
    \ifx\svgscale\undefined%
      \relax%
    \else%
      \setlength{\unitlength}{\unitlength * \real{\svgscale}}%
    \fi%
  \else%
    \setlength{\unitlength}{\svgwidth}%
  \fi%
  \global\let\svgwidth\undefined%
  \global\let\svgscale\undefined%
  \makeatother%
  \begin{picture}(1,1.00000008)%
    \lineheight{1}%
    \setlength\tabcolsep{0pt}%
    \put(0,0){\includegraphics[width=\unitlength,page=1]{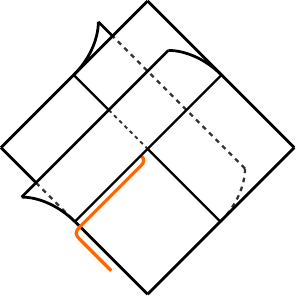}}%
    \put(0.47031062,0.223232){\color[rgb]{0,0,0}\makebox(0,0)[lt]{\lineheight{1.25}\smash{\begin{tabular}[t]{l}\La $s$\end{tabular}}}}%
    \put(0.4695341,0.72487217){\color[rgb]{0,0,0}\makebox(0,0)[lt]{\lineheight{1.25}\smash{\begin{tabular}[t]{l}\La $\s$\end{tabular}}}}%
    \put(0.24983745,0.41004642){\color[rgb]{0,0,0}\makebox(0,0)[lt]{\lineheight{1.25}\smash{\begin{tabular}[t]{l}$aw$\end{tabular}}}}%
    \put(0.63772182,0.41004642){\color[rgb]{0,0,0}\transparent{0.40000001}\makebox(0,0)[lt]{\lineheight{1.25}\smash{\begin{tabular}[t]{l}$a'w$\end{tabular}}}}%
    \put(0,0){\includegraphics[width=\unitlength,page=2]{oneboundaryfanemb.pdf}}%
  \end{picture}%
\endgroup%
}
    \caption{The last term in one of the $c_\beta$ (the orange one in the figure) will be the weight of a sector which is a translate of the one with weight $aw$ (as opposed to the one with weight $a'w$).}
    \label{fig:oneboundaryfanemb}
\end{figure}

More precisely, we can replace the contribution of the last vertex of $c_1$ from $w$ to $\lambda^{-2\chi-b} aw$. Here we use the second item in the hypothesis to ensure that this vertex is not a hook vertex. \Cref{eq:singlehook} then becomes
\begin{align*}
    \lambda^{-2\chi}w &= w + (2N-1)w + (\lambda^{-2\chi-b} aw - w) \\
    &\geq (2N-1)w + \lambda^{-2\chi} \frac{w-\lambda^{-b}w}{2} \\
    &\geq (2N-1)w + \lambda^{-2\chi} \frac{w-\frac{1}{3} w}{2} \\
    &= (2N-1)w + \frac{1}{3} \lambda^{-2\chi}w \\
    N &\leq \frac{1}{3} \lambda^{-2\chi}+\frac{1}{2}
    \qedhere
\end{align*}
\end{proof}

We summarize the main part of the proof of \Cref{prop:oneboundaryfanemb} in \Cref{tab:oneboundaryfanemb}.

\begin{table}[ht]
    \centering
    \caption{The argument in \Cref{prop:oneboundaryfanemb}.}
    \begin{tabular}{|C{6cm}|C{6cm}|C{3cm}|}
    \hline
    Vertices of $-c$ & Quantity & Contribution \\
    \hline \hline
    Last vertex & 1 & $\frac{1}{3} \lambda^{-2\chi}w$ \\
    \hline
    Pairs of vertices that meet a B-resolved hook vertex & $\text{\# hook vertices}-1$ & $2w$ \\
    \hline   
    Remaining non-hook vertices & $2N-2\text{(\# hook vertices)}$ & $w$ \\
    \hline
    \end{tabular}
    \label{tab:oneboundaryfanemb}
\end{table}

Thus we can add to our assumptions that $s$ satisfies (TBT). For the rest of this subsection, we will also assume that $h_1$ is not deep. This assumption will fit into the scheme of the proof in quite an intricate way. In this case the side of the dual edge to $s$ containing $h_1$ must be long. Depending on whether the other side is long, \Cref{prop:oneboundaryfantbtlong} and \Cref{prop:oneboundaryfantbtshort} will conclude our arguments in this subsection.

\begin{prop} \label{prop:oneboundaryfantbtlong}
Let $s$ be a sector of minimum weight. Suppose that there is an Eulerian circuit $c$ containing the hook $h_1$ of $s$, and suppose that:
\begin{itemize}
    \item $s$ is fan
    \item $s \neq s^1_1$ or $s^2_1$
    \item $s$ satisfies (FRC) and (TBT)
    \item $h_1$ is not deep
    \item Both sides of the dual edge $e$ of $s$ are long
\end{itemize}
Then the number of tetrahedra in $\Delta$ is $\leq \min\{\frac{1}{4} \lambda^{-2\chi}+1, \frac{1}{2} \lambda^{-2\chi} - \lambda^{-\chi}\}$, provided that $\lambda^{-\chi} \geq 2$
\end{prop}

Here the assumption $\lambda^{-\chi} \geq 2$ is only used to simplify the statement of \Cref{claim:oneboundaryfantbtlong}, and does not play a role in the main argument. Also, as we shall prove in \Cref{thm:introdilthm}, this hypothesis is actually always true, but here we need to include it to avoid circular reasoning.

\begin{proof}[Proof of \Cref{prop:oneboundaryfantbtlong}]
Without loss of generality suppose that $s$ is blue. If $v^1_{\delta_1}=v^2_{\delta_2}$ then by \Cref{lemma:doublehooktophingesequal}, $\Gamma(h_1 \cup h_2)$ is connected, and by \Cref{prop:doublehookresolution} and \Cref{prop:doublehookbound}, $N \leq \frac{1}{4} \lambda^{-2\chi}+1$. Similarly, if $v^1_1=v^2_1$ then by \Cref{lemma:doublehookbottomhingesequal}, \Cref{prop:doublehookresolution}, and \Cref{prop:doublehookbound}, $N \leq \frac{1}{4} \lambda^{-2\chi}+1$. Hence we can assume that $v^1_{\delta_1} \neq v^2_{\delta_2}$ and $v^1_1 \neq v^2_1$ in the rest of this proof.

We claim that $\Gamma(h_1 \cup \{(v_0,\text{B-resolution}),(v^2_1,\text{A-resolution})\})$ is connected. First notice that by following along the sides of $s$ containing $h_1$, we see that $v_0$ only meets one component of $\Gamma(h_1 \cup \{(v_0,\text{B-resolution}),(v^2_1,\text{A-resolution})\})$. Here we use the fact that $s$ satisfies (TBT) and $v^1_1 \neq v^2_1$. Hence it suffices to show that $\Gamma(h_1 \cup \{(v^2_1,\text{A-resolution})\})$ is connected.

Notice that the bottom vertex of $s^2_1$ is $v^2_{\delta_2}$ since $s$ satisfies (TBT), and this is not equal to $v^1_{\delta_1}$ by our assumption in the first paragraph. Hence if $v^2_1$ does not meet the bottom sides of $s^2_1$ then by following along the sides of $s^2_1$ we see that $v^2_1$ only meets one component of $\Gamma(h_1 \cup \{(v^2_1,\text{A-resolution})\})$.

If $v^2_1$ does meet the bottom sides of $s^2_1$ then it must do so on the bottom side other than $s^2_1$. But if the identification is such that $e^2_1$ lies on the bottom side of $s^2_1$, then we will have $s=s^2_1$, contradicting the second item in the hypothesis. See \Cref{fig:oneboundaryfantbtlong1} left. In the other manner of identification, we see that $v^2_1$ only meets one component of $\Gamma(h_1 \cup \{(v^2_1,\text{A-resolution})\})$. See \Cref{fig:oneboundaryfantbtlong1} right.

\begin{figure}[ht]
    \centering
    \fontsize{12pt}{12pt}\selectfont
    \resizebox{!}{4.5cm}{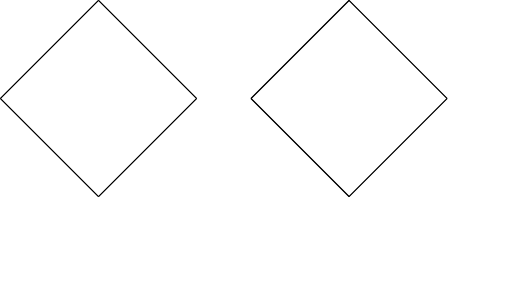}
    \caption{Reasoning that $\Gamma(h_1 \cup \{(v^2_1,\text{A-resolution})\}$ is connected even if $v^2_1$ meets the bottom sides of $s^2_1$.}
    \label{fig:oneboundaryfantbtlong1}
\end{figure}

The point of the claim is that by taking the hook circuit $c$ to be the image of an Eulerian circuit in $\Gamma(h_1 \cup \{(v_0,\text{B-resolution}),(v^2_1,\text{A-resolution})\})$, we can assume that $c$ takes a branching turn at $v_0$ and an anti-branching turn at $v^2_1$. This forces what the last two vertices of $-c$ can be.

We now let $t_i$ be the sector that has the top side of $s^i_1$ other than $e^i_1$ along its bottom side, for $i=1,2$. See \Cref{fig:oneboundaryfantbtlong2}. By the assumption that $v^1_1 \neq v^2_1$, we know that $t_1$ and $t_2$ are distinct from $s$. (But $t_1$ and $t_2$ could be equal.) Let the weight of $t_i$ be $a_i w$, and let the number of times the bottom side of $t_i$ meeting $s^i_1$ intersects the fiber surface be $p_i$. We highlighted these bottom sides in orange in \Cref{fig:oneboundaryfantbtlong2}.

\begin{figure}[ht]
    \centering
    \fontsize{12pt}{12pt}\selectfont
    \resizebox{!}{4.5cm}{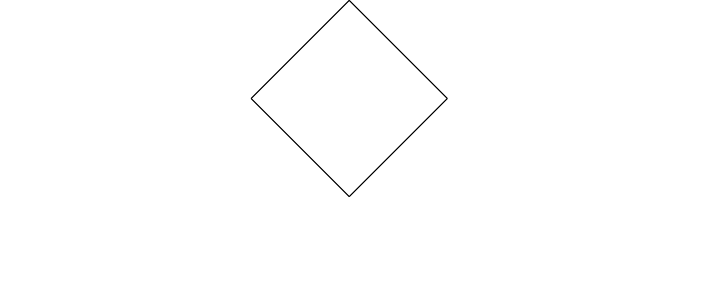}
    \caption{The set up to the main argument of \Cref{prop:oneboundaryfantbtlong}. We first argue that we can assume the ending portion of $-c$ is as indicated. Then we consider the weights of $t_i$ and the highlighted bottom sides of $t_i$ to improve the estimates in \Cref{eq:singlehook}.}
    \label{fig:oneboundaryfantbtlong2}
\end{figure}

We first claim that $\lambda^{p_i} \geq \frac{2}{a_i}$. This can be shown by taking the height $0$ lift of $t_i$, pushing the bottom side we highlighted upwards on the side of $t_i$ and reversing its orientation to get a descending path. The starting point of the path is on a sector of height $p_i$ hence of weight $\geq \lambda^{-p_i} w$ while its ending point is on the height $0$ lift of $t_i$ which is of weight $a_i w$, and the path intersects the branch locus once, where a sector of height $\leq p_i$ hence of weight $\geq \lambda^{-p_i} w$ merges in. This gives the equation $a_i w \geq 2 \lambda^{-p_i} w$ which implies the claim.

Let $b$ be the number of arcs on $s$. We claim that $2b+p_1+p_2 \leq -2\chi$. This is because the bottom sides of the $t_i$ meeting $s^i_1$ are disjoint from the sides of $s$, and the dual graph $\Gamma$ intersects the fiber surface for a total of $-2\chi$ times. Together with the previous claim, this implies that $\lambda^b \leq \lambda^{-\chi} \frac{\sqrt{a_1a_2}}{2}$. When $a_1, a_2 >2$, this bound is actually not ideal, since we can just use the fact that $2b \leq 2b+p_1+p_2 \leq -2\chi$ to write $\lambda^b \leq \lambda^{-\chi}$. So we combine the two inequalities to get $\lambda^b \leq \min \{\lambda^{-\chi} \frac{\sqrt{a_1a_2}}{2}, \lambda^{-\chi}\}$.

Now let $u_i$ be the vertex of the Eulerian hook circuit $c$ sitting at $s^i_1$ for which $t_i$ merges in at the corresponding intersection point with the branch locus. The $u_i$ cannot meet hook vertices since they are the top vertices of blue toggle sectors and $h_1$ is not deep, thus we can improve the contribution of $u_i$ in \Cref{eq:singlehook} from $w$ to $a_i w$.

Meanwhile let $u$ and $u'$ be the last and second-to-last vertex of $-c$ respectively. $u$ is blue hence different from $u_1,u_2,u'$, while $u'$ is distinct from $u_1,u_2$ since we assumed that $v^1_1 \neq v^2_1$ and $c$ takes an anti-branching turn at $v^2_1$. $u$ does not meet a hook vertex by the second item in the hypothesis and $u'$ does not meet a hook vertex since $h_1$ is not deep. Hence their original contributions to \Cref{eq:singlehook} are $w$.

We replace the contribution of $u$ by $\lambda^{-2\chi-b} w$ since the corresponding intersection point has height $-2\chi-b$, and also replace the contribution of $u'$ by $\lambda^{-2\chi-b} w$ since the sector that merges in at the corresponding intersection point is the lift of $s$ at height $-2\chi-b$.

With these modifications, \Cref{eq:singlehook} now reads
\begin{align*}
    \lambda^{-2\chi} &\geq 2N + (a_1-1) + (a_2-1) + 2(\lambda^{-2\chi-b}-1) \\
    &\geq 2N-4 + a_1 + a_2 + 2 \max\{\lambda^{-\chi} \frac{2}{\sqrt{a_1a_2}}, \lambda^{-\chi}\}
\end{align*}

\begin{claim} \label{claim:oneboundaryfantbtlong}
If $\lambda^{-\chi} \geq 2$, the minimum of $2N-4 + a_1 + a_2 + 2 \max\{\lambda^{-\chi} \frac{2}{\sqrt{a_1a_2}}, \lambda^{-\chi}\}$ over $a_1, a_2 \geq 1$ is $2N + 2 \lambda^{-\chi}$.
\end{claim}

\Cref{claim:oneboundaryfantbtlong} will be shown in \Cref{sec:calculus}.
This implies that
\begin{align*}
    \lambda^{-2\chi} &\geq 2N + 2 \lambda^{-\chi} \\
    N &\leq \frac{1}{2}\lambda^{-2\chi} -\lambda^{-\chi}.
    \qedhere
\end{align*}
\end{proof}

We summarize the main argument of \Cref{prop:oneboundaryfantbtlong} in \Cref{tab:oneboundaryfantbtlong}.

\begin{table}[ht]
    \centering
    \caption{The argument in \Cref{prop:oneboundaryfantbtlong}}
    \begin{tabular}{|C{6cm}|C{6cm}|C{3cm}|}
    \hline
    Vertices of $-c$ & Quantity & Contribution \\
    \hline \hline
    Vertex at $v^i_1$ where $t_i$ merges in & 2 & $a_i w$ \\
    \hline
    Second-to-last vertex & 1 & $\lambda^{-2\chi-b} w$ \\
    \hline
    Last vertex & 1 & $\lambda^{-2\chi-b} w$ \\
    \hline
    Pairs of vertices that meet a B-resolved hook vertex & $\text{\# hook vertices}-1$ & $2w$ \\
    \hline   
    Remaining non-hook vertices & $2N-2\text{(\# hook vertices)}-3$ & $w$ \\
    \hline
    \end{tabular}
    \label{tab:oneboundaryfantbtlong}
\end{table}

\begin{prop} \label{prop:oneboundaryfantbtshort}
Let $s$ be a sector of minimum weight. Suppose that there is an Eulerian circuit $c$ containing the hook $h_1$ of $s$, and suppose that:
\begin{itemize}
    \item $s$ is fan
    \item $s \neq s^1_1$ or $s^2_1$
    \item $s$ satisfies (FRC) and (TBT)
    \item $h_1$ is not deep
    \item One side of the dual edge $e$ of $s$ (necessarily the $\beta=2$ side by the item above) is short
\end{itemize}
Then the number of tetrahedra in $\Delta$ is $\leq \min\{ \frac{1}{2}(\lambda^{-2\chi}-\lambda^{-\frac{4}{3}\chi}-\lambda^{-\frac{2}{3}\chi})+1, \frac{1}{6 \sqrt{3}} \lambda^{-2\chi} +\frac{3}{2} \}$.
\end{prop}
\begin{proof}
Let $f$ be the fan sector that has $e^2_1$ as a top edge. Let $f^2_2$ be the other top edge of $f$ and let $f^2_1$ be the bottom side of $f$ below $f^2_2$. See \Cref{fig:oneboundaryfantbtshort} left. By the second item in the hypothesis, $f \neq s$. Also notice that $s$ satisfying (TBT) forces $f$ to satisfy (TBT) as well, in particular $(f^2_i)_{i \in \mathbb{Z}/2}$ is a $\Gamma$-cycle. The proof is divided into two cases depending on whether $f^2_1 = f^2_2$.

\begin{figure}[ht]
    \centering
    \fontsize{12pt}{12pt}\selectfont
    \resizebox{!}{4cm}{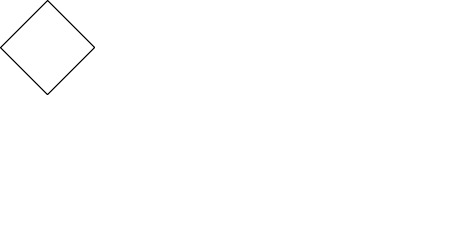}
    \caption{The set up in \Cref{prop:oneboundaryfantbtshort}. Left: when $f^2_1 \neq f^2_2$. Right: when $f^2_1 = f^2_2$.}
    \label{fig:oneboundaryfantbtshort}
\end{figure}

Let us first deal with the case when $f^2_1 \neq f^2_2$. Let $b$ be the number of arcs on $s$. This is equal to the intersection number of the cycle $(e^2_i)_{i \in \mathbb{Z}/2}$ with $S$. But $(e^2_i)$ is homotopic to $(f^2_i)$ and $(e^1_i)_{i \in \mathbb{Z}/{\delta_1+1}}$, and under our assumptions, these cycles are simple and disjoint, so $3b \leq -2\chi$. 

Let $u$ and $u'$ be the last vertex of $-c$ and the last vertex of $-c$ on $(f^2_i)$ respectively. These are blue hence do not meet B-resolved hook vertices, so their original contributions to \Cref{eq:singlehook} are both $w$. We replace the contribution of $u$ by $\lambda^{-2\chi-b} w$ since the corresponding intersection point has height $-2\chi-b$, and replace the contribution of $u'$ by $\lambda^b w$ since at that point $-c$ would have traversed through $(f^2_i)$. \Cref{eq:singlehook} then becomes 
\begin{align*}
    \lambda^{-2\chi} &\geq 2N + (\lambda^b -1) + (\lambda^{-2\chi-b} -1) \\
    &\geq 2N -2 + \lambda^{-\frac{4}{3}\chi} + \lambda^{-\frac{2}{3}\chi}
\end{align*}
Here we use the fact that $\lambda^b + \lambda^{-2\chi-b}$ is decreasing on $b \in [0,-\chi]$.
Hence we have $$N \leq \frac{1}{2}(\lambda^{-2\chi} - \lambda^{-\frac{4}{3}\chi} - \lambda^{-\frac{2}{3}\chi})+1$$ 

We then deal with the case when $f^2_1 = f^2_2$. In this case, $(f^2_i)_{i \in \mathbb{Z}/2}$ is not a simple cycle so the above argument fails.
What we will do instead is modify the argument of \Cref{prop:doublehookbound}.

We first claim that $\Gamma(h_1 \cup \{(v_0,\text{A resolution}\})$ has two components. This is because $\Gamma(h_1)$ is connected and resolving an additional vertex at most creates one more component. Meanwhile $\{e^2_1,e^2_2,f^2_1\}$ forms its own component, as can be inspected from \Cref{fig:oneboundaryfantbtshort} right.

Let $c_1$ be a circuit that is the image of an Eulerian circuit of the component of $\Gamma(h_1 \cup \{(v,\text{A resolution}\})$ containing $h_1$. Let $b$ be the number of arcs on $s$ again. We claim that \Cref{eq:doublehook} applied to $c_1$ reads
$$\lambda^{-2\chi - \frac{3}{2} b} w \geq (2N -3)w$$ 
This is because there are $2N-3$ vertices along $c_1$ (missing the 3 in the other component), $c_1$ intersects the fiber surface $-2\chi - \frac{3}{2} b$ times (missing the $\frac{3}{2} b$ times in the other component, noting that two times $(f^2_1)$ is homotopic to $(e^2_i)$), and all the hook vertices of $h_1$ lie in the same component as $c_1$ (so that the pairing trick works).

Meanwhile we have $\lambda^b \geq 3$ by the same argument as in \Cref{prop:oneboundaryfanemb}, so putting these together, we have $\lambda^{-2\chi} \geq 3 \sqrt{3} (2N-3)$ which implies the second bound in the proposition.
\end{proof}

We summarize the argument in the first case of \Cref{prop:oneboundaryfantbtshort} in \Cref{tab:oneboundaryfantbtshort}.

\begin{table}[ht]
    \centering
    \caption{The argument in the first case of \Cref{prop:oneboundaryfantbtshort}}
    \begin{tabular}{|C{6cm}|C{6cm}|C{3cm}|}
    \hline
    Vertices of $-c$ & Quantity & Contribution \\
    \hline \hline
    Last vertex on $(f^2_i)_{i \in \mathbb{Z}/2}$ & 1 & $\lambda^b w$ \\
    \hline
    Last vertex & 1 & $\lambda^{-2\chi-b} w$ \\
    \hline
    Pairs of vertices that meet a B-resolved hook vertex & $\text{\# hook vertices}-1$ & $2w$ \\
    \hline   
    Remaining non-hook vertices & $2N-2\text{(\# hook vertices)}-1$ & $w$ \\
    \hline
    \end{tabular}
    \label{tab:oneboundaryfantbtshort}
\end{table}

\subsection{When the sector is toggle} \label{subsec:oneboundarytoggle}

In this subsection, we lay out the final propositions we need. These will concern cases when $s$ is toggle.

\begin{prop} \label{prop:oneboundarytoggleemb}
Let $s$ be a sector of minimum weight. Suppose that there is an Eulerian circuit $c$ containing the hook $h_1$ of $s$, and suppose that:
\begin{itemize}
    \item $s$ is toggle
    \item $s$ satisfies (FRC)
    \item $s^1_1 \neq s \neq s^2_1$
\end{itemize}
Then the number of tetrahedra in $\Delta$ is $\leq \min\{\frac{1}{4} \lambda^{-2\chi}+1, \frac{1}{2} \lambda^{-2\chi} - \sqrt{\lambda^{-2\chi}+4\lambda^{-\chi}} +2 \}$.
\end{prop}
\begin{proof}
This proof is morally similar to \Cref{prop:oneboundaryfantbtlong}. Let $t_i$ be the sector that has $e^i_{\delta_i+1}$ along its bottom side. By the third item in the hypothesis, $t_1$ and $t_2$ are distinct from $s$.
If $t_1=t_2$, then $v^1_{\delta_1}=v^2_{\delta_2}$. By \Cref{lemma:doublehooktophingesequal}, $\Gamma(h_1 \cup h_2)$ is connected, hence \Cref{prop:doublehookresolution} and \Cref{prop:doublehookbound} implies the first bound in the proposition.

If $t_1 \neq t_2$. Let the weight of $t_i$ be $a_i w$, and let $p_i$ be the number of times the bottom side of $t_i$ not meeting $s$ meets the fiber surface. $\lambda^{p_i} \geq \frac{2}{a_i}$ by the same argument as in \Cref{prop:oneboundaryfantbtlong}. 

\begin{figure}
    \centering
    \resizebox{!}{4cm}{
\begingroup%
  \makeatletter%
  \providecommand\color[2][]{%
    \errmessage{(Inkscape) Color is used for the text in Inkscape, but the package 'color.sty' is not loaded}%
    \renewcommand\color[2][]{}%
  }%
  \providecommand\transparent[1]{%
    \errmessage{(Inkscape) Transparency is used (non-zero) for the text in Inkscape, but the package 'transparent.sty' is not loaded}%
    \renewcommand\transparent[1]{}%
  }%
  \providecommand\rotatebox[2]{#2}%
  \newcommand*\fsize{\dimexpr\f@size pt\relax}%
  \newcommand*\lineheight[1]{\fontsize{\fsize}{#1\fsize}\selectfont}%
  \ifx\svgwidth\undefined%
    \setlength{\unitlength}{425.25356485bp}%
    \ifx\svgscale\undefined%
      \relax%
    \else%
      \setlength{\unitlength}{\unitlength * \real{\svgscale}}%
    \fi%
  \else%
    \setlength{\unitlength}{\svgwidth}%
  \fi%
  \global\let\svgwidth\undefined%
  \global\let\svgscale\undefined%
  \makeatother%
  \begin{picture}(1,0.47065402)%
    \lineheight{1}%
    \setlength\tabcolsep{0pt}%
    \put(0,0){\includegraphics[width=\unitlength,page=1]{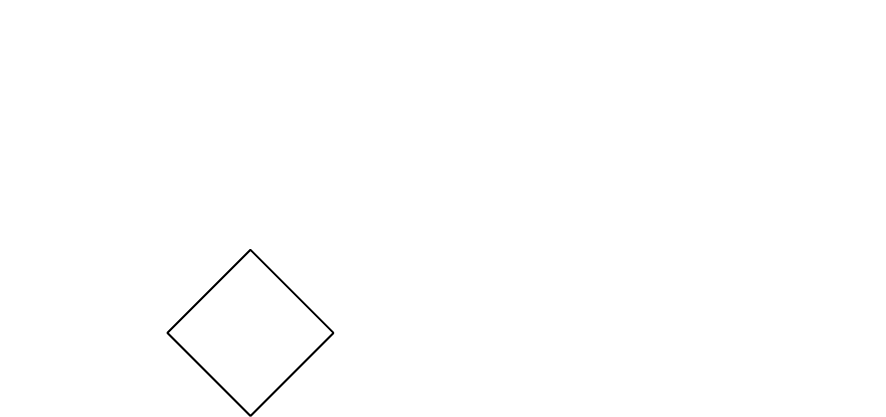}}%
    \put(0.26632515,0.07716089){\color[rgb]{0,0,0}\makebox(0,0)[lt]{\lineheight{1.25}\smash{\begin{tabular}[t]{l}\Huge $s$\end{tabular}}}}%
    \put(0,0){\includegraphics[width=\unitlength,page=2]{oneboundarytoggleemb.pdf}}%
    \put(0.17167543,0.2784045){\color[rgb]{0,0,0}\makebox(0,0)[lt]{\lineheight{1.25}\smash{\begin{tabular}[t]{l}\Huge $t_1$\end{tabular}}}}%
    \put(0,0){\includegraphics[width=\unitlength,page=3]{oneboundarytoggleemb.pdf}}%
    \put(0.61734543,0.0771343){\color[rgb]{0,0,0}\makebox(0,0)[lt]{\lineheight{1.25}\smash{\begin{tabular}[t]{l}\Huge $s$\end{tabular}}}}%
    \put(0.70458803,0.2783778){\color[rgb]{0,0,0}\makebox(0,0)[lt]{\lineheight{1.25}\smash{\begin{tabular}[t]{l}\Huge $t_2$\end{tabular}}}}%
  \end{picture}%
\endgroup%
}
    \caption{The set up in \Cref{prop:oneboundarytoggleemb}. We consider the weights of $t_i$ and the highlighted bottom sides of $t_i$ to improve the estimates in \Cref{eq:singlehook}.}
    \label{fig:oneboundarytoggleemb}
\end{figure}

Let $b$ be the number of arcs in $s$. By the same reason as in \Cref{prop:oneboundaryfantbtlong}, $2b+p_1+p_2 \leq -2\chi$, which implies that $\lambda^b \leq \min\{\lambda^{-\chi} \frac{\sqrt{a_1a_2}}{2}, \lambda^{-\chi} \}$. 

Let $u_i$ be the vertices of the Eulerian hook circuit $c$ sitting at the top vertex of $s$, for which $t_i$ merges in at the corresponding intersection point with the branch locus. By the third item in the hypothesis, $u_i$ do not meet hook vertices, nor are they the last vertex of $c$ on $h_2$. However, notice that one of $u_i$ is the last vertex of $-c$.
For the sake of concreteness we assume that $u_2$ is the last vertex of $-c$. Strictly speaking $u_1$ and $u_2$ do not have symmetric roles here, since we have broken the symmetry by assuming that $c$ contains the hook $h_1$, but the reader can check that this will not matter for the bounds we use below.

Let $u'$ be the last vertex of $-c$ on $h_2$. We modify \Cref{eq:singlehook} by replacing the contributions of $u_1$, $u_2$, and $u'$. For $u_1$, we replace $w$ with $a_1w$. For $u_2$, we replace $w$ with $a_2 \lambda^{-2\chi-b}w$. For $u'$, we replace $w$ with $\lambda^b w$. Then \Cref{eq:singlehook} becomes
\begin{align*}
    \lambda^{-2\chi} &\geq 2N + (a_1-1) + (a_2 \lambda^{-2\chi-b}-1) + (\lambda^b-1) \\
    &\geq 2N + (a_1-1) + (a_2-1) + (\lambda^{-2\chi-b}-1) + (\lambda^b-1) \\
    &= 2N-4 + a_1 + a_2 + \lambda^{-2\chi-b} + \lambda^b
\end{align*}

Here we use the fact that $a_2,\lambda^{-2\chi-b} \geq 1$ thus $(a_2 \lambda^{-2\chi-b}-1) \geq (a_2-1) + (\lambda^{-2\chi-b}-1)$. As in \Cref{prop:oneboundaryfantbtshort}, this last expression is minimized when $b=\min\{\lambda^{-\chi} \frac{\sqrt{a_1a_2}}{2}, \lambda^{-\chi} \}$, since $\lambda^{-2\chi-b} + \lambda^b$ is decreasing for $b \in [0,-\chi]$.

\begin{claim} \label{claim:oneboundarytoggleemb}
The minimum of $2N-4 + a_1 + a_2 + \lambda^{-2\chi-b} + \lambda^b$, where $b=\min\{\lambda^{-\chi} \frac{\sqrt{a_1a_2}}{2}, \lambda^{-\chi} \}$, over $a_1, a_2 \geq 1$ is $2N-4+2\sqrt{\lambda^{-2\chi}+4\lambda^{-\chi}}$.
\end{claim}

\Cref{claim:oneboundarytoggleemb} will be shown in \Cref{sec:calculus}.

This implies that 
\begin{align*}
    \lambda^{-2\chi} &\geq 2N-4 + 2\sqrt{\lambda^{-2\chi}+4\lambda^{-\chi}} \\
    N &\leq \frac{1}{2} \lambda^{-2\chi} - \sqrt{\lambda^{-2\chi}+4\lambda^{-\chi}} +2
    \qedhere
\end{align*}
\end{proof}

We summarize the argument of \Cref{prop:oneboundarytoggleemb} in \Cref{tab:oneboundarytoggleemb}.

\begin{table}[ht]
    \centering
    \caption{The argument in \Cref{prop:oneboundarytoggleemb}}
    \begin{tabular}{|C{6cm}|C{6cm}|C{3cm}|}
    \hline
    Vertices of $-c$ & Quantity & Contribution \\
    \hline \hline
    First vertex at top vertex of $s$ & 1 & $a_1 w$ \\
    \hline
    Second vertex at top vertex of $s$ = Last vertex & 1 & $a_2 \lambda^{2-b}w$ \\
    \hline
    Last vertex of $c$ on $h_2$ & 1 & $\lambda^b w$ \\
    \hline
    Pairs of vertices that meet a B-resolved hook vertex & $\text{\# hook vertices}-1$ & $2w$ \\
    \hline   
    Remaining non-hook vertices & $2N-2\text{(\# hook vertices)}-1$ & $w$ \\
    \hline
    \end{tabular}
    \label{tab:oneboundarytoggleemb}
\end{table}

Similar to \Cref{subsec:oneboundaryfan}, for the rest of the propositions we will assume that $h_1$ is not deep. Depending on whether $e^1_{\delta_1+1}=e^1_1$ and $e^2_{\delta_2+1}=e^2_1$, \Cref{prop:oneboundarytoggleonecycle} and \Cref{prop:oneboundarytogglesbf} will conclude our arguments in this subsection.

\begin{prop} \label{prop:oneboundarytoggleonecycle}
Let $s$ be a sector of minimum weight. Suppose that there is an Eulerian circuit $c$ containing the hook $h_1$ of $s$, and suppose that:
\begin{itemize}
    \item $s$ is toggle
    \item $s$ satisfies (FRC)
    \item $e^2_{\delta_2+1}=e^2_1$
    \item $h_1$ is not deep
\end{itemize}
Then the number of tetrahedra in $\Delta$ is $\leq \frac{1}{4} \lambda^{-2\chi}+1$.
\end{prop}
\begin{proof}
Under the hypothesis, $\Gamma(h_1 \cup h_2)$ satisfies the hypothesis of \Cref{prop:doublehookresolution} by \Cref{lemma:doublehookonecycleothernondeep}, so the bound follows from \Cref{prop:doublehookbound}.
\end{proof}

We come to the final proposition, whose proof contains the most modifications to \Cref{prop:singlehookbound}. To state the bounds in the proposition we need to define some auxiliary functions.

Let $F_1(x)$ be the maximum of 
$$f_1(x,u) = \frac{1}{2} x^2 - \frac{1}{2} x (u+ u^{-1}) - (\frac{3}{2})^{\frac{4}{3}} u^{\frac{2}{3}} +2 -\frac{1}{2x}$$
over $0 < u \leq 1$. 

Let $F_2(x)$ be the maximum of 
$$f_2(x,a) = \frac{1}{2} x^2 - \frac{1}{2} x (\sqrt{\frac{a}{a+1}} + \sqrt{\frac{a+1}{a}}) - \frac{1}{2} a - a^{-1} +2 -\frac{1}{2x}$$
over $a \geq 1$.

\begin{prop} \label{prop:oneboundarytogglesbf}
Let $s$ be a sector of minimum weight. Suppose that there is an Eulerian circuit $c$ containing the hook $h_1$ of $s$, and suppose that:
\begin{itemize}
    \item $s$ is toggle
    \item $s$ satisfies (FRC)
    \item $e^1_{\delta_1+1}=e^1_1$ but $s \neq s^2_1$
\end{itemize}
Then the number of tetrahedra in $\Delta$ is $\leq \max \{\frac{1}{4} \lambda^{-2\chi}+1, F_1(\lambda^{-\chi}), F_2(\lambda^{-\chi}) \}$, provided that $\lambda^{-\chi} \geq 4\sqrt{2}$.
\end{prop}

Like \Cref{prop:oneboundaryfantbtlong}, the assumption $\lambda^{-\chi} \geq 4\sqrt{2}$ is used to simplify the statement and does not play a role in the main argument. Here this hypothesis is nontrivial, but it will not matter for our application to the fully-punctured normalized dilatation problem.

\begin{proof}[Proof of \Cref{prop:oneboundarytogglesbf}]
Notice that $e^1_{\delta_1+1}=e^1_1$ implies that $h_1$ is not deep, otherwise $\Gamma(h_1)$ cannot be connected. Since $\Gamma(e^1_2,...,e^1_{\delta_1+1})$ is connected by \Cref{lemma:singlenondeephook}, we can assume that $-c$ starts with $(-e^1_{\delta_1+1},...,-e^1_2)$. Then $-c$ has to take an anti-branching turn at $v^1_1$ otherwise it would not be an Eulerian circuit. Write $t=s^2_1$. Let $t^2_1$ be the top edge of $t$ other than $e^2_1$, let $t^2_1$ be the bottom side of $t$ below $t^2_1$, and let $t^1_1$ be the other bottom side of $t$. 

By \Cref{lemma:doubledeephookfan}, $\Gamma(h_1 \cup h_2)$ is connected unless $e^1_{\delta_1+1}$ or $e^2_{\delta_2+1}$ lies along $t^1_1$. If $e^1_{\delta_1+1}$ lies along $t^1_1$, then by the hypothesis that $e^1_{\delta_1+1}=e^1_1$, we have $t=s$, but this contradicts the hypothesis that $s \neq s^2_1$. So we either have the first bound in the proposition from \Cref{prop:doublehookresolution} and \Cref{prop:doublehookbound} or $e^2_{\delta_2+1}$ lies along $t^1_1$. See \Cref{fig:oneboundarytogglesbf} left, where we also indicate the initial portion of the descending path (obtained from $c$) in yellow then teal. The second and third bounds in the proposition will follow from splitting into cases when $t^2_1$ is disjoint from $t^2_2$ and when it is not.

\begin{figure}[ht]
    \centering
    \fontsize{10pt}{10pt}\selectfont
    \resizebox{!}{5cm}{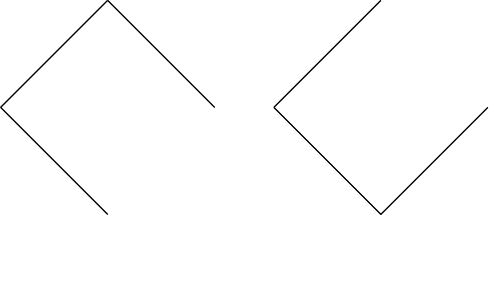}
    \caption{The set up in \Cref{prop:oneboundarytogglesbf}. We can assume that $-c$ first goes through $(-e^1_{\delta_1+1},...,-e^1_2)$ (in yellow) and takes an anti-branching turn at $v^1_1$ (in teal). Left: if $t^2_1$ is disjoint from $t^2_2$. Right: if $t^2_1 \subset t^2_2$.}
    \label{fig:oneboundarytogglesbf}
\end{figure}

We first tackle the case when $t^2_1$ is disjoint from $t^2_2$. Let $b$ be the number of arcs on $s$. Let $b'$ be the number of times $e^2_{\delta_2+1}$ intersects the fiber surface, and let $b''$ be the total number of times $e^2_i$ for $i=1,...,\delta_2$ intersects the fiber surface. Then $b=b'+b''$. Meanwhile let $q$ be the number of arcs on $t$ and let $p+b'$ be the number of times $t^1_1$ intersects the fiber surface. We label these variables on \Cref{fig:oneboundarytogglesbf} left. Finally, let $aw$ be the weight of $t$.

We have $\lambda^q \geq \frac{3}{a}$ by the same argument as in \Cref{prop:oneboundaryfanemb} but applied to $t$. Similarly, $\lambda^{p+b'} \geq \frac{2}{a}$. We claim that $2b+p+q \leq -2\chi$. This follows from the observation that $t^1_1,t^2_1,t^2_2$ are disjoint from the sides of $s$ except for $e^2_{\delta_2+1}$ lying along $t^1_1$, and the fact that $e^1_1$ does not meet the fiber surface, since $h_1$ is not deep. The claim implies that $\lambda^b \leq \min\{\lambda^{-\chi-\frac{p}{2}} \sqrt{\frac{a}{3}}, \lambda^{-\chi}\}$. We also note that we have $\lambda^{b'} \geq \lambda^{-p} \frac{2}{a}$.

Let $u_1$ be the vertex of $-c$ at the top vertex of $s$ for which $t$ merges in. Let $u_2$ be the vertex of $-c$ after it traverses $e^2_{\delta_2+1}$. In other words, $u_1$ and $u_2$ are the first two vertices of $c$ after it traverses $(-e^1_{\delta_1+1},...,-e^1_2)$. This implies that $u_1$ and $u_2$ are distinct from each other and distinct from the last vertex of $-c$ on $h_2$ and the last vertex of $-c$. Let us write $u'$ and $u$ for these last two vertices respectively. Since $t=s^2_1 \neq s$, $u'$ and $u$ are distinct, so $u_1,u_2,u',u$ are all distinct vertices of $-c$.

$u_1$ does not meet a B-resolved hook vertex since $h_1$ is not deep. $u_2$ does not meet a B-resolved hook vertex since it is blue. $u'$ does not meet a B-resolved hook vertex since it is either blue or takes an anti-branching turn. $u$ does not meet a B-resolved hook vertex since $h_1$ is not deep.

Hence we can modify \Cref{eq:singlehook} by replacing the contribution of $u_1$ from $w$ to $aw$, replacing the contribution of $u_2$ from $w$ to $\lambda^{b'} w$, replacing the contribution of $u'$ from $w$ to $\lambda^b w$, and replacing the contribution of $u$ from $w$ to $\lambda^{-2\chi-b}$. We remark that we could have replaced the contribution of $u_1$ from $w$ to $a \lambda^{b'} w$ but it turns out that does not actually buy us any advantage, and would only complicate the arithmetic below.

Finally, we also estimate the first term of \Cref{eq:singlehook}, that is, the term corresponding to the A-resolved hook vertex, by $\lambda^{-b} w$ instead of ignoring it. This estimate comes from the fact that in the proof of \Cref{prop:singlehookbound}, the vertex of $-\widehat{c}$ corresponding to the first term is at height $\leq b$.

Thus \Cref{eq:singlehook} now reads
\begin{align*}
    \lambda^{-2\chi} &\geq 2N + (a-1) + (\lambda^{b'}-1) + (\lambda^b-1) + (\lambda^{-2\chi-b}-1) + \lambda^{-b} \\
    &= 2N-4+\lambda^{-b} + \lambda^b + \lambda^{-2\chi-b} + a + \lambda^{b'} \\
    &\geq 2N-4+\lambda^\chi + \lambda^{-\chi} (\min\{\lambda^{-\frac{p}{2}} \sqrt{\frac{a}{3}},1\} + \min\{\lambda^{-\frac{p}{2}} \sqrt{\frac{a}{3}},1\}^{-1}) + a + \lambda^{-p} \frac{2}{a} \\
    N &\leq \frac{1}{2} \lambda^{-2\chi} - \frac{1}{2} \lambda^{-\chi} (\min\{\} + \min\{\}^{-1}) - \frac{1}{2}(a + \lambda^{-p} \frac{2}{a}) + 2 - \frac{\lambda^\chi}{2}
\end{align*}
where we write $\min\{\}=\min\{\lambda^{-\frac{p}{2}} \sqrt{\frac{a}{3}},1\}$ to save space.

\begin{claim} \label{claim:oneboundarytogglesbf1}
For any $a \geq 1, p \geq 0$, we have 
$$a + \lambda^{-p} \frac{2}{a} \geq 3 (\frac{3}{2})^{\frac{1}{3}} (\lambda^{-\frac{p}{2}} \sqrt{\frac{a}{3}})^{\frac{2}{3}}$$
\end{claim}

\Cref{claim:oneboundarytogglesbf1} will be shown in \Cref{sec:calculus}.

This implies that
\begin{align*}
    N &\leq \frac{1}{2} \lambda^{-2\chi} - \frac{1}{2} \lambda^{-\chi} (\min\{\} + \min\{\}^{-1}) - (\frac{3}{2})^{\frac{4}{3}} (\lambda^{-\frac{p}{2}} \sqrt{\frac{a}{3}})^{\frac{2}{3}} + 2 - \frac{\lambda^\chi}{2} \\
    &\leq \frac{1}{2} \lambda^{-2\chi} - \frac{1}{2} \lambda^{-\chi} (\min\{\} + \min\{\}^{-1}) - (\frac{3}{2})^{\frac{4}{3}} \min\{\}^{\frac{2}{3}} + 2 - \frac{\lambda^\chi}{2} \\
    &= f_1(\lambda^{-\chi}, \min\{\}) \\
    &\leq F_1(\lambda^{-\chi})
\end{align*}

Now we tackle the case when $t^2_1$ is contained in $t^1_1$. Let $b$ be the number of arcs on $s$. Let $b'$ be the number of times $e^2_{\delta_2+1}$ intersects the fiber surface, and let $b''$ be the total number of times $e^2_i$ for $i=1,...,\delta_2$ intersects the fiber surface. Then $b=b'+b''$. Meanwhile let $q$ be the times $t^2_1$ intersects the fiber surface and let $p+b'$ be the number of times $t^1_1$ intersects the fiber surface. See \Cref{fig:oneboundarytogglesbf} right. Finally, let $aw$ be the weight of $t$.

Consider $t^2_1$, push it upwards in the direction of $t$ and reverse its orientation to get a descending path. Using this path, we get the inequality $\lambda^q \geq \frac{a+1}{a}$. We reuse the same inequalities $\lambda^{p+b'} \geq \frac{2}{a}$ and $2b+p+q \leq 2$ as above. This implies that $\lambda^b \leq \lambda^{-\chi-\frac{p}{2}} \sqrt{\frac{a}{a+1}}$, and that $\lambda^{b'} \geq \lambda^{-p} \frac{2}{a}$. 

We let $u_1, u_2, u', u$ be the same vertices as above and apply the exact same modifications as above. Then \Cref{eq:singlehook} reads
\begin{align*}
    \lambda^{-2\chi} &\geq 2N + (a-1) + (\lambda^{b'}-1) + (\lambda^b-1) + (\lambda^{-2\chi-b}-1) + \lambda^{-b} \\
    &= 2N-4+\lambda^{-b} + \lambda^b + \lambda^{-2\chi-b} + a + \lambda^{b'} \\
    &\geq 2N-4+\lambda^\chi + \lambda^{-\chi} (\lambda^{-\frac{p}{2}} \sqrt{\frac{a}{a+1}} + (\lambda^{-\frac{p}{2}} \sqrt{\frac{a}{a+1}})^{-1}) + a + \lambda^{-p} \frac{2}{a} \\
    N &\leq \frac{1}{2} \lambda^{-2\chi} - \frac{1}{2} (\lambda^{-\chi} (\lambda^{-\frac{p}{2}} \sqrt{\frac{a}{a+1}} + (\lambda^{-\frac{p}{2}} \sqrt{\frac{a}{a+1}})^{-1}) + a + \lambda^{-p} \frac{2}{a}) + 2 - \frac{\lambda^\chi}{2}
\end{align*}

\begin{claim} \label{claim:oneboundarytogglesbf2}
If $\lambda^{-\chi} \geq 4 \sqrt{2}$, then 
$$\lambda^{-\chi} (\lambda^{-\frac{p}{2}} \sqrt{\frac{a}{a+1}} + (\lambda^{-\frac{p}{2}} \sqrt{\frac{a}{a+1}})^{-1}) +a + \lambda^{-p} \frac{2}{a} \geq \lambda^{-\chi} (\sqrt{\frac{a}{a+1}} + \sqrt{\frac{a+1}{a}}) +a + \frac{2}{a}$$
for all $a \geq 1$.
\end{claim}

\Cref{claim:oneboundarytogglesbf2} will be shown in \Cref{sec:calculus}.

This implies that
\begin{align*}
    N &\leq \frac{1}{2} \lambda^{-2\chi} - \frac{1}{2} (\lambda^{-\chi} (\sqrt{\frac{a}{a+1}} + \sqrt{\frac{a+1}{a}}) + a + \frac{2}{a}) + 2 - \frac{\lambda^\chi}{2} \\
    &= f_2(\lambda^{-\chi}, a) \\
    &\leq F_2(\lambda^{-\chi})
    \qedhere
\end{align*}
\end{proof}

We summarize the argument of \Cref{prop:singlehookbound} in \Cref{tab:oneboundarytogglesbf}.

\begin{table}[ht]
    \centering
    \caption{The argument in \Cref{prop:oneboundarytogglesbf}}
    \begin{tabular}{|C{6cm}|C{6cm}|C{3cm}|}
    \hline
    Vertices of $-c$ & Quantity & Contribution \\
    \hline \hline
    Vertex at top vertex of $s$ where $t$ merges in & 1 & $a w$ \\
    \hline
    Vertex after $-c$ traverses $e^2_{\delta_2+1}$ & 1 & $\lambda^{b'}w$ \\
    \hline
    Last vertex on $h_2$ & 1 & $\lambda^b w$ \\
    \hline
    Last vertex & 1 & $\lambda^{-2\chi-b} w$ \\
    \hline
    Pairs of vertices that meet a B-resolved hook vertex & $\text{\# hook vertices}-1$ & $2w$ \\
    \hline   
    Remaining non-hook vertices & $2N-2\text{(\# hook vertices)}-3$ & $w$ \\
    \hline
    A-resolved hook vertex & 1 & $\lambda^{-b} w$ \\
    \hline
    \end{tabular}
    \label{tab:oneboundarytogglesbf}
\end{table}

\subsection{Putting everything together} \label{subsec:EIIRP}

By combining all the arguments we had in the previous subsections, we will prove the following theorem in this subsection.

\begin{thm} \label{thm:EIIRP}
Let $f:S \to S$ be a fully-punctured pseudo-Anosov map with normalized dilatation $\lambda^{-\chi}$. Suppose the mapping torus of $f$ has only one boundary component and suppose $4 \sqrt{2} \leq \lambda^{-\chi}<8$, then the mapping torus of $f$ admits a veering triangulation with the number of tetrahedra less than or equal to
\begin{align*}
    \max \{ & \frac{1}{3}\lambda^{-2\chi}+\frac{1}{2}, \frac{1}{2}\lambda^{-2\chi}-\lambda^{-\chi}, \frac{1}{2}(\lambda^{-2\chi} - \lambda^{-\frac{4}{3} \chi} - \lambda^{-\frac{2}{3} \chi} +3), \\
    & \frac{1}{2} \lambda^{-2\chi} - \sqrt{\lambda^{-2\chi}+4\lambda^{-\chi}}+2, F_1(\lambda^{-\chi}), F_2(\lambda^{-\chi}), 8\log_3 \lambda^{-\chi} \}
\end{align*}
\end{thm}
\begin{proof}
Suppose we are given $f$ as in the statement. We first check if the number of branch cycles $l$ is greater than $\frac{\log \lambda^{-\chi}}{\log 2}$. If so, we apply \Cref{prop:manybranchcycle}, noting that 
$$\frac{1}{4} \lambda^{-2\chi} +1 \leq \frac{1}{3}\lambda^{-2\chi}+\frac{1}{2}$$
for $\lambda^{-\chi} \geq 4\sqrt{2}$, and 
$$\frac{1}{2} (\lambda^{-2\chi} - \frac{\lambda^{-2\chi} - \lambda^{-\frac{2}{l}\chi}}{\lambda^{-\frac{2}{l}\chi} - 1} + l) \leq \frac{1}{2}(\lambda^{-2\chi} - \lambda^{-\frac{4}{3} \chi} - \lambda^{-\frac{2}{3} \chi} +3)$$
for $l > \frac{\log 4\sqrt{2}}{\log 2}=2.5$, i.e. $l \geq 3$, and we are done. If not, then by the hypothesis that $\lambda^{-\chi}<8$, we have $l=1$ or $2$.

Suppose there exists some choice of fiber surface such that there is a minimum weight sector $s$ that is not deep. Recall (from \Cref{subsec:singlehookexist}) that this means one of the hooks $h_\beta$ of $s$ is not deep. Without loss of generality we assume $h_1$ is not deep. Notice by \Cref{lemma:singlenondeephook}, $\Gamma(h_1)$ is connected, so there exists an Eulerian hook circuit containing $h_1$. If $e^1_{\delta_1+1} = e^2_1$ or $e^2_{\delta_2+1} = e^1_1$, then we apply \Cref{prop:fewbranchcyclenotranslate} and we are done. If $s$ does not satisfy (FRC), then we apply \Cref{prop:fewbranchcyclefrc}, noting that 
$$\frac{1}{2}(\lambda^{-2\chi} - \lambda^{-\frac{4}{3} \chi} - \lambda^{-\chi})+1 \leq \frac{1}{2}(\lambda^{-2\chi} - \lambda^{-\frac{4}{3} \chi} - \lambda^{-\frac{2}{3} \chi} +3)$$
and we are done. Hence we assume that $e^1_{\delta_1+1} \neq e^2_1$, $e^2_{\delta_2+1} \neq e^1_1$, and $s$ satisfies (FRC) from this point onwards.

Suppose $s$ is fan. If $s=s^1_1$ or $s^2_1$, we apply \Cref{prop:oneboundaryfannonemb} and we are done, so we assume that $s \neq s^1_1$ or $s^2_1$ from this point onwards. If $s$ does not satisfy (TBT), then we apply \Cref{prop:oneboundaryfanemb} and we are done. If $s$ satisfies (TBT), we apply \Cref{prop:oneboundaryfantbtlong} and \Cref{prop:oneboundaryfantbtshort}, noting that 
$$\frac{1}{6\sqrt{3}}\lambda^{-2\chi}+\frac{3}{2} \leq \frac{1}{3}\lambda^{-2\chi}+\frac{1}{2}$$
for $\lambda^{-\chi} \geq 4\sqrt{2}$, and we are done.

Suppose on the other hand that $s$ is toggle. If $s \neq s^1_1$ or $s^2_1$ then we apply \Cref{prop:oneboundarytoggleemb}. If $s=s^2_1$, then by the assumption that $e^1_{\delta_1+1} \neq e^2_1$, we must have $e^2_{\delta_2+1} = e^2_1$, and we can apply \Cref{prop:oneboundarytoggleonecycle}. So we can assume that $s=s^1_1$ but $s \neq s^2_1$. By the assumption that $e^2_{\delta_2+1} \neq e^1_1$, we must have $e^1_{\delta_1+1} = e^1_1$, so we can apply \Cref{prop:oneboundarytogglesbf}. 

So we can assume now that for any choice of fiber surface, every minimum weight sector is deep. For a sector $s$ of $B$, use \Cref{prop:manyfibersurfaces} to pick a fiber surface so that $s$ is the only deep sector, thus the only minimum weight sector. If $e^1_{\delta_1+1} = e^2_1$ or $e^2_{\delta_2+1} = e^1_1$, then we apply \Cref{prop:fewbranchcyclenotranslate}. If $s$ does not satisfy (FRC), then we apply \Cref{prop:fewbranchcyclefrc}. Hence we assume that $e^1_{\delta_1+1} \neq e^2_1$, $e^2_{\delta_2+1} \neq e^1_1$, and $s$ satisfies (FRC) from this point onwards.

Suppose $s$ is fan. We claim that either the theorem holds or $s$ satisfies (TBT). If neither $\Gamma(h_1)$ nor $\Gamma(h_2)$ are connected, then by \Cref{lemma:singledeephookfan}, $s$ satisfies (TBT) and we have proved the claim. Hence we can assume that one of $\Gamma(h_\beta)$, say $\Gamma(h_1)$ is connected. If $s=s^1_1$ or $s^2_1$, we apply \Cref{prop:oneboundaryfannonemb}. If $s \neq s^1_1$ or $s^2_1$ and $s$ does not satisfy (TBT), then we apply \Cref{prop:oneboundaryfanemb}. So the remaining case is if $s$ satisfies (TBT), as claimed.

Repeating this argument for all fan sectors, we can assume that all fan sectors of $B$ satisfy (TBT).

Suppose $s$ is toggle. We claim that either the theorem holds or $s$ satisfies (SBF) on some side. If neither $\Gamma(h_1)$ nor $\Gamma(h_2)$ are connected, then by \Cref{lemma:singledeephooktoggle}, $s$ satisfies (BSBF) hence (SBF) and we have proved the claim. Hence we can assume that one of $\Gamma(h_\beta)$, say $\Gamma(h_1)$ is connected. If $s \neq s^1_1$ or $s^2_1$ then we apply \Cref{prop:oneboundarytoggleemb} and the theorem holds. Hence we can assume that $s=s^1_1$ or $s^2_1$. By the assumption that $e^1_{\delta_1+1} \neq e^2_1$ and $e^2_{\delta_2+1} \neq e^1_1$, we must have $e^1_{\delta_1+1}=e^1_1$ or $e^2_{\delta_2+1}=e^2_1$ in the respective cases, that is, $s$ satisfies (SBF). 

Repeating this argument for all toggle sectors, we can assume that all toggle sectors of $B$ satisfy (SBF).
Hence the proof of \Cref{thm:EIIRP} is completed by the following proposition. 
\end{proof}

\begin{prop} \label{prop:oneboundarylogbound}
Let $\Delta$ be a veering triangulation and let $B$ be its stable branched surface. Suppose that:
\begin{itemize}
    \item For any fiber surface, every minimum weight sector is deep. 
    \item Every fan sector satisfies (TBT) and every toggle sector satisfies (SBF) on some side.
\end{itemize}
Then the number of tetrahedra in $\Delta$ is $\leq 8 \log_3 \lambda^{-\chi}$.
\end{prop}
\begin{proof}
For each blue sector $s$ of $B$, we choose a branch cycle $c_s$ as follows: 
\begin{itemize}
    \item If $s$ is fan, take $c_s=(e^\beta_i)_{i \in \mathbb{Z}/{\delta_\beta+1}}$ for some arbitrary choice of $\beta$. Since $s$ satisfies (TBT), $c_s$ is a $\Gamma$-cycle.
    \item If $s$ is toggle, $s$ satisfies (SBF) on some side, so $e^\beta_1=e^\beta_{\delta_\beta+1}$ for some $\beta$. Take $c_s=(e^\beta_i)_{i \in \mathbb{Z}/{\delta_\beta}}$. If $s$ satisfies (BSBF), we take some arbitrary choice of $\beta$.
\end{itemize} 
The $c_s$ for $s$ toggle are disjoint from each other, and also disjoint from $c_s$ for $s$ fan. The $c_s$ for $s$ fan are not necessarily disjoint, but each edge meets at most two such $c_s$. Hence if we let $p_s$ be the number of times $c_s$ intersects a fiber surface, we have $\sum_{s \text{ blue fan}} \frac{1}{2} p_s + \sum_{s \text{ blue toggle}} p_s \leq -2\chi$. 

Meanwhile, for each sector $s$, we can choose a fiber surface so that $s$ is the only deep sector by \Cref{prop:manyfibersurfaces}. By the first assumption, $s$ must be the only minimum weight sector under this choice of fiber surface. We can then bound $p_s$ as demonstrated in the previous subsections. Namely, if $s$ is fan, then we have $\lambda^{p_s} \geq 3$; if $s$ is toggle, then we have $\lambda^{p_s} \geq 2$. Combining this with the inequality we have from the last paragraph, we get $$\lambda^{-2\chi} \geq \prod_{s \text{ blue fan}} \lambda^{\frac{1}{2} p_s} \prod_{s \text{ blue toggle}} \lambda^{p_s} \geq \sqrt{3}^{\text{\# blue sectors}}$$
$$\text{\# blue sectors} \leq 4 \log_3 \lambda^{-\chi}$$

Similarly, $\text{\# red sectors} \leq 4 \log_3 \lambda^{-\chi}$, so
$$\text{\# tetrahedra}=\text{\# sectors} \leq 8 \log_3 \lambda^{-\chi} \eqno\qedhere$$
\end{proof}

In \Cref{fig:flowchart}, we provide a flowchart that illustrates the strategy of the proof of \Cref{thm:EIIRP}.

Showing \Cref{thm:intro16tet} now essentially amounts to substituting $\lambda^{-\chi} = 6.86$ in each of the bounds in \Cref{thm:EIIRP} and checking that they are all less than 17. We relegate this computation to \Cref{sec:calculus}.

\begin{landscape}

\begin{figure}
    \centering
    \tiny
    \begin{tikzpicture}[node distance=2cm] \label{flowchart}
        \tikzstyle{arrow} = [thick,->,>=stealth]
        \node (start) [rectangle, draw, minimum width=2cm, minimum height=0.8cm] {Start};
        
        \node (branchcycle) [diamond, draw, aspect=2, below of=start, yshift=0.2cm] {\# branch cycles $l$};        
        \draw [arrow] (start) -- (branchcycle);
        
        \node (large l) [rectangle, draw, minimum width=2.4cm, minimum height=0.8cm, left of=branchcycle, xshift=-2cm] {\Cref{prop:manybranchcycle}};
        \draw [arrow] (branchcycle) -- node[anchor=north] {$l \geq 3$} (large l);
        
        \node (small l) [diamond, draw, aspect=2, text width=2.4cm, below of=branchcycle, yshift=-0.4cm] {Min wt sector w/ non-deep hook?};
        \draw [arrow] (branchcycle) -- node[anchor=east] {$l \leq 2$} (small l);
        
        \node (nondeep hook) [rectangle, draw, minimum width=3cm, minimum height=0.8cm, text width=3.2cm, left of=small l, xshift=-3cm] 
        {Let $s$ be such a sector.\\
        Suppose $h_1$ is not deep.};
        \node (deep hook) [rectangle, draw, minimum width=3cm, minimum height=1cm, text width=3.2cm, right of=small l, xshift=3cm] 
        {Pick a sector $s$.\\
        Choose fiber surface s.t. $s$ is only deep sector.};
        \draw [arrow] (small l) -- node[anchor=north] {Yes} (nondeep hook);
        \draw [arrow] (small l) -- node[anchor=north] {No} (deep hook);
        
        \node (nondeep translate?) [diamond, draw, aspect=2, text width=2cm, below of=nondeep hook] {Is $e^1_{\delta_1+1}=e^2_1$ or $e^2_{\delta_2+1}=e^1_1$?};
        \draw [arrow] (nondeep hook) -- (nondeep translate?);
        
        \node (nondeep translate) [rectangle, draw, minimum width=2.4cm, minimum height=0.8cm, left of=nondeep translate?, xshift=-2cm]{\Cref{prop:fewbranchcyclenotranslate}};
        \draw [arrow] (nondeep translate?) -- node[anchor=north] {Yes} (nondeep translate);
        
        \node (nondeep FRC?) [diamond, draw, aspect=2, text width=2cm, below of=nondeep translate?, yshift=-0.5cm] {Does $s$ satisfy (FRC)?};
        \draw [arrow] (nondeep translate?) -- node[anchor=east] {No} (nondeep FRC?);
        
        \node (nondeep FRC) [rectangle, draw, minimum width=2.4cm, minimum height=0.8cm, left of=nondeep FRC?, xshift=-2cm]{\Cref{prop:fewbranchcyclefrc}};
        \draw [arrow] (nondeep FRC?) -- node[anchor=north] {Yes} (nondeep FRC);
        
        \node (nondeep fan) [rectangle, draw, minimum width=3cm, text width=4.2cm, below of=nondeep FRC?, xshift=-2.4cm, yshift=-0.4cm]
        {$s = s^1_1$ or $s^2_1$: \Cref{prop:oneboundaryfannonemb} \\
        $s$ does not satisfy (TBT): \\
        \Cref{prop:oneboundaryfanemb} \\
        $s$ satisfies (TBT): \\
        \Cref{prop:oneboundaryfantbtlong}, \Cref{prop:oneboundaryfantbtshort}}; 
        \draw [arrow] (nondeep FRC?) -- node[anchor=east] {No, $s$ fan} (nondeep fan);      
        
        \node (nondeep toggle) [rectangle, draw, minimum width=3cm, text width=4.2cm, below of=nondeep FRC?, xshift=2.4cm, yshift=-0.4cm] 
        {$s \neq s^1_1$ or $s^2_1$: \Cref{prop:oneboundarytoggleemb} \\
        $s=s^2_1$: \Cref{prop:oneboundarytoggleonecycle} \\
        $s=s^1_1$: \Cref{prop:oneboundarytogglesbf}};
        \draw [arrow] (nondeep FRC?) -- node[anchor=east] {No, $s$ toggle} (nondeep toggle);
        
        \node (deep translate?) [diamond, draw, aspect=2, text width=2cm, below of=deep hook] {Is $e^1_{\delta_1+1}=e^2_1$ or $e^2_{\delta_2+1}=e^1_1$?};
        \draw [arrow] (deep hook) -- (deep translate?);
        
        \node (deep translate) [rectangle, draw, minimum width=2.4cm, minimum height=0.8cm, left of=deep translate?, xshift=-2cm]{\Cref{prop:fewbranchcyclenotranslate}};
        \draw [arrow] (deep translate?) -- node[anchor=north] {Yes} (deep translate);
        
        \node (deep FRC?) [diamond, draw, aspect=2, text width=2cm, below of=deep translate?, yshift=-0.5cm] {Does $s$ satisfy (FRC)?};
        \draw [arrow] (deep translate?) -- node[anchor=east] {No} (deep FRC?);
        
        \node (deep FRC) [rectangle, draw, minimum width=2.4cm, minimum height=0.8cm, left of=deep FRC?, xshift=-2cm]{\Cref{prop:fewbranchcyclefrc}};
        \draw [arrow] (deep FRC?) -- node[anchor=north] {Yes} (deep FRC);
        
        \node (deep fan) [rectangle, draw, minimum width=3cm, minimum height=1cm, text width=4.2cm, below of=deep FRC?, xshift=-2.4cm, yshift=-0.2cm]
        {$s = s^1_1$ or $s^2_1$: \Cref{prop:oneboundaryfannonemb} \\
        $s$ does not satisfy (TBT): \Cref{prop:oneboundaryfanemb}}; 
        \draw [arrow] (deep FRC?) -- node[anchor=east] {No, $s$ fan} (deep fan);      
        
        \node (deep toggle) [rectangle, draw, minimum width=3cm, minimum height=0.8cm, text width=4.2cm, below of=deep FRC?, xshift=2.4cm, yshift=-0.2cm] 
        {$s \neq s^1_1$ or $s^2_1$: \Cref{prop:oneboundarytoggleemb}};
        \draw [arrow] (deep FRC?) -- node[anchor=east] {No, $s$ toggle} (deep toggle);
        
        \node (tbt) [rectangle, draw, minimum height=0.8cm, text width=4.2cm, below of=deep fan, yshift=0.2cm]{All fan sectors satisfy (TBT)};
        \draw [arrow] (deep fan) -- node[anchor=east] {$s$ satisfies (TBT)} (tbt);
        
        \node (sbf) [rectangle, draw, minimum height=0.8cm, text width=4.2cm, below of=deep toggle, yshift=0.2cm]{All toggle sectors satisfy (SBF)};
        \draw [arrow] (deep toggle) -- node[anchor=east] {$s$ satisfies (SBF)} (sbf);
        
        \node (merge) [coordinate, below of=tbt, xshift=2.4cm, yshift=1.2cm] {};
        \draw (tbt) |- (merge);
        \draw (sbf) |- (merge);
        
        \node (log) [rectangle, draw, minimum width=3cm, minimum height=0.8cm, below of=merge, yshift=1.2cm]{\Cref{prop:oneboundarylogbound}};
        \draw [arrow] (merge) -- (log);
    \end{tikzpicture}
    \caption{Flowchart for \Cref{thm:EIIRP}.}
    \label{fig:flowchart}
\end{figure}

\end{landscape}

\section{Application to the fully-punctured normalized dilatation problem} \label{sec:dilatation}

In this section, we prove \Cref{thm:introdilthm}. As explained in the introduction, with \Cref{thm:intro16tet}, this amounts to running a computation on veering triangulations in the census \cite{GSS}. Our main task in this section is to explain how to run this computation using the Veering code \cite{Veering} written by Parlak, Schleimer, and Segerman, and SageMath scripts written by the author, and how to use these results to conclude \Cref{thm:introdilthm}.

We assume that the reader is familiar with Thurston-Fried fibered face theory and the Teichmüller polynomial. See \cite[Exposé 14]{FLP79} and \cite{McM00} respectively.

\subsection{Isolated points} \label{subsec:isolatedpoints}

We first classify the isolated points of $\mathcal{D}$, as well as the maps that attain such normalized dilatations. As pointed out in \Cref{thm:introHT}, there are mapping tori with Betti number $\geq 2$ and with minimum normalized dilatation given by $\mu^4$. Hence these isolated points must be strictly less than $\mu^4 \approx 6.854$. 

Let $f$ be a fully-punctured pseudo-Anosov map with normalized dilatation equal to one of these isolated points. The mapping torus $T_f$ must have Betti number $1$. By \Cref{thm:introHT}, any fully-punctured pseudo-Anosov map $f$ with normalized dilatation strictly less than $\mu^4$ must have only one punctured orbit, that is, its mapping torus $T_f$ has only one boundary component. By \Cref{thm:intro16tet}, $T_f$ has a layered veering triangulation with $\leq 16$ tetrahedra.

The task now is to go through all the layered veering triangulation with $\leq 16$ tetrahedra and with Betti number $1$, of which there are $29698$, and compute the normalized dilatations of the corresponding monodromies.
To do this we run the following code, included in the auxiliary file named \texttt{dilatation2.py}, in SageMath.

\begin{spverbatim}
sage: dilatation2.dilatation_script_betti_one()
\end{spverbatim}

This outputs a list of the $29698$ triangulations in the format

\begin{center}
    (number in census, isoSig, normalized dilatation, Euler characteristic of unique fiber)
\end{center}

This list is included as an auxiliary file named \texttt{betti\_one\_compile.txt}.

Since $\mu^4 \approx 6.854$, we look for entries of the output that have logarithm of normalized dilatation less than $6.86$. There are $18$ of these, namely:
\begin{center}
\footnotesize
\begin{tabular}{ccc}
    \texttt{cPcbbbdxm\_10} & \texttt{cPcbbbiht\_12} & \\
    \texttt{dLQacccjsnk\_200} & \texttt{dLQbccchhfo\_122} & \texttt{dLQbccchhsj\_122} \\
    \texttt{eLMkbcdddhhhdu\_1221} & \texttt{eLMkbcdddhhhml\_1221} & \texttt{eLMkbcdddhhqqa\_1220} \\
    \texttt{eLMkbcdddhhqxh\_1220} & \texttt{eLMkbcdddhxqdu\_1200} & \texttt{eLMkbcdddhxqlm\_1200} \\
    \texttt{eLPkaccddjnkaj\_2002} & \texttt{eLPkbcdddhrrcv\_1200} \\
    \texttt{fLLQcbeddeehhnkhh\_21112} & \texttt{fLMPcbcdeeehhhhkn\_12211} & \texttt{fLMPcbcdeeehhhhvc\_12211} \\
    \texttt{gLMzQbcdefffhhhhhit\_122112} & \texttt{gLMzQbcdefffhhhhhpe\_122112}
\end{tabular}
\normalsize
\end{center}

For each of these, we can then compute the exact normalized dilatation by calculating the Teichm\"uller polynomial as follows

\begin{spverbatim}
sage: sig1='cPcbbbdxm_10'
sage: taut_polynomial.taut_polynomial_via_fox_calculus(sig1)
sage: sig2='cPcbbbiht_12'
sage: taut_polynomial.taut_polynomial_via_fox_calculus(sig2)
...
\end{spverbatim}

and computing the largest root.

$5$ of these triangulations have normalized dilatation $\mu^4$ which is not an isolated point. The rest of them have normalized dilatation strictly less than $\mu^4$, hence do determine isolated points of $\mathcal{D}$. These $13$ triangulations and their normalized dilatation are recorded in \Cref{tab:introisolateddil}. 
The descriptions of the maps in \Cref{tab:introisolateddil} follow by analyzing the structure of the veering triangulation and the topology of the underlying $3$-manifold.

\subsection{The minimum accumulation point} \label{subsec:accumpoint}

By \Cref{thm:introHT}, to show that the minimum accumulation point of $\mathcal{D}$ is $\mu^4$, one has to show that there are no fully-punctured pseudo-Anosov maps $f$ with normalized dilatation strictly less than $\mu^4$, whose mapping torus $T_f$ has only one boundary component but has Betti number $\geq 2$.

By \Cref{thm:intro16tet}, such a mapping torus $T_f$ would have a layered veering triangulation with $\leq 16$ tetrahedra. So the task now is to go through all the layered veering triangulations with $\leq 16$ tetrahedra and with one boundary component and Betti number $\geq 2$, of which there are $381$, and compute the minimum normalized dilatations of the corresponding monodromies.

We wrote two scripts, \texttt{dilatation\_betti\_two\_fibred} and \texttt{dilatation\_betti\_two\_fibred\_eucl}, again included in the auxiliary file \texttt{dilatation2.py}, to carry out the computation for the $374$ triangulations among the $381$ that have $b_1=2$.
The first script is in general faster but fails on a handful of triangulations; the second script is used for those outlying triangulations. See \Cref{sec:code} for an explanation of what it means for the first script to fail, and how the two scripts differ.

We first run

\begin{spverbatim}
sage: dilatation2.dilatation_script_one_cusp_betti_two()
\end{spverbatim}

On the author's run of this line, the script failed on 8 triangulations:

\begin{center}
\footnotesize
\begin{tabular}{c}
    \texttt{pLLLPwLLMQQcegeehjmkonoomnnqhqxqvqcsqpqqsta\_022210001222100} \\ \texttt{pLLvLAMPPAQbefgikjjimlnnoooxxhvcqrfrhfjrmla\_211120020212120} \\
    \texttt{qLLLLwzMAAQkacfighlkmkkopnpopjkglwlfvbjkduajrc\_2002121012100202} \\
    \texttt{qLLLLzLQwMQkbegfjlimkionnnoppxxmxxmwhdsephterr\_1022101100112222} \\
    \texttt{qLLvAALzQzQkbeghfilkmlnmnpopphhxagbqqqokbjqagb\_0111022020111020} \\ \texttt{qLLvLMvzQQQkbdjgjminpkloopmopdwbwbagpadbssrjos\_2101022222110001} \\
    \texttt{qLLvMLzzAQQkbefgjkionmplnmnpphhqqaqfhxbawvbnha\_0111022001111210} \\
    \texttt{qLLvzzwPPQQkcdekjnokljmpnnopphshepahphegbgbvnn\_1222011112220200} \\
\end{tabular}
\normalsize
\end{center}

So we run 

\begin{spverbatim}
sage: sig1='pLLLPwLLMQQcegeehjmkonoomnnqhqxqvqcsqpqqsta_022210001222100'
sage: dilatation2.dilatation_betti_two_fibred_eucl(sig1)
sage: sig2='pLLvLAMPPAQbefgikjjimlnnoooxxhvcqrfrhfjrmla_211120020212120'
sage: dilatation2.dilatation_betti_two_fibred_eucl(sig2)
...
\end{spverbatim}

We compile the result of these computations as a list in the format

\begin{center}
    (number in census, isoSig, min normalized dilatation, gcd of norms of spanning rays)
\end{center}

and include this list as an auxiliary file named \texttt{one\_cusp\_betti\_two\_compile.txt}.

The smallest value for the minimum normalized dilatation among these 374 triangulations is 17.944. In particular all of them are strictly greater than $\mu^4$.

The remaining 7 triangulations out of the 381 have $b_1=3$. For these triangulations we did the computations entirely by hand. A fact that made these computations manageable was that the minimum normalized dilatation for all 7 triangulations are attained at the center of the fibered face.

Below we show the results of the computations. Similarly as above, each line records a triangulation as 

\begin{center}
    (number in census, isoSig, min normalized dilatation, gcd of norms of spanning rays)
\end{center}

\footnotesize
\begin{verbatim}
21390 ovLLLLPMQQceeekjmlimmnnllnfssfjhhshhhahhh_20110222222110 582.6871 3
21444 ovLLLMPPPQccdjfghlijnmnlmnnkqxnkavkaxhhcc_12020111111202 582.6871 3
42251 pvLLLMPzPQQcdjfghlinonolmonnkqxnkavhaxhhccv_120201111112002 1124.3809 5
66862 qLLvLQwLQPMkbefgigilnkmnnopppxxxgbrglheabnphwr_1022101010011222 1523.2123 5
80635 qLvvAMQvAQPkbhighhkjmnolmppophharrwarqqbbraxgh_2111220020111110 2867.8560 7
86454 qvLvvLPAQQQkekjinlolnpmpmopongiwwvwaoflflfipmo_2100100211112211 1523.2123 5
86954 qvvLPAMzMQMkfhfghjlmlononmpppqhqxaxaqhaqqhhxha_2100122222210102 1153.9991 4
\end{verbatim}
\normalsize

Again, all of the minimum normalized dilatations are strictly greater than $\mu^4$. 
As explained at the beginning of this subsection, this shows that $\mu^4$ is the minimum accumulation point of $\mathcal{D}$.

We now classify the fully-punctured maps $f$ that have normalized dilatation $\mu^4$. For such a map $f$ whose mapping torus has only one boundary component, the computations in this and the last subsection show that the corresponding layered veering triangulation on the mapping torus must be one of the following $5$ triangulations:

\begin{center}
\footnotesize
\begin{tabular}{cc}
    \texttt{eLMkbcdddhxqdu\_1200} & \texttt{eLMkbcdddhxqlm\_1200} \\
    \texttt{fLLQcbeddeehhnkhh\_21112} \\
    \texttt{gLMzQbcdefffhhhhhit\_122112} & \texttt{gLMzQbcdefffhhhhhpe\_122112}
\end{tabular}
\normalsize
\end{center}

The descriptions of the maps giving rise to these triangulations in \Cref{tab:introminaccumdil} follow by analyzing the structure of the veering triangulations.

For such a map $f$ whose mapping torus has at least two boundary components, the statement of \Cref{thm:introHT} shows that $f$ must be defined on a surface $S$ with $\chi(S)=-2$. There are only two such surfaces, namely the 4-punctured sphere $S_{0,4}$ and the 2-punctured torus $S_{1,2}$. 

For $S_{0,4}$, its mapping classes are well-understood. See, for example, the appendix of \cite{Gue06}. In particular, it is straightforward to check that the only pseudo-Anosov map with dilatation $\mu^2$ is the one induced by $\begin{bmatrix} 2 & 1 \\ 1 & 1 \end{bmatrix}$ as recorded in \Cref{tab:introminaccumdil}.

For $S_{1,2}$, we can fill in one of the punctures to get a map on the once-punctured torus $S_{1,1}$ with dilatation $\mu^2$ as well. The mapping classes on $S_{1,1}$ are well understood. See for example, \cite{Gue06}. In particular, it is straightforward to check that the only pseudo-Anosov maps with dilatation $\mu^2$ are the ones induced by $\begin{bmatrix} 2 & 1 \\ 1 & 1 \end{bmatrix}$ and $\begin{bmatrix} -2 & -1 \\ -1 & -1 \end{bmatrix}$.
The filled-in puncture is some fixed point of this map on $S_{1,1}$. But one can check that the map induced by $\begin{bmatrix} 2 & 1 \\ 1 & 1 \end{bmatrix}$ has no fixed points, so the map $f$ must be that induced by $\begin{bmatrix} -2 & -1 \\ -1 & -1 \end{bmatrix}$ as recorded in \Cref{tab:introminaccumdil}.

The veering triangulations associated to these two maps can be recovered from the descriptions of $S_{0,4}$- and $S_{1,1}$-bundles in \cite{Gue06}. These are as recorded in \Cref{tab:introminaccumdil}.

\section{Discussion and further questions} \label{sec:questions}

\subsection{Further questions about the set $\mathcal{D}$}

One can interpret the minimum accumulation point of $\mathcal{D}$ as the minimum element of 
$$\mathcal{D}_2 := \{\text{Normalized dilatations of fully-punctured maps $f$ with $b_1(T_f) \geq 2$}\}$$
Motivated by this, one can define 
$$\mathcal{D}_k := \{\text{Normalized dilatations of fully-punctured maps $f$ with $b_1(T_f) \geq k$}\}$$
and ask

\begin{quest} \label{quest:higherbettinumber}
What is the minimum element of $\mathcal{D}_k$ for $k \geq 3$? What are the maps that attain these normalized dilatations?
\end{quest}

One should compare \Cref{quest:higherbettinumber} with \cite[Question 8.3]{HT22}, which asks for the minimum normalized dilatations among fully-punctured maps $f$ whose mapping torus has at least $k$ boundary components.

Notice that these questions, at least in part, can be solved via the same approach of using veering triangulations, provided that one can improve \Cref{thm:introsinglehook} or improve the technology in generating census of veering triangulations.

\subsection{The golden ratio conjecture}

Given \Cref{thm:introdilthm}, one approach for proving the golden ratio conjecture (\Cref{conj:goldenratio}) is to prove the following conjecture.

\begin{conj} \label{conj:uniformsing}
There exists a sequence of pseudo-Anosov maps $f_g:S_{g,0} \to S_{g,0}$ realizing the minimum dilatations $\delta_{g,0}$ and which have a uniformly bounded number of singularities. 
\end{conj}

We remark that the examples in \cite{Hir10}, \cite{AD10}, and \cite{KT13} do have a uniformly bounded number of singularities, giving some evidence towards \Cref{conj:uniformsing}.

One can also consider using the approach of veering triangulations again. However, this approach generally becomes much weaker in the closed case. The reason is that veering triangulations can only exist on fully-punctured mapping tori, and so one has to fully puncture the pseudo-Anosov map before applying the notion. Without good knowledge of the number and types of singularities, one can in general only bound the Euler characteristic of the punctured surface by 3 times the Euler characteristic of the original closed surface, making the exponent on the bounds 3 times as worse as in \Cref{cor:singlehook}.

\subsection{Improvements on bounds}

As remarked at the start of \Cref{sec:doublehook}, \Cref{prop:doublehookbound} provides a bound better than \Cref{prop:singlehookbound} by a factor of 2. Even though we are unable to show so, we suspect that \Cref{prop:doublehookbound} can always be applied.

\begin{conj} \label{conj:doublehookexist}
For every layered veering triangulation, there exists a fiber surface such that the hypothesis of \Cref{prop:doublehookbound} is satisfied for some minimum weight sector. 
\end{conj}

Another approach to improving the bound would be to find a way to bypass the cases when we have a minimum weight sector that is not deep. As seen in \Cref{subsec:singlehookexist} and \Cref{subsec:EIIRP}, when one can assume that all minimum weight sectors are deep, one can use the flexibility granted by \Cref{prop:manyfibersurfaces} to strongly constrain the triangulation. 
That is, a positive answer to the following question would likely lead to sharper bounds.

\begin{quest} \label{quest:deepminweightsector}
Let $\Delta$ be a layered veering triangulation. Is it true that for every sector $s$, there is a fiber surface such that $s$ is a minimum weight sector?
\end{quest}

It is instructive to consider the particular case of veering triangulations on once-punctured torus bundles.
In these triangulations, the normalized dilatation grows at worse linearly in the number of tetrahedra. 
Moreover, in this case, the normalized dilatation seems to behave differently with respect to the number of fan and toggle tetrahedra; morally, it grows linearly with the former and grows exponentially with the latter. 
From this, we suspect that one can obtain bounds that treat the number of fan and toggle tetrahedra separately. These would be of a different nature than \Cref{thm:singlehook}, but they should be significantly sharper for applications.

Of course, it is interesting to know the best possible bound one can hope for at all.

\begin{quest} \label{quest:bestexponent}
What is the smallest exponent $\alpha$ such that the number of tetrahedra at worse grows as the $\alpha^{\text{th}}$ power of the normalized dilatation, across all layered veering triangulations?
\end{quest}

The triangulations on once-punctured torus bundles show that the smallest exponent is at least 1 and \Cref{thm:singlehook} shows that it is at most 2.

\appendix

\section{Calculus exercises} \label{sec:calculus}

\begin{proof}[Proof of \Cref{claim:oneboundaryfantbtlong}]
By symmetry, the minimum is attained when $a_1=a_2$, so we have to calculate the minimum of $2N-4 + 2a + 2\lambda^{-\chi} \max\{ \frac{2}{a}, 1 \}$.

Let 
$$h(a)=2N-4 + 2a + \frac{4\lambda^{-\chi}}{a}.$$ 
Then 
$$h'(a)=2-\frac{4\lambda^{-\chi}}{a^2}>0 \Leftrightarrow a>\sqrt{2\lambda^{-\chi}}.$$

If $\lambda^{-\chi} \geq 2$, then $\sqrt{2\lambda^{-\chi}} \geq 2$, which is when the second term in the maximum takes over. So the minimum of the whole expression is $2N+ 2\lambda^{-\chi}$.
\end{proof}

\begin{proof}[Proof of \Cref{claim:oneboundarytoggleemb}]
By symmetry, the minimum is attained when $a_1=a_2$, so we have to calculate the minimum of $2N-4+2a+\lambda^{-\chi} (\min\{\frac{a}{2},1 \}^{-1} + \min\{\frac{a}{2},1 \})$.

Let 
\begin{align*}
    h(a) &= 2N-4+2a+\lambda^{-\chi}(\frac{2}{a}+\frac{a}{2}) \\
    &= 2N-4+(2+\frac{\lambda^{-\chi}}{2})a+\frac{2\lambda^{-\chi}}{a}.
\end{align*}
Then 
$$h'(a)=2+\frac{\lambda^{-\chi}}{2}-\frac{2\lambda^{-\chi}}{a^2}>0 \Leftrightarrow a>\sqrt{\frac{4\lambda^{-\chi}}{\lambda^{-\chi}+4}}.$$

Now $\sqrt{\frac{4\lambda^{-\chi}}{\lambda^{-\chi}+4}}$ is always less than $2$, which is where the first term in the minimum in $b$ takes over. So the minimum of the whole expression is $2N-4+2\sqrt{\lambda^{-2\chi}+4\lambda^{-\chi}}$.
\end{proof}

\begin{proof}[Proof of \Cref{claim:oneboundarytogglesbf1}]
We perform a variable change $t=\lambda^{p}, u=\sqrt{\frac{a}{3t}}$, so that $$a + \lambda^{-p} \frac{2}{a} = 3u^2t+\frac{2}{3}u^{-2}t^{-2}.$$
Letting $h(u,t)$ be this last expression, we compute 
$$\frac{\partial h}{\partial t} = 3u^2-\frac{4}{3}u^{-2}t^{-3}>0 \Leftrightarrow t>(\frac{2}{3})^{\frac{2}{3}} u^{-\frac{4}{3}}.$$
Hence 
\begin{align*}
    h(u,t) &\geq h(u,(\frac{2}{3})^{\frac{2}{3}} u^{-\frac{4}{3}}) \\
    &= 3(\frac{3}{2})^{\frac{1}{3}} u^{\frac{2}{3}} \\
    &= 3 (\frac{3}{2})^{\frac{1}{3}} (\lambda^{-\frac{p}{2}} \sqrt{\frac{a}{3}})^{\frac{2}{3}}.
    \qedhere
\end{align*}
\end{proof}

\begin{proof}[Proof of \Cref{claim:oneboundarytogglesbf2}]
Let 
$$h(t,a)=\lambda^{-\chi} (t^{-1} \sqrt{\frac{a}{a+1}} + t \sqrt{\frac{a+1}{a}}) +a + t^{-2} \frac{2}{a}$$
where $t \geq 1$. Then 
\begin{align*}
    \frac{\partial h}{\partial t} &= \lambda^{-\chi} (-t^{-2} \sqrt{\frac{a}{a+1}} + \sqrt{\frac{a+1}{a}}) - 4t^{-3}a^{-1} \\
    &\geq \lambda^{-\chi} (-\sqrt{\frac{a}{a+1}} + \sqrt{\frac{a+1}{a}}) - 4a^{-1} \\
    &= \lambda^{-\chi} \sqrt{\frac{a}{a+1}} (-1+\frac{a+1}{a}) - 4a^{-1} \\
    &= (\lambda^{-\chi} \sqrt{\frac{a}{a+1}}-4) a^{-1} \geq 0
\end{align*}
for $a \geq 1$ if $\lambda^{-\chi} \geq 4\sqrt{2}$.

Hence
\begin{align*}
    &\lambda^{-\chi} (\lambda^{-\frac{p}{2}} \sqrt{\frac{a}{a+1}} + (\lambda^{-\frac{p}{2}} \sqrt{\frac{a}{a+1}})^{-1}) +a + \lambda^{-p} \frac{2}{a} \\
    &= h(\lambda^{\frac{p}{2}}, a) \\
    &\geq h(1,a) \\
    &= \lambda^{-\chi} (\sqrt{\frac{a}{a+1}} + \sqrt{\frac{a+1}{a}}) +a + \frac{2}{a} 
    \qedhere
\end{align*}
\end{proof}

\begin{proof}[Proof of \Cref{thm:intro16tet}]
Let $\Delta$ be the veering triangulation on the mapping torus of $f$. If $\lambda^{-\chi} \leq 4\sqrt{2}$, then by \Cref{thm:introsinglehook}, $\Delta$ has less than or equal to $\frac{1}{2} (4\sqrt{2})^2=16$ tetrahedra.
If $\lambda^{-\chi} \geq 4\sqrt{2}$, we can apply \Cref{thm:EIIRP} to $\Delta$. Our task is to show that each of the bounds in \Cref{thm:EIIRP} is strictly less than $17$ when we substitute a value of $\lambda^{-\chi}$ between $4\sqrt{2}$ and $6.86$.

We first claim that each of the bounds in \Cref{thm:EIIRP} that are not $F_i(\lambda^{-\chi})$ is an increasing function in $\lambda^{-\chi}$ for $\lambda^{-\chi} \geq 4\sqrt{2}$. This would imply that we only have to check that the bounds are strictly less $17$ when $\lambda^{-\chi}=6.86$. 

The claim is clear for $\frac{1}{3}\lambda^{-2\chi}+\frac{1}{2}$, $\frac{1}{2}\lambda^{-2\chi}-\lambda^{-\chi}$, and $8\log_3 \lambda^{-\chi}$. 

Let 
$$h_1(x)=x^2-x^{\frac{4}{3}}-x^{\frac{2}{3}}+3.$$
Then 
\begin{align*}
    h'_1(x)&=2x-\frac{4}{3}x^{\frac{1}{3}}-\frac{2}{3}x^{-\frac{1}{3}} \\
    &\geq 2x-\frac{4}{3}x-\frac{2}{3}x=0
\end{align*}
for $x \geq 1$. This shows that $\frac{1}{2}(\lambda^{-2\chi} - \lambda^{-\frac{4}{3} \chi} - \lambda^{-\frac{2}{3} \chi} +3)$ is an increasing function in $\lambda^{-\chi}$. 

Let 
$$h_2(x)=\frac{1}{2} x^2 - \sqrt{x^2+4x}+2.$$
Then 
$$h'_2(x)=x-\frac{x+2}{\sqrt{x^2+4x}} \geq 0 \Leftrightarrow x^2(x^2+4x) \geq x^2+4x+4$$
which is evidently true for $x \geq 2$. This shows that $\frac{1}{2} \lambda^{-2\chi} - \sqrt{\lambda^{-2\chi}+4\lambda^{-\chi}}+2$ is an increasing function in $\lambda^{-\chi}$.

Now substituting in $\lambda^{-\chi}=6.86$, we have
\begin{align*}
    \frac{1}{3} 6.86^2 + \frac{1}{2} &\approx 16.187 \\
    \frac{1}{2} 6.86^2 - 6.86 &\approx 16.670 \\
    \frac{1}{2}(6.86^2 - 6.86^{\frac{4}{3}} - 6.86^{\frac{2}{3}} + 3) &\approx 16.707 \\
    \frac{1}{2} 6.86^2 - \sqrt{6.86^2+4 \times 6.86}+2 &\approx 16.898 \\
    8 \log_3 6.86 &\approx 14.023.
\end{align*}

We now move on to $F_1(\lambda^{-\chi})$ and $F_2(\lambda^{-\chi})$. The strategy is the same: We first show that these are increasing functions then evaluate them at $\lambda^{-\chi}=6.86$.

Recall that $F_1(x)$ is the maximum of 
$$f_1(x,u) = \frac{1}{2} x^2 - \frac{1}{2} x (u+ u^{-1}) - (\frac{3}{2})^{\frac{4}{3}} u^{\frac{2}{3}} +2 -\frac{1}{2x}$$
over $0 < u \leq 1$. We compute 
$$\frac{\partial f_1}{\partial u} = -\frac{1}{2} x (1-u^{-2}) - (\frac{3}{2})^{\frac{1}{3}}u^{-\frac{1}{3}}.$$
Note that $\frac{\partial f_1}{\partial u}  \geq -\frac{1}{2} x + (\frac{1}{2} x - (\frac{3}{2})^{\frac{1}{3}}) u^{-2} \geq \frac{3}{2} x - 4 (\frac{3}{2})^{\frac{1}{3}} >0$ for $u \leq \frac{1}{2}$, and $\frac{\partial f_1}{\partial u}$ is negative for $u$ close to $1$, so the maximum is attained in the interior of $[\frac{1}{2},1]$. If we let $u(x)$ be the point where this maximum is attained for fixed $x$, then 
\begin{align*}
    \frac{df_1(x,u(x))}{dx} &= \frac{\partial f_1}{\partial x}+\frac{\partial f_1}{\partial u}u'(x) = \frac{\partial f_1}{\partial x} \\
    &= x-\frac{1}{2} (u+ u^{-1}) +\frac{1}{2x^2} \\
    &\geq x-\frac{5}{4} +\frac{1}{2x^2}>0
\end{align*}
for $x \geq 4\sqrt{2}$. This shows that $F_1(\lambda^{-\chi})$ is an increasing function for $\lambda^{-\chi} \geq 4\sqrt{2}$.
Using a computer algebra system, we check that $F_1(6.86) \approx 16.966$.

Similarly, recall that $F_2(x)$ is the maximum of 
$$f_2(x,a) = \frac{1}{2} x^2 - \frac{1}{2} x (\sqrt{\frac{a}{a+1}} + \sqrt{\frac{a+1}{a}}) - \frac{1}{2} a - a^{-1} +2 -\frac{1}{2x}$$
over $a \geq 1$. We compute 
$$\frac{\partial f_2}{\partial a} = - \frac{1}{2} x (\frac{1}{2a^{\frac{1}{2}}(a+1)^{\frac{3}{2}}} - \frac{1}{2a^{\frac{3}{2}}(a+1)^{\frac{1}{2}}}) - \frac{1}{2} + a^{-2} \to \frac{1}{8\sqrt{2}}x+\frac{1}{2} \geq 0$$
as $a \to 1^+$, so the maximum is attained in the interior of $(1,\infty)$.

If we let $a(x)$ be the point where this maximum is attained for fixed $x$, then 
\begin{align*}
    \frac{df_2(x,a(x))}{dx} &= \frac{\partial f_2}{\partial x}+\frac{\partial f_2}{\partial a}a'(x)=\frac{\partial f_2}{\partial x} \\
    &= x - \frac{1}{2} (\sqrt{\frac{a}{a+1}} + \sqrt{\frac{a+1}{a}}) +\frac{1}{2x^2} \\
    &\geq x - \frac{3}{2\sqrt{2}} +\frac{1}{2x^2} > 0
\end{align*}
for $x \geq 4\sqrt{2}$. This shows that $F_2(\lambda^{-\chi})$ is an increasing function for $\lambda^{-\chi} \geq 4\sqrt{2}$.
Using a computer algebra system, we check that $F_2(6.86) \approx 16.975$.
\end{proof}

\section{Explanation of code used for computation} \label{sec:code}

The scripts we use are included in \texttt{dilatation2.py} in the auxiliary files. Among these, the three main ones are \texttt{dilatation\_betti\_one\_fibred}, \texttt{dilatation\_betti\_two\_fibred}, and \texttt{dilatation\_betti\_two\_fibred\_eucl}. 

\texttt{dilatation\_betti\_one\_fibred} takes in a layered veering triangulation with $b_1=1$ and outputs the associated normalized dilatation. The workings of this script are as follows: 
\begin{itemize}
    \item It computes the Alexander polynomial and the taut polynomial of the triangulation using Fox calculus. See \cite[Proposition 5.7]{Par21}.
    \item It computes the Euler characteristic of the fiber surface as the span of the Alexander polynomial minus 1.
    \item It computes the dilatation of the monodromy as the largest root of the taut polynomial using the SageMath function \texttt{real\_roots}. This uses the fact the taut polynomial equals the Teichmüller polynomial, see \cite[Theorem 7.1]{LMT24}.
\end{itemize}

\texttt{dilatation\_betti\_two\_fibred} takes in a layered veering triangulation with $b_1=2$ and outputs the minimum normalized dilatation. The workings of this script are as follows: 
\begin{itemize}
    \item It computes the Alexander polynomial and the taut polynomial of the triangulation (as above).
    \item It computes two spanning vectors of the fibered cone.
    \item It computes the Euler characteristic of the surface corresponding to each spanning ray, using the fact that the Thurston norm equals the Alexander norm in a fibered cone, see \cite[Theorem 7.1]{McM00}. Using this information, it parametrizes the fibered face $F$ by one parameter $t$.
    \item It removes some cyclotomic factors from the taut polynomial $\Theta$ for simpler computations, and computes its derivative $\frac{\partial \Theta}{\partial t}$ along $F$.
    \item It checks whether $\frac{\partial \Theta}{\partial t}=0$ at the midpoint of $F$. \\[0.5em]
    If yes, then the minimum normalized dilatation occurs at the midpoint of $F$, so the script does the following:
    \begin{itemize}
        \item It computes the single-variable polynomial obtained by restricting the taut polynomial to the mid-ray of the fibered cone.
        \item It computes the minimum normalized dilatation as the largest root of this polynomial using the SageMath function \texttt{real\_roots}.
    \end{itemize}
    Otherwise the scripts attempts to solve the system $\begin{cases} \Theta=0 \\ \frac{\partial \Theta}{\partial t}=0 \end{cases}$ as follows:
    \begin{itemize}
        \item It performs variable changes such that $\Theta$ and $\frac{\partial \Theta}{\partial t}$ are polynomials with relatively prime exponents.
        \item It uses the SageMath function \texttt{solve} to solve the simplified system.
    \end{itemize}
\end{itemize}

As mentioned in \Cref{subsec:accumpoint}, \texttt{dilatation\_betti\_two\_fibred} works for most triangulations. The main problem with it, however, is that the SageMath function \texttt{solve} is not guaranteed to succeed; it might only simplify the system symbolically or might get stuck and show no sign of terminating.

On the author's run of the script, this happens for the $8$ triangulations mentioned in \Cref{subsec:accumpoint}. We remark that on a more powerful system, the script may terminate and succeed for some, if not all, of these triangulations.

For us to deal with these triangulations, we need a more robust way of solving the equation $\begin{cases} \Theta=0 \\ \frac{\partial \Theta}{\partial t}=0 \end{cases}$. For this we use the following simple algebraic fact.

\begin{lemma} \label{lemma:euclalgo}
Suppose $a,b \in \mathbb{R}[x,y], p, q \in \mathbb{R}[x]$. Then any root of the system $\begin{cases} a(x,y)=0 \\ b(x,y)=0 \end{cases}$ is a root of the system $\begin{cases} a(x,y)=0 \\ p(x)a(x,y)-q(x)b(x,y)=0 \end{cases}$
\end{lemma}

In the setting of the lemma, we can consider $a$ and $b$ as polynomials of $y$ with coefficients in $\mathbb{R}[x]$. By taking $p$ to be the leading coefficient of $b$ and $q$ to be the leading coefficient of $a$, $p(x)a(x,y)-q(x)b(x,y)$ will have a smaller $y$-degree and in passing from $\begin{cases} a(x,y)=0 \\ b(x,y)=0 \end{cases}$ to $\begin{cases} a(x,y)=0 \\ p(x)a(x,y)-q(x)b(x,y)=0 \end{cases}$ we have reduced the complexity of the system in terms of its total $y$-degree. Repeating this procedure inductively, we eventually arrive at a system where one equation is only a polynomial in $x$. We can then compute the roots of this polynomial, substitute these back in the polynomial containing $y$, and compute the corresponding values for $y$. This process will of course produce many extraneous solutions, but we can substitute these into the original system to verify whether they are true solutions.
We wrote the script \texttt{eucl\_eq\_solver} to exactly implement this process. 

The script \texttt{dilatation\_betti\_two\_fibred\_eucl} differs from \texttt{dilatation\_betti\_two\_fibred} by replacing \texttt{solve} by \texttt{eucl\_eq\_solver}. \texttt{dilatation\_betti\_two\_fibred\_eucl} works for the $8$ triangulations that \texttt{dilatation\_betti\_two\_fibred} fails on.

In general, \texttt{eucl\_eq\_solver} is very slow because of its iterative nature. Hence we have chosen to tackle most of the cases using the faster \texttt{dilatation\_betti\_two\_fibred}.

\bibliographystyle{alpha}

\bibliography{bib.bib}

\end{document}